\documentclass[11pt]{article}
\usepackage[top=2cm, left=2cm,right=2cm]{geometry}
\usepackage{graphicx}
\usepackage{epstopdf}
\usepackage{amsfonts}
\usepackage{amsmath}
\usepackage{mathrsfs}
\usepackage{amsthm}
\usepackage{subfigure}
\usepackage{multicol,multicap}
\usepackage{cite}
\usepackage{hyperref}\hypersetup{urlcolor=blue, citecolor=blue}

\newtheorem{theorem}{Theorem}[section]
\newtheorem{lemma}[theorem]{Lemma}

\newtheorem{remark}[theorem]{Remark}

\numberwithin{equation}{section}

\allowdisplaybreaks[1]
\begin{document}
\markboth{Z. Shen, S. Chen $\&$ J. Wei}{Double Hopf bifurcation in nonlocal reaction-diffusion
	systems with spatial average kernel}

\title{Double Hopf bifurcation in nonlocal reaction-diffusion
	systems with spatial average kernel\footnote{This research is supported by National Natural Science Foundation of China (No 11771109).}}

\author{Zuolin Shen\footnote{Email: mathust\_lin@foxmail.com},
Shanshan Chen\footnote{Email: chenss@hit.edu.cn.}, and Junjie Wei\footnote{Corresponding author, Email: weijj@hit.edu.cn.}\\
{\small Department of Mathematics,
Harbin Institute of Technology, \hfill{\ }}\\
{\small Harbin, Heilongjiang, 150001, P.R.China\hfill{\ }}\\
}

\date{}
\maketitle

\begin{abstract}

In this paper, we consider a general reaction-diffusion system with
nonlocal effects and Neumann boundary conditions, where a spatial average kernel is chosen to be the nonlocal kernel.
By virtue of the center manifold
reduction technique and normal form theory, we present a new algorithm for computing normal forms associated with the codimension-two double Hopf bifurcation of nonlocal reaction-diffusion equations. The theoretical results are applied to a predator-prey model, and complex dynamic behaviors such as spatially nonhomogeneous periodic oscillations and spatially nonhomogeneous quasi-periodic oscillations could occur.

\end{abstract}
\section{Introduction}
Reaction-diffusion equations have been proposed to model the complex phenomenon in cell biology, neural network, virus dynamics, biochemical reaction, etc., see \cite{Keller-1970, Shigesada-1979} and references therein.
However, individuals of a species at different locations may compete for common resource or communicate by chemical means \cite{Britton-1989,DuH-2010,Furter-1989}, and nonlocal interactions should be considered.
In 1989, Britton \cite{Britton-1989} proposed a single population model with nonlocal effect, where the nonlocal term takes the following form:
\begin{equation}\label{g*u}
g*u=\int_{\Omega}g(x,y)u(y,t)\mathrm{d}y.
\end{equation}
Here $\Omega$ is the region where the population lives, $u(x,t)$ represents the density of the species at location $x$ and
time $t$. The model is based on the following two assumptions:
\begin{enumerate}
\item [(i)] individuals in grouping together can reduce the risk of predation, which is referred to as the aggregation mechanism;
\item [(ii)] the intraspecific competition at a certain point depends on not only the density
at this point but also a weighted average in the neighborhood of this point.
\end{enumerate}
For unbounded one spatial dimension domain $\Omega=(-\infty,\infty)$, Britton \cite{Britton-1989} also considered the nonlocal effects on two species competition model, and it was shown that the aggregation may lead to the coexistence of the two species.
For bounded domain $\Omega$, a typical scenario of nonlocal dispersal is the \lq\lq spatial average kernel\rq\rq,
that is,
\begin{equation}\label{K*u}
g(x,y)\equiv\dfrac{1}{\mathrm{vol}\,\Omega}.
\end{equation}
Furter and Grinfeld \cite{Furter-1989} obtained that this average kernel can induce spatial nonhomogeneous patterns even for single population model,
see \cite{ShiS-2019} for more general models.

There have been extensive results on the nonlocal effects, including existence and stability of solutions, traveling wave solutions, pattern formation, bifurcation analysis, etc., see \cite{Berestycki-2009,Billingham-2003,Gourley-2000,Gourley-1996,FangZ-2011,LiangS-2003,DuH-2010,Faye_Holzer-2015,Gourley-1993,Merchant-2011,ChenS-2012,ChenY-2016,ChenY-2018,ChenW-2018} and references therein.
For unbounded one spatial dimension domain $\Omega=(-\infty,\infty)$,
Merchant and Nagata \cite{Merchant-2011} chose different types of kernel $g(x,y)$, and showed that the nonlocal competition could induce complex spatiotemporal patterns, see also \cite{Banerjee-2017,Segal-2013}.
Motivated by \cite{Merchant-2011}, Chen {\it et al.} \cite{ChenY-2018,ChenW-2018} considered the case that the spatial domain $\Omega=(0,L)$, and chose the spatial average kernel, i.e., $g(x,y)\equiv1/L$. They found that, for the nonlocal Rosenzweig-MacArthur (RM) and
Holling-Tanner predator-prey models,
the constant positive steady state could lose the stability when
the given parameter passes through some Hopf bifurcation values, and the bifurcating periodic solutions
near such values can be spatially nonhomogeneous.
It is well known that Hopf bifurcation has been used to illustrate the
periodic phenomena in the natural world, such as regular changes in population size, the periodic outbreak of infectious diseases, and chemical oscillations of some autocatalytic reactions \cite{Hassard-1981, Murray-1993, Hale-1977}. However, for reaction-diffusion system without nonlocal effect, the bifurcating periodic solutions near the threshold are always spatially homogeneous, while nonhomogeneous periodic solutions can also
occur through Hopf bifurcation, but they are always unstable, and consequently hard to simulate.
Note that, nonlocal effect could also induced spatially nonhomogeneous steady states through steady state bifurcation\cite{Furter-1989,ShiS-2019}.
Therefore, nonlocal effect could be used as a new mechanism for pattern formation.

We point out that, for reaction-diffusion equations with nonlocal delay effect, the Hopf bifurcation
has also been studied by many researchers, see \cite{ChenS-2012,ChenY-2016,Guo-2015,GuoY-2016,GuoL-2008,SuZ-2014,WangL-2006,Gourley-2002,Gourley-2004} and references therein.
For the homogeneous Neumann boundary condition, Gourley and So \cite{Gourley-2002} showed the existence of Hopf bifurcation for
a food-limited population model, see also \cite{LiangS-2003} for the model with age structure.
For the homogeneous Dirichlet boundary condition,
Chen and Shi \cite{ChenS-2012} developed the method proposed by Busenberg and Huang
\cite{Busenberg-1996}, and studied the existence of Hopf bifurcation near the positive
spatially nonhomogeneous steady state.
Different from \cite{Busenberg-1996,ChenS-2012},
Guo \cite{Guo-2015} proposed the method of Lyapunov-Schmidt reduction to show the existence and even the multiplicity of the spatially nonhomogeneous steady state, and the associated Hopf bifurcation. There are also other results on Hopf bifurcation for reaction-diffusion models with nonlocal delay, see e.g. \cite{Gourley-2004,HuY-2012,SuZ-2014}.

The above mentioned steady state and Hopf bifurcations are both codimension-one bifucation.
A natural question is that whether nonlocal effect could induce codimension-two bifurcation for reaction-diffusion equations.
For the reaction-diffusion equation, the
interaction of a Turing instability (leading to spatially nonhomogeneous steady states) with a
Hopf bifurcation (giving rise to temporal
oscillations)
has been observed and studied in chemical, biological and physical systems, see \cite{Meixner-1997,Maini-1997,DeWit-1996,Rovinsky-1992} and references therein.
For example, Rovinsky and Menzinger \cite{Rovinsky-1992} studied this Turing-Hopf interaction for three models of chemically active media by using Poincar\'{e}-Birkhoff method and shown the bistability of spatially nonhomogeneous steady states and homogeneous oscillations.
In the framework of amplitude equation formalism, De Wit
{\it et al.} \cite{DeWit-1996} investigated the bifurcation scenarios near
the Turing-Hopf singularity. Recently, to show an
accurate dynamic classification at this singularity,
Song {\it et al.}\cite{Song_ZP-2016} applied the normal form theory
proposed by Faria \cite{Faria-2000} to a general reaction-diffusion equation, and obtained a series of explicit formulas for calculating the normal forms associated with the Turing-Hopf bifurcation. This spatiotemporal dynamics induced by the Turing-Hopf bifurcation were observed in several reaction-diffusion models \cite{SongJ-2017,XuW-2018,Baurmann-2007}, see also \cite{AnJ-2018,ShenW-2019} for the reaction-diffusion system with delay.
Another typical codimension-two bifurcation is the double Hopf bifurcation.
As in the Turing-Hopf bifurcation, when the parameters vary near the threshold value, the system may exhibit rich dynamics such as periodic orbit, invariant two torus, invariant three-torus, and even chaos, see e.g. \cite{Guckenheimer-1983,NiuJ-2013, DuN-2019, JiangS-2015, LiJ-2012, XuC-2007}. Recently, for a general delayed reaction-diffusion system with time delay, Du et al. \cite{DuN-2019} also obtained an algorithm for deriving the normal form near a
codimension-two double Hopf bifurcation by virtue of the the normal form method proposed by Faria \cite{Faria-1995,Faria-2000}.

To our knowledge, compared to the classical
reaction-diffusion system, few studies have considered the high codimensional bifurcations in nonlocal reaction-diffusion systems.
Recently, Wu and Song \cite{WuS-2019} studied the dynamical
classification of a nonlocal diffusive Rosenzweig-MacArthur model near the Turing-Hopf singularity.
A numerical simulation \cite{ChenY-2018} revealed that two Hopf bifurcation curves could intersect in a two-parameter plane.
However, there exist no results on double Hopf bifurcation for nonlocal reaction-diffusion systems. In this paper, we aim to consider this problem,
and consider the following general reaction-diffusion system
\begin{equation}\label{model_original}
\begin{cases}
\cfrac{\partial U}{\partial t}=D(\mu)\Delta U
+f(\mu,U,\widehat{U}),&x\in(0,\ell\pi),\;t>0,\\
\cfrac{\partial U}{\partial x}(0,t)=\cfrac{\partial U}{\partial x}(\ell\pi,t)=0,&t>0,\\
\end{cases}
\end{equation}
where $D(\mu)=\text{diag}(d_1(\mu), d_2(\mu),\cdots,d_n(\mu))$ with $d_i(\mu)>0$ and $\mu\in \mathbb{R}^2$, $U=(u_1,u_2,\cdots,u_n)^T\in X$, $\widehat{U}=(\widehat{u}_1, \widehat{u}_2,\cdots,\widehat{u}_n)^T$ with  $ \widehat u_i=\frac{1}{\ell\pi}\int_{0}^{\ell \pi}u_i(y,t)dy $, $i=1,2,\cdots,n$,
$f=(f_1, f_2,\cdots,f_n)^T$ with $f_i$ is $C^k(k\geq 3)$ smooth, and $f_i(\mu,\mathbf{0},\mathbf{0})=0$.
We point out that if $n=1$, then system \eqref{model_original} is reduced to a
general form in \cite{Furter-1989}. If $n=2$, then
\eqref{model_original} becomes a two-component interaction system, which could model the
nonlocal intraspecific and interspecific competition for population models, see \cite{Merchant-2011,Banerjee-2016,Banerjee-2017,Segal-2013}.
The purpose of this paper is to develop an explicit algorithm for computing
normal forms on the center manifold near a codimension-two double-Hopf singularity for model \eqref{model_original}. We should remark that when the ratio of two angular frequencies is some particular value, e.g. 1:2, the corresponding double-Hopf bifurcation may be codimension-three, referred to as the strong resonance case.
In this article, we will not consider this case and focus only the codimension-two double Hopf bifurcation.
We find that, compared with the traditional reaction-diffusion system, \eqref{model_original} is more likely to induce spatial nonhomogeneous patterns, and consequently exhibit rich dynamical behaviors at the corresponding singularity, such as spatially nonhomogeneous periodic solutions, spatially nonhomogeneous quasi-periodic solutions, coexistence of homogeneous/nonhomogeneous oscillations, and so on.

We also adopt the framework of \cite{Faria-2000} to compute the normal forms on the center manifold of system \eqref{model_original} at the codimension-two double Hopf singularity. In summary, we first rewrite system \eqref{model_original} into an abstract form, and by decomposing the phase space into center subspace and its complementary space, we obtain the equivalent system on the center manifold. Then a recursive transformation of variables is used to derive the four-dimensional normal forms. During this process, we construct a Boolean function to deal with the impact of nonlocal terms on the computation, which is the innovation.
Particularly, for the case of $n=2$, we list some additional formulas in Appendix A which could help to obtain all the coefficient vectors that appear in the process of computing normal forms.

The rest of the paper is organized as follows.
The decomposition of phase space and some preliminaries are given in Section 2. The computation of normal forms associated with the codimension-two double Hopf
bifurcation is presented in Section 3.
In Section 4, we apply our theoretical results in Section 3 to a diffusive Holling-Tanner system with spatial average kernel in prey and obtain the normal forms near the
duoble-Hopf singularity.
Some periodic oscillations and quasi-periodic quasi-periodic oscillations are also derived numerically in this section.
Finally, we give some discussion and conclusion for this paper, and in the Appendix, we collect the
details of the coefficient vectors that appear in Section 3
when $n=2$.
Throughout the paper, we denote by $\mathbb{N}$ the set of positive integers, and
$\mathbb{N}_0=\mathbb{N}\cup\{0\}$ the set of non-negative integers.

\section{Decomposition of the phase space}\label{decomp_phase_space}
In this section, we adopt the framework of \cite{Faria-2000} to compute the normal forms of the double Hopf bifurcation.
 To use the center manifold theory for reduction\cite{LinS-1992,Guckenheimer-1983,Hassard-1981}, we need rewrite system \eqref{model_original} into an abstract form and decompose the phase space.

We first define the following real-value Sobolev space
$$ X:=\Big\{(u_1,u_2,\cdots,u_n)^{^T}\in \big[H^2(0, \ell\pi)\big]^n|
~{\partial_x}u_i(0,t)=\partial_x u_i(\ell\pi,t)=0, i=1,2,\cdots,n\Big\},
$$
and then the linear map
$u\rightarrow \frac{1}{\ell\pi}\int_{0}^{\ell \pi}u(y,t)dy$
is smooth from $ H^2(0, \ell\pi)$ to $ H^2(0, \ell\pi)$.
Denote
\begin{equation*}
\mathcal{F}_i:(\mu, U)\rightarrow f_i(\mu,U,\widehat{U}),~ i=1,2,\cdots,n.\\
\end{equation*}
It follows from Appendix C of \cite{Hassard-1981} that $\mathcal{F}_i$ is also smooth from
$\mathbb{R}^2\times \left[H^2(0,\ell\pi)\right]^n$ to $H^2(0,\ell\pi)$. Hence, system \eqref{model_original} can be written as the following abstract form
\begin{equation}\label{abstract_D_F}
\dfrac{d U(t)}{dt}=D(\mu)\Delta U+\mathcal{F}(\mu, U),
\end{equation}
where
$$\mathcal{F}(\mu, U)=\left(\begin{array}{cc}
\mathcal{F}_1(\mu, U)\\
\mathcal{F}_2(\mu,U)\\
~~\vdots\\
\mathcal{F}_n(\mu,U)
\end{array}\right)=\left(\begin{array}{cc}
f_1(\mu,U,\widehat{U})\\
f_2(\mu,U,\widehat{U})\\
~~\vdots\\
f_n(\mu,U,\widehat{U})
\end{array}\right).
$$
Let
	\begin{equation}
	\mathscr{L}(\mu)=D(\mu)\Delta+D_U \mathcal{F}(\mu,0),
	\end{equation}
	where $D_U \mathcal{F}(\mu, 0)$ stands for the Fr\'{e}chet derivative of
	$\mathcal{F}(\mu, U)$ with respect to $U$ at $U=\mathbf{0}$.
To figure out the double Hopf bifurcation with two pairs of purely
imaginary eigenvalues, we define the following complexification
space of $X$:
 $$X_{\mathbb{C}}:=X\oplus iX=\{x_1+i x_2|~x_1,x_2 \in X\},$$
with the complex-valued $L^2$ inner product $\langle\cdot,\cdot\rangle$, defined by
$$
\langle U,V\rangle=\int_0^{\ell\pi}(\bar{u}_1v_1+\bar{u}_2v_2+\cdots+\bar{u}_nv_n)dx,
$$
where $U=(u_1, u_2,\cdots,u_n)^T \in X_{\mathbb{C}}$, $V=(v_1, v_2,\cdots,v_n)^T \in X_{\mathbb{C}}$.

Considering the perturbation caused by
 the nonlocal terms, we rewrite system \eqref{abstract_D_F} in a intuitive form
\begin{equation}\label{abstract_D_L_L_F}
   \dfrac{d U(t)}{dt}=D(\mu)\Delta U+L(\mu)U+\widehat{L}(
   \mu)\widehat{U}+F(U,\widehat{U},\mu),
\end{equation}
where $L$ and $\widehat{L}$ are bounded linear operators from
$\mathbb{R}^2\times X_{\mathbb{C}}$ to $X_{\mathbb{C}}$, and
$F: X_{\mathbb{C}}\times X_{\mathbb{C}}\times\mathbb{R}^2
\rightarrow X_{\mathbb{C}}$ is a $C^k(k\geq 3)$ function
such that $F(0,0,\mu)=0$ and $D_UF(0, 0,\mu)=
D_{\widehat{U}}F(0,0,\mu)=0$.

Then the linearization of system \eqref{abstract_D_L_L_F} at $\mathbf{0}$ takes the following form
\begin{equation}\label{linearization}
\dfrac{d U(t)}{dt}=D(\mu)\Delta U+L(\mu)U+\widehat{L}(
\mu)\widehat{U}.
\end{equation}
It is well known that the eigenvalue problem
\begin{equation*}
-\Delta \xi=\sigma\xi, ~x\in(0,\ell\pi),~\xi'(0)=\xi'(\ell\pi)=0
\end{equation*}
has eigenvalues $\sigma_n=\frac{n^2}{\ell^2}$ ($n\in \mathbb{N}_0$), and the corresponding normalized eigenfunctions
\begin{equation}\label{xi}
\xi_{n}(x)=\cfrac{\cos{\frac{n}{\ell}x}}
{\parallel\cos{\frac{n}{\ell}x}\parallel}=\begin{cases}
\sqrt{\frac{1}{\ell\pi}}, &n=0,\\
\sqrt{\frac{2}{\ell\pi}}\cos{\frac{n}{\ell}x},&n\in\mathbb{N}.
\end{cases}
\end{equation}
Letting $\beta_n^i(x)=\xi_n(x)e_i$, $i=1,2,\cdots,n$, where $e_i$ is the $i$th
unit coordinate vector of $\mathbb{R}^n$, we see that $\{\beta_n^i\}_{n\in\mathbb{N}_0}$ are eigenfunctions of $-D(\mu)\Delta$ with corresponding eigenvalues $d_i(\mu)\frac{n^2}{\ell^2}$, and $\{\beta_n^i\}_{n\in\mathbb{N}_0}$ form an orthonormal basis
of $X_{\mathbb{C}}$.

For $U\in X_{\mathbb{C}}$ and $\beta_k=(\beta_k^1, \beta_k^2,\cdots,\beta_k^n)$, we define
$\langle\beta_k,U\rangle=\left(
\langle \beta_k^1,U\rangle,
\langle \beta_k^2,U\rangle,\cdots,\langle \beta_k^n,U\rangle
\right)^{^T}$,
and denote 
$$\mathcal{B}_k:=\text{span}\left\{~\langle \beta_k^j,U\rangle\beta_k^j
~|~U\in X_{\mathbb{C}}, j=1,2,\cdots,n~\right\}.$$
Let
$$\varphi = \sum^{\infty}_{k=0}\xi_k(x)\left(\begin{array}{c}
a_k^1\\a_k^2\\
~\vdots\\
a_k^n
\end{array}\right)$$
be the eigenfunction with respect to eigenvalue $\lambda(\mu)$. Then
\begin{equation}\label{CE_varphi}
	\lambda(\mu)\varphi-D(\mu)\Delta\varphi-L(\mu)\varphi
	-\widehat{L}(\mu)\widehat{\varphi}=0,
\end{equation}
where $\widehat{\varphi}=\frac{1}{\ell\pi}
\int_{0}^{\ell\pi}\varphi dx$. Note that
\begin{equation}\label{hat_xi}
	\dfrac{1}{\ell\pi}
	\int_{0}^{\ell\pi}\xi_n(x)dx=\begin{cases}
	\dfrac{1}{\sqrt{\ell\pi}},~&n=0,\\
	~~0,~&n\in\mathbb{N}.
	\end{cases}
\end{equation}
Then \eqref{CE_varphi} is equivalent to a sequence of
characteristic equations:
\begin{equation}\label{CE_n}
	\begin{cases}
	\text{det}\Big(\lambda(\mu)I-L(\mu)-\widehat{L}(\mu)
	\Big)=0, &n=0,\\
	\text{det}\Big(\lambda(\mu)I+\dfrac{n^2}{\ell^2}D(\mu)-L(\mu)
	\Big)=0, &n\in\mathbb{N}.\\

	\end{cases}
\end{equation}

To consider the double Hopf bifurcation, we assume that there exists
$\mu_0\in \mathbb{R}^2$ such that the following conditions hold:

\begin{itemize}\label{assumption}
\item [$\mathbf{(H_1)}$] There exist a neighborhood $\mathscr{N}$ of $\mu_0$ and $n_1, n_2\in \mathbb{N}_0$ such that, for $\mu\in\mathscr N$, the linear system \eqref{linearization} has two pairs of complex simple eigenvalues $\alpha_{1}(\mu)\pm\omega_1(\mu)$ and
$\alpha_{2}(\mu)\pm\omega_2(\mu)$, which are both continuously differentiable
in $\mu$ with $\alpha_{1}(\mu_0)=0$, $\omega_1(\mu_0)=\omega_1>0$, $\alpha_{2}(\mu_0)=0$, $\omega_2(\mu_0)=\omega_2>0$, and all other eigenvalues of \eqref{linearization} have non-zero real parts for $\mu\in\mathscr N$.
\item [$\mathbf{(H_2)}$] Assume that $\omega_1:\omega_2\neq i:j$ for $i,j\in\mathbb{N}$
and $1\leq i\leq j\leq 4$, i.e.,
we only consider the codimension-two double Hopf bifurcation of non-resonance and weak resonance
instead of the codimension-three of strongly resonant case.
\item [$\mathbf{(H_3)}$] The conjugate eigenvalues $\alpha_k(\mu)\pm\omega_{k}(\mu)$ are obtained by \eqref{CE_n}$_{n_k}$, and the corresponding
eigenvalues belong to $\mathcal{B}_{n_k}$ for $k=1, 2$. Without lose of
generality, we assume $n_1\leq n_2$.
\end{itemize}

Let $\mu=\mu_0+\alpha$, where $\alpha=(\alpha_1, \alpha_2)\in\mathbb{R}^2$, and then \eqref{abstract_D_L_L_F} can be transformed as
\begin{equation}\label{abstract_D_L_L_F_0}
	\dfrac{d U(t)}{dt}=D_0\Delta U+L_0U+\widehat{L}_0\widehat{U}
	+\widetilde{F}(U,\widehat{U},\alpha),
\end{equation}
where $D_0=D(\mu_0)$, $L_0=L(\mu_0)$,
$\widehat{L}_0=\widehat{L}(\mu_0)$, and
$\widetilde{F}(U,\widehat{U}, \alpha)=[D(\alpha+\mu_0)-D_0]
\Delta U+[L(\alpha+\mu_0)-L_0] U
+[\widehat{L}(\alpha+\mu_0)-\widehat{L}_0]\widehat{U}
+F(U,\widehat{U},\alpha+\mu_0)$.
Then the linear system of \eqref{abstract_D_L_L_F_0} on $\mathcal{B}_{n_k}$ is equivalent to the following ODEs on $\mathbb{C}^n$:
\begin{equation}\label{ode_A_n}
	\dot{z}(t)=A_{n_k}z(t),
\end{equation}
where $A_{n_k}$ is an $n\times n$ matrix, and
\begin{equation*}
	A_{n_k} y(t)=\begin{cases}
	~L_0y(t)+\widehat{L}_0 y(t), &n_k=0,\\
	-\dfrac{n_k^2}{\ell^2}D_0y(t)+L_0y(t),&n_k\in\mathbb{N}.
	\end{cases}
\end{equation*}
Denote by $A_{n_k}^*$ the formal adjoint of $A_{n_k}$ under the scalar
product on $\mathbb{C}^n$:
$$(\psi,\phi)_{_{\mathbb{C}^n}}=\overline{\psi}\,\phi,
~\text{for}~\psi^{^T}, \phi\in \mathbb{C}^n.$$

Let $\Lambda=\{\pm i\omega_1, \pm i\omega_2\}$ and let
\begin{equation*}
	\begin{array}{ll}
	\Phi_1=(\phi_1, \phi_2),~\Phi_2=(\phi_3, \phi_4),~
	\Psi_1=\left(\begin{array}{cc}
	\psi_1\\ \psi_2
	\end{array}\right),~\Psi_2=\left(\begin{array}{cc}
	\psi_3\\ \psi_4
	\end{array}\right),
	\end{array}
\end{equation*}
be the basis of the generalized eigenspace of $A_{n_k}$ and $A_{n_k}^*$
corresponding to the eigenvalues $\Lambda$, respectively. Then
\begin{equation}\label{A_k_B_k}
	A_{n_k}\Phi_k=\Phi_k B_k,~
	 A_{n_k}^*\Psi_k=\bar{B}_k\Psi_k,~(\Psi_k, \Phi_k)_{_{\mathbb{C}^n}}=I_2,~k=1,2,
\end{equation}
where $B_1=\text{diag}(i\omega_1, -i\omega_1)$,
$B_2=\text{diag}(i\omega_2, -i\omega_2)$, and
$I_2$ is an $2\times2$ identity matrix. Then we can decompose
	the phase space $X_{\mathbb{C}}$:
\begin{equation}\label{P_Ker}
	X_{\mathbb{C}}=P\oplus \text{Ker}\pi,
\end{equation}
where $P=\text{Im} \pi$, and $\pi:X_{\mathbb{C}}\rightarrow P$ is the
projection, defined by
\begin{equation*}
	\pi(U)=\sum_{k=1}^2\Phi_k\left(\Psi_k,
	\langle \beta_{n_k},U\rangle\right)_{_{\mathbb{C}^n}}\xi_{n_k}.
\end{equation*}
Therefore, $U\in X_{\mathbb{C}}$ can be rewritten in the following form:
\begin{equation}\label{U_decompose}
\begin{array}{ll}
    	U&=\sum_{k=1}^2(\Phi_k \tilde{z}_k(t))\xi_{n_k}+w\\
    	&=\Phi z^x+w,
\end{array}
\end{equation}
where $\tilde{z}_k(t)=\left(\Psi_k,
\langle \beta_{n_k},U\rangle\right)_{_{\mathbb{C}^n}}\in \mathbb{C}^2$, $\Phi=(\Phi_1, \Phi_2)$, $z^x=(z_1\xi_{n_1}, z_2\xi_{n_1},
z_3\xi_{n_2},z_4\xi_{n_2})^T$,  and $w\in \text{Ker}\pi$.
For simplification of notations, we denote $z(t)=\text{col}(\tilde{z}_1(t), \tilde{z}_2(t))=(z_1(t), z_2(t), z_3(t), z_4(t))^{^T}\in \mathbb{C}^4$
and
\begin{equation}\label{F_z_w_w_lapha}
\widetilde{F}(z,w,\widehat{w},\alpha)
=\widetilde{F}\left(\sum_{k=1}^2(\Phi_k \tilde{z}_k(t))\xi_{n_k}+w,\frac{1}{\ell\pi}
\int_{0}^{\ell\pi}\left(\sum_{k=1}^2(\Phi_k \tilde{z}_k(t))\xi_{n_k}+w\right)dx,\alpha\right).
\end{equation}
In the following, we will also use the symbol $(z,w,\widehat{w},\alpha)$ instead of
$(U,\widehat{U},\alpha)$.
Now system \eqref{abstract_D_L_L_F_0} is equivalent to the following abstract ODEs in $\mathbb{C}^4\times \mathrm{Ker} \pi$
\begin{equation}\label{z_w_beta}
	\begin{cases}
	~\dot{z}~=~Bz+\bar{\Psi}\left(\begin{array}{cc}
	\big\langle\beta_{n_1},\widetilde{F}(z,w,\widehat{w},\alpha)
	\big\rangle\\
	\big\langle\beta_{n_2},\widetilde{F}(z,w,\widehat{w},\alpha)
	\big\rangle
	\end{array}\right),\\
	\dfrac{d}{dt}w=\mathscr{L}_1w+(I-\pi)
	\widetilde{F}(z,w,\widehat{w},\alpha),	
	\end{cases}
	\end{equation}
where $B=\text{diag}(B_1, B_2)$, $\Psi=\text{diag}(\Psi_1, \Psi_2)$, and $\mathscr L_1$ is
the restriction of $\mathscr{L}(\mu_0)$ on $\text{Ker}\pi$.

\section{Center manifold reduction and normal forms for double Hopf bifurcation}
\label{section_Normal_Form}
\subsection{Center manifold reduction}
Consider the formal Taylor expansions
$$\widetilde{F}(U,\widehat{U}, \alpha)=\sum_{j\geq 2}
\dfrac{1}{j!}\widetilde{F}_j(U,\widehat{U},\alpha),~\alpha\in
\mathbb{R}^2, U\in X_{\mathbb{C}},$$
where $\widetilde{F}_j$ is the $j$th Fr\'{e}chet derivation of
$\widetilde{F}$. Then system \eqref{z_w_beta} can be rewritten
as
\begin{equation}\label{z_w_f}
	\begin{cases}
		~\dot{z}~=~Bz+\displaystyle\sum_{j\geq 2}\dfrac{1}{j!}f_j^1(z,w,\widehat{w},\alpha),\\
		\dfrac{d}{dt}w=\mathscr{L}_1w+\displaystyle\sum_{j\geq 2}\dfrac{1}{j!}f_j^2(z,w,\widehat{w},\alpha),
	\end{cases}
\end{equation}
where $\widehat{w}=\frac{1}{\ell\pi}\int_{0}^{\ell\pi}w \mathrm{d}x
\in\mathrm{Ker}\pi$, and $f_j=(f_j^1, f_j^2)$ is defined by
\begin{equation}\label{f_j}
	\begin{array}{ll}
	f_j^1(z,w,\widehat{w},\alpha)
	=\bar{\Psi}\left(\begin{array}{cc}
	\big\langle \beta_{n_1},\widetilde{F}_j(z,w,\widehat{w},\alpha)
	\big\rangle\\
	\big\langle \beta_{n_2},\widetilde{F}_j(z,w,\widehat{w},\alpha)
	\big\rangle
	\end{array}\right),
	\vspace{0.3cm}\\
    f_j^2(z,w,\widehat{w},\alpha)
    =(I-\pi)\widetilde{F}_j(z,w,\widehat{w},\alpha),
	\end{array}
\end{equation}
with $\widetilde{F}_j(z,w,\widehat{w},\alpha)=
\widetilde{F}_j(U,\widehat{U}, \alpha)$.

It follows from \cite{Faria-2000} (see also \cite{Chow-1982}) that the normal forms of
\eqref{z_w_f} can be obtained by a recursive transformation of
variables
$$(z,w,\alpha)=(\widetilde{z}, \widetilde{w}, \alpha)
+\dfrac{1}{j!}(U_j^1(\widetilde{z}, \alpha), U_j^2(\widetilde{z}, \alpha)), j
\geq 2,$$
with $U_j=(U_j^1, U_j^2)\in V_j^{4+2}(\mathbb{C}^4)\times
V_j^{4+2}(\text{Ker}\pi)$. Here, for a normed space $Y$, we
denote $V_j^{4+2}(Y)$ be the space of homogeneous polynomials
of degree $j$ in $4+2$ variables $z=(z_1, z_2, z_3, z_4)$
, $\alpha=(\alpha_1, \alpha_2)$ with coefficients in $Y$, that is,
\begin{equation}\label{v_4+2}
	V_j^{4+2}=\left\{\sum_{|(p,l)|=j}c_{(p,l)}z^p\alpha^l: (p,l)\in \mathbb{N}^{4+2}_0, c_{(p,l)}\in Y\right\},
\end{equation}
where $p=(p_1,p_2,p_3,p_4)\in\mathbb{N}_0^4$, $l=(l_1,l_2)\in\mathbb{N}_0^2$, $\sum_{i=1}^4 p_i+\sum_{i=1}^2 l_i=j$, $z^p=z_1^{p_1}z_2^{p_2}z_3^{p_3}z_4^{p_4}$, $\alpha^{l}=\alpha_1^{l_1}\alpha_2^{l_2}$, and the norm is defined as the sum
 of the norms of the coefficients $|\sum_{|(q,l)|=j}c_{(q,l)}z^q\alpha^l|
 =\sum_{|(q,l)|=j}|c_{(q,l)}|_{Y}$.

We denote by $\bar{f}_j=(\bar{f}_j^1, \bar{f}_j^2)$ the terms of order $j$ obtained after the computation of normal forms in the preceding steps, and define the operators
$M_j=(M_j^1, M_j^2), j\geq 2$ by
\begin{equation}\label{M_j_operator}
\begin{array}{ll}
M_j^1:~V_j^{4+2}(\mathbb{C}^4)~\rightarrow~
V_j^{4+2}(\mathbb{C}^4),\\
(M_j^1 p)(z,\alpha)=D_zp(z, \alpha)Bz-Bp(z,\alpha),\\
M_j^2:~V_j^{4+2}(\text{Ker}\pi)~\rightarrow~
V_j^{4+2}(\text{Ker}\pi),\\
(M_j^2 h)(z,\alpha)=D_zh(z, \alpha)Bz-\mathscr{L}_1p(z,\alpha).\\
\end{array}
\end{equation}
With the recursive procedure and dropping the tilde for simplicity of notations, \eqref{z_w_f} becomes
\begin{equation}\label{z_w_g}
\begin{cases}
~\dot{z}~=~Bz+\displaystyle\sum_{j\geq 2}\dfrac{1}{j!}g_j^1(z,w,\widehat{w},\alpha),\\
\dfrac{d}{dt}w=\mathscr{L}_1w+\displaystyle\sum_{j\geq 2}\dfrac{1}{j!}g_j^2(z,w,\widehat{w},\alpha),
\end{cases}
\end{equation}
where $g_j=(g_j^1,g_j^2), j\geq 2$, are the new terms of
order $j$ and given by
$$g_j(z,w,\widehat{w},\alpha)=\bar{f}_j(z,w,\widehat{w},\alpha)
-M_jU_j(z,\alpha).$$
Here, $U_j\in V_j^{4+2}(\mathbb{C}^4)\times
V_j^{4+2}(\text{Ker}\pi)$ can be computed by
\begin{equation}\label{U_j}
U_j(z,\alpha)=(M_j)^{-1}\mathbf{P}_{\mathrm{Im},j}
\bar{f}_j(z,0,0,\alpha),
\end{equation}
where $M_j^{-1}$ is the inverse of $M_j$ with range defined
on $\text{Ker}(M_j^1)^c\times\text{Ker}(M_j^2)^c$,
$\mathbf{P}_{\mathrm{Im},j}=(\mathbf{P}^1_{\mathrm{Im},j},\mathbf{P}^2_{\mathrm{Im},j})$ is the projection operator associated
with the preceding decomposition of $V_j^{4+2}(\mathbb{C}^4)\times
V_j^{4+2}(\text{Ker}\pi)$ over $\mathrm{Im}(M_j^1)\times \mathrm{Im}(M_j^2)$.

\subsection{Normal forms up to second order}\label{normal_forma_2}


By \eqref{M_j_operator} and assumption $\mathbf{(H_2)}$, it is
easy to verify that
\begin{equation}\label{M_j^1(z,lapha)}
	\begin{array}{ll}
	M_j^1(z^p \alpha^l e_k)=D_z(z^p \alpha^l e_k)Bz-Bz^p \alpha^l e_k\\
	=\begin{cases}
	\left(i\omega_1(p_1-p_2)+i\omega_2(p_3-p_4)
	+(-1)^ki\omega_1\right)z^p \alpha^l e_k, k=1,2,\\
	\left(i\omega_1(p_1-p_2)+i\omega_2(p_3-p_4)
	+(-1)^ki\omega_2\right)z^p \alpha^l e_k, k=3,4.
	\end{cases}
	\end{array}
\end{equation}
Here $e_k$ is the $k$th unit coordinate vector of $\mathbb{R}^4$, and $z^p, \alpha^{l}$ are defined as in \eqref{v_4+2}.
 Therefore,
\begin{equation}\label{KerM_2^1}
	\mathrm{Ker}(M_2^1)=\text{span}\{\alpha_iz_1e_1, \alpha_iz_2e_2,
	\alpha_iz_3e_3, \alpha_iz_4e_4\},~ i=1,2.
\end{equation}
Hence, the normal forms up to second order of \eqref{abstract_D_F} on the center manifold of
the origin near $\mu=\mu_0$ has the form
\begin{equation}\label{norm_z_2_order}
	\dot{z}=Bz+\dfrac{1}{2!}g_2^1(z,0,0,\alpha)+h.o.t.,
\end{equation}
with $g_2^1(z,0,0,\alpha)=\mathrm{Proj}_{\mathrm{Ker}(M_2^1)}f_2^1(z,0,0,\alpha)$.

To show the specific expressions of $g_2^1(z,0,0,\alpha)$, we consider the Taylor expansions
of $D(\mu)$, $L(\mu)$ and $\widehat{L}(\mu)$:
\begin{equation*}
	\begin{array}{ll}
	D(\mu)=&D_0+\alpha_1D_1^{(1,0)}+\alpha_2 D_1^{(0,1)}
	+\dfrac{1}{2!}\left(\alpha_1^2D_2^{(2,0)}+2\alpha_1\alpha_2
	D_2^{(1,1)}+\alpha_2^2D_2^{(0,2)}\right)+\cdots,
	\vspace{0.2cm}\\
	L(\mu)=&L_0+\alpha_1L_1^{(1,0)}+\alpha_2 L_1^{(0,1)}
	+\dfrac{1}{2!}\left(\alpha_1^2L_2^{(2,0)}+2\alpha_1\alpha_2
	L_2^{(1,1)}+\alpha_2^2L_2^{(0,2)}\right)+\cdots,
	\vspace{0.2cm}\\
	\widehat{L}(\mu)=&\widehat{L}_0+\alpha_1\widehat{L}_1^{(1,0)}+\alpha_2 \widehat{L}_1^{(0,1)}
	+\dfrac{1}{2!}\left(\alpha_1^2\widehat{L}_2^{(2,0)}+2\alpha_1\alpha_2
	\widehat{L}_2^{(1,1)}+\alpha_2^2\widehat{L}_2^{(0,2)}\right)+\cdots.\\
	\end{array}
\end{equation*}
Therefore, the second order term of $\widetilde{F}$ is
\begin{equation}\label{tilde_F_2_U}
    \begin{array}{ll}
    \widetilde{F}_2(U,\widehat{U},\alpha)&=
    2\left(\alpha_1D_1^{(1,0)}+\alpha_2 D_1^{(0,1)} \right)\Delta U
    +2\left(\alpha_1L_1^{(1,0)}+\alpha_2 L_1^{(0,1)}\right)U\\
    &~~+2\left(\alpha_1\widehat{L}_1^{(1,0)}+\alpha_2 \widehat{L}_1^{(0,1)}\right)
    \widehat{U}+F_2(U,\widehat{U},\alpha).
    \end{array}
\end{equation}
Recalling that $F(0,0,\mu)=0$, $D_UF(0,0,\mu)=0$ and
$D_{\widehat{U}}F(0,0,\mu)=0$, we have $F_2(U,\widehat{U},\alpha)=F_2(U,\widehat{U},0)$. Plug \eqref{U_decompose} into \eqref{tilde_F_2_U} at $w=0$, and then $\widetilde{F}_2(U,\widehat{U},\alpha)$ becomes
\begin{equation*}
\begin{array}{ll}
\widetilde{F}_2(z,0,0,\alpha)&=\widetilde{F}_2(\Phi z^x,\Phi \widehat{z}^x,
\alpha)\\
&=2\left(\alpha_1D_1^{(1,0)}+\alpha_2 D_1^{(0,1)} \right)\Delta (\Phi z^x)
+2\left(\alpha_1L_1^{(1,0)}+\alpha_2 L_1^{(0,1)}\right)\Phi z^x\\
&~~+2\left(\alpha_1\widehat{L}_1^{(1,0)}+\alpha_2 \widehat{L}_1^{(0,1)}\right)
\Phi \widehat{z}^x+F_2(\Phi z^x,\Phi \widehat{z}^x,0),
\end{array}
\end{equation*}
where $\widehat{z}^x=\frac{1}{\ell\pi}\int_{0}^{\ell\pi}z^x\mathrm{d}x$ with $z^x$ is defined as in \eqref{U_decompose}.
By \eqref{f_j}, we have
\begin{equation}\label{f_2^1_z_0_alpha}
\frac{1}{2!}f_2^1(z,0,0,\alpha)=\frac{1}{2!}\bar{\Psi}\left(\begin{array}{l}
\big\langle\beta_{n_1},\widetilde{F}_2(z,0,0,\alpha)\big\rangle\\
\big\langle\beta_{n_2},\widetilde{F}_2(z,0,0,\alpha)\big\rangle
\end{array}\right).
\end{equation}
To write \eqref{f_2^1_z_0_alpha} explicitly, we
define the following Boolean function
\begin{equation}\label{Boolean}
	\delta(k)=\langle\widehat{\xi}_{k},\xi_{k}\rangle=\begin{cases}
	1,&k=0,\\
	0,&k\neq 0.
	\end{cases}
\end{equation}

It follows from \eqref{M_j^1(z,lapha)} and the fact
\begin{equation*}
	\int_{0}^{\ell\pi}\xi_{n_1}^2dx=\int_{0}^{\ell\pi}\xi_{n_2}^2dx=1
\end{equation*}
that
\begin{equation}\label{g_2^1}
	\dfrac{1}{2!}g_2^1(z,0,0,\alpha)=\dfrac{1}{2!}\mathrm{Proj}
	_{\mathrm{Ker}(M_2^1)}f_2^1(z,0,0,\alpha)=\left(
	\begin{array}{cc}
	B_{11}\alpha_1z_1+B_{21}\alpha_2z_1\\
	\overline{B_{11}}\alpha_1z_2+\overline{B_{21}}\alpha_2z_2\\
	B_{13}\alpha_1z_3+B_{23}\alpha_2z_3\\
	\overline{B_{13}}\alpha_1z_4+\overline{B_{23}}\alpha_2z_4\\
	\end{array}\right),
\end{equation}
where
\begin{equation}\label{normal_forms_B2}
    \begin{array}{ll}
    B_{11}=\bar{\psi}_1\left(-\dfrac{n_1^2}{\ell^2}D_1^{(1,0)}\phi_1
    +L_1^{(1,0)}\phi_1+\widehat{L}_1^{(1,0)}\phi_1\delta(n_1)\right),\\
    B_{21}=\bar{\psi}_1\left(-\dfrac{n_1^2}{\ell^2}D_1^{(0,1)}\phi_1
    +L_1^{(0,1)}\phi_1+\widehat{L}_1^{(0,1)}\phi_1\delta(n_1)\right),\\
    B_{13}=\bar{\psi}_3\left(-\dfrac{n_2^2}{\ell^2}D_1^{(1,0)}\phi_3
    +L_1^{(1,0)}\phi_3+\widehat{L}_1^{(1,0)}\phi_3\delta(n_2)\right),\\
    B_{23}=\bar{\psi}_3\left(-\dfrac{n_2^2}{\ell^2}D_1^{(0,1)}\phi_3
    +L_1^{(0,1)}\phi_1+\widehat{L}_1^{(0,1)}\phi_1\delta(n_2)\right).\\
    \end{array}
\end{equation}

\subsection{Normal forms up to third order}
From \eqref{M_j^1(z,lapha)}, we have
\begin{equation}\label{Ker_M_31}
	\mathrm{Ker}(M_3^1)=\mathrm{span}
	\left\{z_1^2z_2e_1, z_1z_3z_4e_1, z_1z_2^2e_2,
	z_2z_3z_4e_2, z_3^2z_4e_3, z_1z_2z_3e_3,
	z_3z_4^2e_4, z_1z_2z_4e_4\right\}.
\end{equation}
According to \eqref{z_w_g}, the normal forms up to third order has the form
$$\dot{z}=Bz+\dfrac{1}{2!}g_2^1(z,0,0,\alpha)
+\dfrac{1}{3!}g_3^1(z,0,0,0)+\cdots,$$
where $g_3^1(z,0,0,0)=\mathrm{Proj}_{\mathrm{Ker}(M_3^1)}
\bar{f}^1_3(z,0,0,0)$. The new third order $\bar{f}^1_3(z,0,0,0)$ can be calculated by
\begin{equation}\label{bar_f_3^1}
    \begin{array}{lr}
    \bar{f}^1_3(z,0,0,0)=f_3^1(z,0,0,0)&+\dfrac{3}{2}
    \Big(D_zf_2^1(z,0,0,0)U_2^1(z,0)
    +D_wf_2^1(z,0,0,0)U_2^2(z,0)\\
    &+D_{\widehat{w}}f_2^1(z,0,0,0)\widehat{U}
    _2^2(z,0)-D_zU_2^1(z,0)g_2^1(z,0,0,0)\Big),
    \end{array}	
\end{equation}
where $U_2^1, U_2^2$ are given as in \eqref{U_j}, and $\widehat U_2^2=\frac{1}{\ell\pi}\int_0^{\ell\pi}U_2^2 ~\text{d}x$.
It follows from \eqref{g_2^1} that $g_2^1(z,0,0,0)=0$, and we still have to calculate the following four parts:
\begin{equation*}
\begin{array}{ll}
\text{Proj}_{\mathrm{Ker}(M_3^1)}f_3^1(z,0,0,0),
~~\text{Proj}_{\mathrm{Ker}(M_3^1)}(D_zf_2^1(z,0,0,0)U_2^1(z,0)),~~\\
\text{Proj}_{\mathrm{Ker}(M_3^1)}( D_wf_2^1(z,0,0,0)U_2^2(z,0)),~~
\text{Proj}_{\mathrm{Ker}(M_3^1)}( D_{\widehat w}f_2^1(z,0,0,0)\widehat U_2^2(z,0)),
\end{array}
\end{equation*}
which will be shown in the following.

\paragraph{(a) The computation of $\mathrm{Proj}_{\mathrm{Ker} (M_3^1)}f_3^1(z,0,0,0)$}.

The third order Fr\'{e}chet derivative of $\widetilde{F}(U,\widehat{U},\alpha)$ at
$(\Phi z^x, \Phi\widehat{z}^x,0)$ is
\begin{equation*}
\begin{array}{ll}
\widetilde{F}_3(z,0,0,0)=\sum_{|\iota|=3}
F_{\iota_1\iota_2\iota_3\iota_4}\xi_{n_1}^{\iota_1+\iota_2}(x)
\xi_{n_2}^{\iota_3+\iota_4}(x)z_1^{\iota_1}{z}_2^{\iota_2}
z_3^{\iota_3}{z}_4^{\iota_4},
\end{array}
\end{equation*}
where $\iota=(\iota_1,\iota_2, \iota_3,\iota_4)\in\mathbb{N}_0^4$, $|\iota|=\sum_{j=1}^4\iota_j$, $F_{\iota_1\iota_2\iota_3\iota_4}$ is the coefficient vector of
$z_1^{\iota_1}{z}_2^{\iota_2}
z_3^{\iota_3}{z}_4^{\iota_4}$.
Then we have
\begin{equation*}
    \begin{array}{ll}
    f_3^1(z,0,0,0)&=\bar{\Phi}\left(\begin{array}{l}
    \big\langle\beta_{n_1},\widetilde{F}_3(z,0,0,0)\big\rangle\\
    \big\langle\beta_{n_2},\widetilde{F}_3(z,0,0,0)\big\rangle
    \end{array}\right)\\
    &=\bar{\Phi}\left(\begin{array}{l}
    \sum_{|\iota|=3}
    F_{\iota_1\iota_2\iota_3\iota_4}
    \int_{0}^{\ell\pi}\xi_{n_1}^{\iota_1+\iota_2+1}
    \xi_{n_2}^{\iota_3+\iota_4}\mathrm{d}x
    z_1^{\iota_1}{z}_2^{\iota_2}z_3^{\iota_3}{z}_4^{\iota_4}\\
   \sum_{|\iota|=3}
    F_{\iota_1\iota_2\iota_3\iota_4}
    \int_{0}^{\ell\pi}\xi_{n_1}^{\iota_1+\iota_2}
    \xi_{n_2}^{\iota_3+\iota_4+1}\mathrm{d}xz_1^{\iota_1}{z}_2^{\iota_2}
    z_3^{\iota_3}{z}_4^{\iota_4}
    \end{array}\right).
    \end{array}
\end{equation*}
Thus,
\begin{equation}\label{normal_forms_C}
\frac{1}{3!}\text{Proj}_{\mathrm{Ker}(M_3^1)}f_3^1(z,0,0,0)
=\left(\begin{array}{ll}
C_{2100}z_1^2z_2+C_{1011}z_1z_3z_4\\
\overline{C_{2100}}z_1z_2^2+\overline{C_{1011}}z_2z_3z_4\\
C_{0021}z_3^2z_4+C_{1110}z_1z_2z_3\\
\overline{C_{0021}}z_3z_4^2+\overline{C_{1110}}z_1z_2z4\\
\end{array}\right),
\end{equation}
where
\begin{equation*}
\begin{array}{ll}
C_{2100}=\frac{1}{6}\bar{\psi}_1F_{2100}\gamma_{40},~
C_{1011}=\frac{1}{6}\bar{\psi}_1F_{1011}\gamma_{22},~
C_{0021}=\frac{1}{6}\bar{\psi}_3F_{0021}\gamma_{04},~
C_{1110}=\frac{1}{6}\bar{\psi}_3F_{1110}\gamma_{22},~\\
\end{array}
\end{equation*}
with $\gamma_{ij}=\int_0^{\ell\pi}\xi_{n_1}^i(x)\xi_{n_2}^j(x) \text{d}x$, and
\begin{equation*}
	\begin{array}{ll}
	\displaystyle~\qquad\int_{0}^{\ell\pi}\xi_{n_k}^4(x)\mathrm{d}x=
	\begin{cases}
    \cfrac{1}{\ell\pi},&n_k=0,
	\vspace{0.1cm}\\
	\cfrac{3}{2\ell\pi},&n_k\neq 0,\\
	\end{cases}~~~~
	\displaystyle\int_{0}^{\ell\pi}\xi_{n_1}^2(x)\xi_{n_2}^2(x)\mathrm{d}x=
	\displaystyle\begin{cases}
	\cfrac{3}{2\ell\pi},&n_1=n_2\neq 0,
	\vspace{0.1cm}\\
	\cfrac{1}{\ell\pi},& otherwise.\\
	\end{cases}
	\end{array}
\end{equation*}

\paragraph{(b) The computation of $\mathrm{Proj}_{\mathrm{Ker}(M_3^1)}D_zf_2^1(z,0,0,0)U_2^1(z,0)$}.

From section \ref{normal_forma_2}, we know that $F_2(U,\widehat{U},\alpha)=F_2(U,\widehat{U},0)$. Moreover, by \eqref{U_decompose}, we have
\begin{equation}\label{hat_F_2_z_w}
    \begin{array}{ll}
    F_2(z,w,\widehat{w},0)&=F_2(U,\widehat{U},0)\\
    &=\displaystyle\sum_{|\iota|=2}F_{\iota_1\iota_2\iota_3\iota_4}
    \xi_{n_1}^{\iota_1+\iota_2}
    \xi_{n_2}^{\iota_3+\iota_4}z_1^{\iota_1}{z}_2^{\iota_2}
    z_3^{\iota_3}{z}_4^{\iota_4}+S_2(w)+S_2(\widehat{w})
    +o(|w|^2,|w\widehat{w}|,\widehat{w}^2),
    \end{array}
\end{equation}
where $S_2(w), S_2(\widehat{w})$ represent the linear terms of $w$ and $\widehat{w}$, respectively.

From \eqref{f_2^1_z_0_alpha} and \eqref{hat_F_2_z_w}, we have
\begin{equation}\label{f_2^1_z_0_0_0}
    \begin{array}{ll}
    f_2^1(z,0,0,0)&=\bar{\Psi}\left(\begin{array}{l}
    \big\langle \beta_{n_1},F_2(z,0,0,0)\big\rangle\\
    \big\langle \beta_{n_2},F_2(z,0,0,0)\big\rangle
    \end{array}\right)\\
    &=\bar{\Psi}\left(\begin{array}{l}
    \sum_{|\iota|=2}F_{\iota_1\iota_2\iota_3\iota_4}
    \int_{0}^{\ell\pi}\xi_{n_1}^{\iota_1+\iota_2+1}
    \xi_{n_2}^{\iota_3+\iota_4}\mathrm{d}x
    z_1^{\iota_1}{z}_2^{\iota_2}
    z_3^{\iota_3}{z}_4^{\iota_4}\\
    \sum_{|\iota|=2}F_{\iota_1\iota_2\iota_3\iota_4}
    \int_{0}^{\ell\pi}\xi_{n_1}^{\iota_1+\iota_2+1}
    \xi_{n_2}^{\iota_3+\iota_4}
    \mathrm{d}xz_1^{\iota_1}{z}_2^{\iota_2}
    z_3^{\iota_3}{z}_4^{\iota_4}
    \end{array}\right).
    \end{array}
\end{equation}
Denote $f_2^1(z,0,0,0)=(f_2^{1(1)},f_2^{1(2)},f_2^{1(3)},f_2^{1(4)})$, and we have
\begin{equation}\label{f_2^1_1234}
	\begin{array}{ll}
f_2^1(1)=f_{2000}^{1(1)}z_1^2&+f_{1100}^{1(1)}z_2z_2+f_{1010}^{1(1)}z_1z_3
	+f_{1001}^{1(1)}z_1z_4+f_{0200}^{1(1)}z_2^2
	 \vspace{0.2cm} \\
	&+f_{0110}^{1(1)}z_2z_3
	+f_{0101}^{1(1)}z_2z_4+f_{0020}^{1(1)}z_3^2+f_{0011}^{1(1)}z_3z_4
	+f_{0002}^{1(1)}z_4^2,\vspace{0.2cm}\\
f_2^1(2)=f_{2000}^{1(2)}z_1^2&+f_{1100}^{1(2)}z_2z_2+f_{1010}^{1(2)}z_1z_3
	+f_{1001}^{1(2)}z_1z_4+f_{0200}^{1(2)}z_2^2
	\vspace{0.2cm}\\
	&+f_{0110}^{1(2)}z_2z_3
	+f_{0101}^{1(2)}z_2z_4+f_{0020}^{1(2)}z_3^2+f_{0011}^{1(2)}z_3z_4
	+f_{0002}^{1(2)}z_4^2,\vspace{0.2cm}\\
f_2^1(3)=f_{2000}^{1(3)}z_1^2&+f_{1100}^{1(3)}z_2z_2+f_{1010}^{1(3)}z_1z_3
	+f_{1001}^{1(3)}z_1z_4+f_{0200}^{1(3)}z_2^2
	\vspace{0.2cm}\\
	&+f_{0110}^{1(3)}z_2z_3
	+f_{0101}^{1(3)}z_2z_4+f_{0020}^{1(3)}z_3^2+f_{0011}^{1(3)}z_3z_4
	+f_{0002}^{1(3)}z_4^2,\vspace{0.2cm}\\
f_2^1(4)=f_{2000}^{1(4)}z_1^2&+f_{1100}^{1(4)}z_2z_2+f_{1010}^{1(4)}z_1z_3
	+f_{1001}^{1(4)}z_1z_4+f_{0200}^{1(4)}z_2^2
	\vspace{0.2cm}\\
	&+f_{0110}^{1(4)}z_2z_3
	+f_{0101}^{1(4)}z_2z_4+f_{0020}^{1(4)}z_3^2+f_{0011}^{1(4)}z_3z_4
	+f_{0002}^{1(4)}z_4^2,\vspace{0.2cm}\\
	\end{array}
\end{equation}
where
\begin{equation*}
    \begin{array}{ll}
        f_{2000}^{1(1)}=\bar{\psi}_1F_{2000}\gamma_{30},~
    f_{2000}^{1(2)}=\bar{\psi}_2F_{2000}\gamma_{30},~
    f_{2000}^{1(3)}=\bar{\psi}_3F_{2000}\gamma_{21},~
    f_{2000}^{1(4)}=\bar{\psi}_4F_{2000}\gamma_{21},~\vspace{0.1cm}\\
        f_{1100}^{1(1)}=\bar{\psi}_1F_{1100}\gamma_{30},~
    f_{1100}^{1(2)}=\bar{\psi}_2F_{1100}\gamma_{30},~
    f_{1100}^{1(3)}=\bar{\psi}_3F_{1100}\gamma_{21},~
    f_{1100}^{1(4)}=\bar{\psi}_4F_{1100}\gamma_{21},~\vspace{0.1cm}\\
        f_{1010}^{1(1)}=\bar{\psi}_1F_{1010}\gamma_{21},~
    f_{1010}^{1(2)}=\bar{\psi}_2F_{1010}\gamma_{21},~
    f_{1010}^{1(3)}=\bar{\psi}_3F_{1010}\gamma_{12},~
    f_{1010}^{1(4)}=\bar{\psi}_4F_{1010}\gamma_{12},~\vspace{0.2cm}\\
        f_{1001}^{1(1)}=\bar{\psi}_1F_{1001}\gamma_{21},~
    f_{1001}^{1(2)}=\bar{\psi}_2F_{1001}\gamma_{21},~
    f_{1001}^{1(3)}=\bar{\psi}_3F_{1001}\gamma_{12},~
    f_{1001}^{1(4)}=\bar{\psi}_4F_{1001}\gamma_{12},~\vspace{0.2cm}\\
        f_{0200}^{1(1)}=\bar{\psi}_1F_{0200}\gamma_{30},~
    f_{0200}^{1(2)}=\bar{\psi}_2F_{0200}\gamma_{30},~
    f_{0200}^{1(3)}=\bar{\psi}_3F_{0200}\gamma_{21},~
    f_{0200}^{1(4)}=\bar{\psi}_4F_{0200}\gamma_{21},~\vspace{0.2cm}\\
        f_{0110}^{1(1)}=\bar{\psi}_1F_{0110}\gamma_{21},~
    f_{0110}^{1(2)}=\bar{\psi}_2F_{0110}\gamma_{21},~
    f_{0110}^{1(3)}=\bar{\psi}_3F_{0110}\gamma_{12},~
    f_{0110}^{1(4)}=\bar{\psi}_4F_{0110}\gamma_{12},~\vspace{0.2cm}\\
        f_{0101}^{1(1)}=\bar{\psi}_1F_{0101}\gamma_{21},~
    f_{0101}^{1(2)}=\bar{\psi}_2F_{0101}\gamma_{21},~
    f_{0101}^{1(3)}=\bar{\psi}_3F_{0101}\gamma_{12},~
    f_{0101}^{1(4)}=\bar{\psi}_4F_{0101}\gamma_{12},~\vspace{0.2cm}\\
        f_{0020}^{1(1)}=\bar{\psi}_1F_{0020}\gamma_{12},~
    f_{0020}^{1(2)}=\bar{\psi}_2F_{0020}\gamma_{12},~
    f_{0020}^{1(3)}=\bar{\psi}_3F_{0020}\gamma_{03},~
    f_{0020}^{1(4)}=\bar{\psi}_4F_{0020}\gamma_{03},~\vspace{0.2cm}\\
        f_{0011}^{1(1)}=\bar{\psi}_1F_{0011}\gamma_{12},~
    f_{0011}^{1(2)}=\bar{\psi}_2F_{0011}\gamma_{12},~
    f_{0011}^{1(3)}=\bar{\psi}_3F_{0011}\gamma_{03},~
    f_{0011}^{1(4)}=\bar{\psi}_4F_{0011}\gamma_{03},~\vspace{0.2cm}\\
        f_{0002}^{1(1)}=\bar{\psi}_1F_{0002}\gamma_{12},~
    f_{0002}^{1(2)}=\bar{\psi}_2F_{0002}\gamma_{12},~
    f_{0002}^{1(3)}=\bar{\psi}_3F_{0002}\gamma_{03},~
    f_{0002}^{1(4)}=\bar{\psi}_4F_{0002}\gamma_{03},~\vspace{0.2cm}\\
    \end{array}
\end{equation*}
with $\gamma_{ij}=\int_0^{\ell\pi}\xi_{n_1}^i(x)\xi_{n_2}^j(x) \text{d}x$, and
\begin{equation*}
\begin{array}{ll}
\displaystyle\int_{0}^{\ell\pi}\xi_{n_1}(x)\xi_{n_2}^2(x)\mathrm{d}x&=
\displaystyle\begin{cases}
\cfrac{1}{\sqrt{\ell\pi}},&n_1=0,
\vspace{0.1cm}\\
~~~0,& otherwise,\\
\end{cases}\\
\displaystyle\int_{0}^{\ell\pi}\xi_{n_1}^2(x)\xi_{n_2}(x)\mathrm{d}x&=
\begin{cases}
\cfrac{1}{\sqrt{\ell\pi}},&n_2=n_1=0,
\vspace{0.1cm}\\
\cfrac{1}{\sqrt{2\ell\pi}},&n_2=2n_1\neq 0,
\vspace{0.1cm}\\
~~~0,& otherwise.\\
\end{cases}\\
\end{array}
\end{equation*}

Let $U_2^1(z,0)=(U_2^{1(1)}, U_2^{1(2)}, U_2^{1(3)}, U_2^{1(4)})$, and then by
\eqref{U_j}, \eqref{M_j^1(z,lapha)}, \eqref{f_2^1_1234} and \eqref{f_2^1_z_0_0_0}, we have
\begin{equation*}
    \begin{array}{ll}
 U_2^{1(1)}=&
     \cfrac{1}{i\omega_1}f_{2000}^{1(1)}z_1^2
    -\cfrac{1}{i\omega_1}f_{1100}^{1(1)}z_2z_2
    +\cfrac{1}{i\omega_2}f_{1010}^{1(1)}z_1z_3
    -\cfrac{1}{i\omega_2}f_{1001}^{1(1)}z_1z_4
    -\cfrac{1}{3i\omega_1}f_{0200}^{1(1)}z_2^2
    -\cfrac{1}{2i\omega_1-i\omega_2}f_{0110}^{1(1)}z_2z_3\\
   &-\cfrac{1}{2i\omega_1+i\omega_2}f_{0101}^{1(1)}z_2z_4
    -\cfrac{1}{i\omega_1-2i\omega_2}f_{0020}^{1(1)}z_3^2
    -\cfrac{1}{i\omega_1}f_{0011}^{1(1)}z_3z_4
    -\cfrac{1}{i\omega_1+2i\omega_2}f_{0002}^{1(1)}z_4^2,\\
 U_2^{1(2)}=&
    \cfrac{1}{3i\omega_1}f_{2000}^{1(2)}z_1^2
    +\cfrac{1}{i\omega_1}f_{1100}^{1(2)}z_1z_2
    +\cfrac{1}{2i\omega_1+i\omega_2}f_{1010}^{1(2)}z_1z_3
    +\cfrac{1}{2i\omega_1-i\omega_2}f_{1001}^{1(2)}z_1z_4
    -\cfrac{1}{i\omega_1}f_{0200}^{1(2)}z_2^2\\
    &+\cfrac{1}{i\omega_2}f_{0110}^{1(2)}z_2z_3
    -\cfrac{1}{i\omega_2}f_{0101}^{1(2)}z_2z_4
    +\cfrac{1}{i\omega_1+2i\omega_2}f_{0020}^{1(2)}z_3^2
    +\cfrac{1}{i\omega_1}f_{0011}^{1(2)}z_3z_4
    +\cfrac{1}{i\omega_1-2i\omega_2}f_{0002}^{1(2)}z_4^2,\\
 U_2^{1(3)}=&
    \cfrac{1}{2i\omega_1-i\omega_2}f_{2000}^{1(3)}z_1^2
    -\cfrac{1}{i\omega_2}f_{1100}^{1(3)}z_1z_2
    +\cfrac{1}{i\omega_1}f_{1010}^{1(3)}z_1z_3
    +\cfrac{1}{i\omega_1-2i\omega_2}f_{1001}^{1(3)}z_1z_4
    -\cfrac{1}{2i\omega_1+i\omega_2}f_{0200}^{1(3)}z_2^2\\
    &-\cfrac{1}{i\omega_1}f_{0110}^{1(3)}z_2z_3
    -\cfrac{1}{i\omega_1+2i\omega_2}f_{0101}^{1(3)}z_2z_4
    +\cfrac{1}{i\omega_2}f_{0020}^{1(3)}z_3^2
    -\cfrac{1}{i\omega_2}f_{0011}^{1(3)}z_3z_4
    -\cfrac{1}{3i\omega_2}f_{0002}^{1(3)}z_4^2,\\
 U_2^{1(4)}=&
    \cfrac{1}{2i\omega_1+i\omega_2}f_{2000}^{1(4)}z_1^2
    +\cfrac{1}{i\omega_2}f_{1100}^{1(4)}z_1z_2
    +\cfrac{1}{i\omega_1+2i\omega_2}f_{1010}^{1(4)}z_1z_3
    +\cfrac{1}{i\omega_1}f_{1001}^{1(4)}z_1z_4
    -\cfrac{1}{2i\omega_1-i\omega_2}f_{0200}^{1(4)}z_2^2\\
    &-\cfrac{1}{i\omega_1-2i\omega_2}f_{0110}^{1(4)}z_2z_3
    -\cfrac{1}{i\omega_1}f_{0101}^{1(4)}z_2z_4
    +\cfrac{1}{3i\omega_2}f_{0020}^{1(4)}z_3^2
    +\cfrac{1}{i\omega_2}f_{0011}^{1(4)}z_3z_4
    -\cfrac{1}{i\omega_2}f_{0002}^{1(4)}z_4^2.\\
    \end{array}
\end{equation*}
Therefore,
\begin{equation}\label{D_zf21_U21}
	\dfrac{1}{3!}\mathrm{Proj}_{\mathrm{Ker}(M_3^1)}(D_zf_2^1(z,0,0,0)U_2^1(z,0))
	=\left(\begin{array}{cc}
	D_{2100}z_1^2z_2+D_{1011}z_1z_3z_4\\
	\overline{D_{2100}}z_1z_2^2+\overline{D_{1011}}z_2z_3z_4\\
	D_{0021}z_3^2z_4+D_{1110}z_1z_2z_3\\
	\overline{D_{0021}}z_3z_4^2+\overline{D_{1110}}z_1z_2z4\\
	\end{array}\right),
\end{equation}
where
\begin{equation*}
	\begin{array}{ll}
 D_{2100}=&\dfrac{1}{6}\Big(
	-\cfrac{1}{i\omega_1}f_{2000}^{1(1)}f_{1100}^{1(1)}
	+\cfrac{1}{i\omega_1}f_{1100}^{1(1)}f_{1100}^{1(2)}
	+\cfrac{2}{3i\omega_1}f_{0200}^{1(1)}f_{2000}^{1(2)}
	-\cfrac{1}{i\omega_2}f_{1010}^{1(1)}f_{1100}^{1(3)}\\
	&+\cfrac{1}{2i\omega_1-i\omega_2}f_{0110}^{1(1)}f_{2000}^{1(3)}
	+\cfrac{1}{i\omega_2}f_{1001}^{1(1)}f_{1100}^{1(4)}
	+\cfrac{1}{2i\omega_1+i\omega_2}f_{0101}^{1(1)}f_{2000}^{1(4)}\Big),\\
 D_{1011}=&\dfrac{1}{6}\Big(
	-\cfrac{2}{i\omega_1}f_{2000}^{1(1)}f_{0011}^{1(1)}
	+\cfrac{1}{i\omega_1}f_{1100}^{1(1)}f_{0011}^{1(2)}
	+\cfrac{1}{2i\omega_1-i\omega_2}f_{0110}^{1(1)}f_{1001}^{1(2)}
	+\cfrac{1}{2i\omega_1+i\omega_2}f_{0101}^{1(1)}f_{1010}^{1(2)}
	-\cfrac{1}{i\omega_2}f_{1010}^{1(1)}f_{0011}^{1(3)}\\
	&+\cfrac{2}{i\omega_1-2i\omega_2}f_{0020}^{1(1)}f_{1001}^{1(3)}
	+\cfrac{1}{i\omega_1}f_{0011}^{1(1)}f_{1010}^{1(3)}
	+\cfrac{1}{i\omega_2}f_{1001}^{1(1)}f_{0011}^{1(4)}
	+\cfrac{1}{i\omega_1}f_{0011}^{1(1)}f_{1001}^{1(4)}
	+\cfrac{2}{i\omega_1+2i\omega_2}f_{0002}^{1(1)}f_{1010}^{1(4)}\Big),\\
 D_{0021}=&\dfrac{1}{6}\Big(
	-\cfrac{1}{i\omega_1}f_{1010}^{1(3)}f_{0011}^{1(1)}
	-\cfrac{1}{i\omega_1-2i\omega_2}f_{1001}^{1(3)}f_{0020}^{1(1)}
	+\cfrac{1}{i\omega_1}f_{0110}^{1(3)}f_{0011}^{1(2)}
	+\cfrac{1}{i\omega_1+2i\omega_2}f_{0101}^{1(3)}f_{0020}^{1(2)}\\
	&-\cfrac{1}{i\omega_2}f_{0020}^{1(3)}f_{0011}^{1(3)}
	+\cfrac{1}{i\omega_2}f_{0011}^{1(3)}f_{0011}^{1(4)}
	+\cfrac{2}{3i\omega_2}f_{0002}^{1(3)}f_{0020}^{1(4)}\Big),\\
 D_{1110}=&\dfrac{1}{6}\Big(
	-\cfrac{2}{2i\omega_1-i\omega_2}f_{2000}^{1(3)}f_{0110}^{1(1)}
	+\cfrac{1}{i\omega_2}f_{1100}^{1(3)}f_{1010}^{1(1)}
	-\cfrac{1}{i\omega_1}f_{1010}^{1(3)}f_{1100}^{1(1)}
	+\cfrac{1}{i\omega_2}f_{1100}^{1(3)}f_{0110}^{1(2)}
	+\cfrac{2}{2i\omega_1+i\omega_2}f_{0200}^{1(3)}f_{1010}^{1(2)}\\
	&+\cfrac{1}{i\omega_1}f_{0110}^{1(3)}f_{1100}^{1(2)}
	-\cfrac{2}{i\omega_2}f_{0020}^{1(3)}f_{1100}^{1(3)}
	-\cfrac{1}{i\omega_1-2i\omega_2}f_{1001}^{1(3)}f_{0110}^{1(4)}
	+\cfrac{1}{i\omega_1+2i\omega_2}f_{0101}^{1(3)}f_{1010}^{1(4)}
	+\cfrac{1}{i\omega_2}f_{0011}^{1(3)}f_{1100}^{1(4)}\Big).\\
	\end{array}
\end{equation*}

\paragraph{(c) The computation of $\mathrm{Proj}_{\mathrm{Ker}(M_3^1)}D_wf_2^1(z,0,0,0)U_2^2(z,0)$ and $\mathrm{Proj}_{\mathrm{Ker}(M_3^1)}D_{\widehat w}f_2^1(z,0,0,0)\widehat U_2^2(z,0)$}.

Recalling from \eqref{tilde_F_2_U} that $\widetilde{F}_2(z,w,\widehat{w},0)
=F_2(z,w,\widehat{w},0)$ and by virtue of \eqref{hat_F_2_z_w}, we have
\begin{equation}\label{tilde_F_2_S_wz}
    \begin{array}{ll}
    \widetilde{F}_2(z,w,\widehat{w},0)&=
    \displaystyle\sum_{|\iota|=2}F_{\iota_1\iota_2\iota_3\iota_4}
    \xi_{n_1}^{\iota_1+\iota_2}
    \xi_{n_2}^{\iota_3+\iota_4}z_1^{\iota_1}{z}_2^{\iota_2}
    z_3^{\iota_3}{z}_4^{\iota_4}+S_2(w)+S_2(\widehat{w})
    +o(|w|^2,|w\widehat{w}|,\widehat{w}^2)\\
    &=S_2(w)+S_2(\widehat{w})
    +o(|w|^2,|w\widehat{w}|,\widehat{w}^2,z^2)\\
    &=S_{wz}(w)z^x+S_{\widehat{w}z}(\widehat{w})z^x
    +o(|w|^2,|w\widehat{w}|,\widehat{w}^2,z^2),\\
    \end{array}
\end{equation}
where $z^x=(z_1\xi_{n_1}, z_2\xi_{n_1},
z_3\xi_{n_2},z_4\xi_{n_2})^T$, and
\begin{equation}\label{S_wz}
    \begin{array}{cc}
    S_{wz}(w)=(S_{wz_1}(w), S_{wz_2}(w),
    S_{wz_3}(w), S_{wz_4}(w)),\\
    S_{\widehat{w}z}(\widehat{w})=(S_{\widehat{w}z_1}(\widehat{w}), S_{\widehat{w}z_2}(\widehat{w}),
    S_{\widehat{w}z_3}(\widehat{w}), S_{\widehat{w}z_4}(\widehat{w})),
    \end{array}
\end{equation}
with $S_{wz_i}$ and $S_{\widehat{w}z_i}$ are linear operators from
$\mathrm{Ker}\pi$ to $X_{\mathbb{C}}$, defined by
\begin{equation}\label{S_wz_i}
	\begin{array}{cc}
	S_{wz_i}(y_1)=F_{w_1z_i}y_1^{(1)}+F_{w_2z_i}y_1^{(2)},~i=1,2,3,4,\\
	S_{\widehat{w}z_i}(y_2)=F_{\widehat{w}_1z_i}y_2^{(1)}
	                     +F_{\widehat{w}_2z_i}y_2^{(2)},~i=1,2,3,4.\\
	\end{array}
\end{equation}
Here, $F_{w_1z_i}, F_{w_2z_i}$, $F_{\widehat{w}_1z_i}, F_{\widehat{w}_2z_i}$
are coefficient vectors.
By \eqref{tilde_F_2_S_wz}, we can easily obtain
\begin{equation*}
	\begin{array}{ll}
	D_w\widetilde{F}(z,0,0,0)(y_1)=S_{wz_1}(y_1)z_1\xi_{n_1}
	+S_{wz_2}(y_1)z_2\xi_{n_1}+S_{wz_3}(y_1)z_3\xi_{n_2}
	+S_{wz_4}(y_1)z_4\xi_{n_2},\\
	D_{\widehat{w}}\widetilde{F}(z,0,0,0)(y_2)=S_{\widehat{w}z_1}(y_2)z_1\xi_{n_1}
	+S_{\widehat{w}z_2}(y_2)z_2\xi_{n_1}
	+S_{\widehat{w}z_3}(y_2)z_3\xi_{n_2}
	+S_{\widehat{w}z_4}(y_2)z_4\xi_{n_2}.\\
	\end{array}
\end{equation*}
Let $U_2^2(z,0)=\sum_{j\geq 0}h_j(z)\xi_j(x)$ with
\begin{equation*}
	h_j(z)=\left(\begin{array}{cc}
	h_j^{(1)}(z)\\
	h_j^{(2)}(z)\\
	\vdots\\
	h_j^{(n)}(z)
	\end{array}\right)
	=\sum_{|\iota|=2}\left(
	\begin{array}{cc}
	h_{j,\iota_1\iota_2\iota_3\iota_4}^{(1)}\\
	h_{j,\iota_1\iota_2\iota_3\iota_4}^{(2)}\\
	\vdots\\
	h_{j,\iota_1\iota_2\iota_3\iota_4}^{(n)}
	\end{array}
	\right)z_1^{\iota_1}z_2^{\iota_2}z_3^{\iota_3}z_4^{\iota_4}.
\end{equation*}
Then we have
\begin{equation*}
    \begin{array}{ll}
    &D_wf_2^1(z,0,0,0)U_2^2(z,0)=\bar{\Psi}\left(
    \begin{array}{ll}
    \big\langle \beta_{n_1},~D_w\widetilde{F}_2(z,0,0,0)U_2^2(z,0)\big\rangle\\
    \big\langle \beta_{n_2},~D_w\widetilde{F}_2(z,0,0,0)U_2^2(z,0)\big\rangle
    \end{array}\right)\\
    &=\bar{\Psi}\left(
    \begin{array}{ll}
  \sum_{j\geq 0}S_{wz_1}(h_j)\gamma_{jn_1n_1}z_1
     + \sum_{j\geq 0}S_{wz_2}(h_j)\gamma_{jn_1n_1}z_2
    &+ \sum_{j\geq 0}S_{wz_3}(h_j)\gamma_{jn_1n_2}z_3\\
    &+ \sum_{j\geq 0}S_{wz_4}(h_j)\gamma_{jn_1n_2}z_4\\
  \sum_{j\geq 0}S_{wz_1}(h_j)\gamma_{jn_1n_2}z_1
     + \sum_{j\geq 0}S_{wz_2}(h_j)\gamma_{jn_1n_2}z_2
    &+ \sum_{j\geq 0}S_{wz_3}(h_j)\gamma_{jn_2n_2}z_3\\
    &+ \sum_{j\geq 0}S_{wz_4}(h_j)\gamma_{jn_2n_2}z_4
    \end{array}\right),\
    \end{array}
\end{equation*}
and
\begin{equation*}
  \begin{array}{ll}
   &D_{\widehat{w}}f_2^1(z,0,0,0)\widehat{U}_2^2(z,0)=\bar{\Psi}\left(
   \begin{array}{ll}
   \big\langle \beta_{n_1},~D_{\widehat{w}}\widetilde{F}_2(z,0,0,0)\widehat{U}_2^2(z,0)
   \big\rangle\\
   \big\langle\beta_{n_2},~ D_{\widehat{w}}\widetilde{F}_2(z,0,0,0)\widehat{U}_2^2(z,0)
   \big\rangle
   \end{array}\right)\\
   &=\bar{\Psi}\left(
   \begin{array}{ll}
   S_{\widehat{w}z_1}(h_0)\gamma_{0n_1n_1}z_1
   + S_{\widehat{w}z_2}(h_0)\gamma_{0n_1n_1}z_2
   + S_{\widehat{w}z_3}(h_0)\gamma_{0n_1n_2}z_3
   + S_{\widehat{w}z_4}(h_0)\gamma_{0n_1n_2}z_4\\
 S_{\widehat{w}z_1}(h_0)\gamma_{0n_1n_2}z_1
   + S_{\widehat{w}z_2}(h_0)\gamma_{0n_1n_2}z_2
   + S_{\widehat{w}z_3}(h_0)\gamma_{0n_2n_2}z_3
   + S_{\widehat{w}z_4}(h_0)\gamma_{0n_2n_2}z_4
  \end{array}\right),\\
  \end{array}
\end{equation*}
where $\gamma_{ijk}=\int_{0}^{\ell\pi}\xi_i(x)\xi_j(x)\xi_k(x)\mathrm{d}x$.
According to \eqref{Ker_M_31}, we only need to calculate the following types of
$h_j$ for some $j\in\mathbb{N}_0$:
$$h_{j,2000},~ h_{j,1100}, ~h_{j,0011}, ~h_{j,1010},
 ~h_{j,1001},~ h_{j,0020}, ~h_{j,0110},$$
and the following discussion is divided into three cases:
$$ \mathrm{I}: n_1=n_2=0,~~~\mathrm{II}: n_1=0, n_2\neq0,~~~\mathrm{III}: n_1\neq 0, n_2\neq 0.$$

\noindent $\mathbf{Case~I:} ~n_1=n_2= 0.$

Clearly,
\begin{equation*}
    \gamma_{jn_1n_2}=\begin{cases}
    \cfrac{1}{\sqrt{\ell\pi}},&j=0,\\
    ~~~0,&j\neq 0,\\
    \end{cases}
\end{equation*}
Then
\begin{equation*}
    \begin{array}{ll}
    &D_wf_2^1(z,0,0,0)U_2^2(z,0)\\
    &=\dfrac{1}{\sqrt{\ell\pi}}\bar{\Psi}\left(
    \begin{array}{cc}
    S_{wz_1}(h_0)z_1 + S_{wz_2}(h_0)z_2
    + S_{wz_3}(h_{0})z_3 + S_{wz_4}(h_{0})z_4\\
    S_{wz_1}(h_{0})z_1 + S_{wz_2}(h_{0})z_2
    + S_{wz_3}(h_0)z_3 + S_{wz_4}(h_0)z_4
    \end{array}\right),\\
    \end{array}
\end{equation*}
and
\begin{equation*}
    \begin{array}{ll}
    D{\widehat{w}}f_2^1(z,0,0,0)\widehat{U}_2^2(z,0)=
    \cfrac{1}{\sqrt{\ell\pi}}\bar{\Psi}\left(
    \begin{array}{ll}
    S_{\widehat{w}z_1}(h_0)z_1
    + S_{\widehat{w}z_2}(h_0)z_2+S_{\widehat{w}z_3}(h_0)z_3
    + S_{\widehat{w}z_4}(h_0)z_4\\
    S_{\widehat{w}z_1}(h_0)z_1
    + S_{\widehat{w}z_2}(h_0)z_2
    +S_{\widehat{w}z_3}(h_0)z_3
    + S_{\widehat{w}z_4}(h_0)z_4
    \end{array}\right).\\
    \end{array}
\end{equation*}
Therefore, we have
\begin{equation*}
\begin{array}{ll}
&\dfrac{1}{3!}\mathrm{Proj}_{\mathrm{Ker}(M_3^1)}
\left(D_wf_2^1(z,0,0,0)U_2^2(z,0)+
D_{\widehat w}f_2^1(z,0,0,0)\widehat U_2^2(z,0)\right)\\
&\quad=\left(\begin{array}{ll}
\big(E_{2100}+\widehat E_{2100}\big)z_1^2z_2+\big(E_{1011}+\widehat E_{1011}\big)z_1z_3z_4\\
\big(E_{1200}+\widehat E_{1200}\big)z_1z_2^2+\big(E_{0111}+\widehat E_{0111}\big)z_2z_3z_4\\
\big(E_{0021}+\widehat E_{0021}\big)z_3^2z_4+\big(E_{1110}+\widehat E_{1110}\big)z_1z_2z_3\\
\big(E_{0012}+\widehat E_{0012}\big)z_3z_4^2+\big(E_{1101}+\widehat E_{1101}\big)z_1z_2z_4\\
\end{array}\right),
\end{array}
\end{equation*}
where
\begin{equation*}
   \begin{array}{ll}
   E_{2100}=\cfrac{1}{6\sqrt{\ell\pi}}\bar{\psi}_1
   \Big(S_{wz_1}(h_{0,1100})
   +S_{wz_2}(h_{0,2000})\Big),\\
   E_{1011}=\cfrac{1}{6\sqrt{\ell\pi}}\bar{\psi}_1
   \Big(S_{wz_1}(h_{0,0011})
   +S_{wz_3}(h_{0,1001})+S_{wz_4}(h_{0,1010})\Big),\\
   E_{0021}=\cfrac{1}{6\sqrt{\ell\pi}}\bar{\psi}_3\Big(
   S_{wz_3}(h_{0,0011})
   +S_{wz_4}(h_{0,0020}\big)\Big),\\
   E_{1110}=\cfrac{1}{6\sqrt{\ell\pi}}\bar{\psi}_3\Big(
   S_{wz_1}(h_{0,0110})
   +S_{wz_2}(h_{0,1010})+S_{wz_3}(h_{0,1100})
   \Big),\\
   \widehat E_{2100}=\cfrac{1}{6\sqrt{\ell\pi}}\bar{\psi}_1\Big(S_{\widehat wz_1}(h_{0,1100})
   +S_{\widehat wz_2}(h_{0,2000})\Big),~~\\
   \widehat E_{1011}=\cfrac{1}{6\sqrt{\ell\pi}}\bar{\psi}_1 \Big(
   S_{\widehat wz_1}(h_{0,0011})+S_{\widehat wz_3}(h_{0,1001})
   +S_{\widehat wz_4}(h_{0,1010})\Big),\\
   \widehat E_{0021}=\cfrac{1}{6\sqrt{\ell\pi}}\bar{\psi}_3\Big(
   S_{\widehat wz_3}(h_{0,0011})
   +S_{\widehat wz_4}(h_{0,0020})\Big),~~\\
   \widehat E_{1110}=\cfrac{1}{6\sqrt{\ell\pi}}\bar{\psi}_3
   \Big(S_{\widehat wz_1}(h_{0,0110})
   +S_{\widehat wz_2}(h_{0,1010})
   +S_{\widehat wz_3}(h_{0,1100})
   \Big),\\
   E_{1200}=\overline{E_{2100}},~E_{0111}=\overline{E_{1011}},~
   E_{0012}=\overline{E_{0021}},~E_{1101}=\overline{E_{1110}},~\\
   \widehat{E}_{1200}=\overline{\widehat{E}_{2100}},
   ~\widehat{E}_{0111}=\overline{\widehat{E}_{1011}},~
   \widehat{E}_{0012}=\overline{\widehat{E}_{0021}},
   ~\widehat{E}_{1101}=\overline{\widehat{E}_{1110}}.~
   \end{array}
\end{equation*}

Now, we compute the $h_{j,\iota_1\iota_2\iota_3\iota_4}$. From \eqref{M_j_operator}, we have
\begin{equation*}
\begin{array}{ll}
M_2^2U_2^2(z,0)=D_z\left(\sum_{j\geq 0}h_j(z)\xi_j(x)\right)Bz
-\mathscr{L}_1\left(\sum_{j\geq 0}h_j(z)\xi_j(x)\right),
\end{array}
\end{equation*}
which leads to
\begin{equation}\label{M_2^2_beat_n_k}
\begin{array}{ll}
&~~~\left\langle \beta_k,~M_2^2\left(\sum_{j\geq 0}h_j(z)\xi_j(x)\right)
\right\rangle\\
&=\left\langle \beta_k,~M_2^2\Big(h_k(z)\xi_k(x)\Big)
\right\rangle\\
&=\left\langle \beta_k,~M_2^2\left(\sum_{|\iota|=2}h_{k,
	\iota_1\iota_2\iota_3\iota_4}z_1^{\iota_1}
z_2^{\iota_2}z_3^{\iota_3}z_4^{\iota_4}\xi_k(x)\right)
\right\rangle\\
&=\left\langle\beta_k,~ D_z\left(\sum_{|\iota|=2}h_{k,
	\iota_1\iota_2\iota_3\iota_4}z_1^{\iota_1}
z_2^{\iota_2}z_3^{\iota_3}z_4^{\iota_4}\xi_k(x)\right)Bz
-\mathscr{L}_1 h_k(z)\xi_k(x)
\right\rangle\\
&=\sum_{|\iota|=2}D_z\left(h_{k,
	\iota_1\iota_2\iota_3\iota_4}z_1^{\iota_1}
z_2^{\iota_2}z_3^{\iota_3}z_4^{\iota_4}\right)Bz+\dfrac{k^2}{\ell^2}
D_0h_k-L_0(h_k)-\widehat{L}_0(h_k)\delta(k)\\
&=2i\omega_1h_{k,2000}z_1^2-2i\omega_1h_{k,0200}z_2^2
+2i\omega_2h_{k,0020}z_3^2-2i\omega_2h_{k,0002}z_4^2
+(i\omega_1+i\omega_2)h_{k,1010}z_1z_3\\
&~~+(i\omega_1-i\omega_2)h_{k,1001}z_1z_4+(-i\omega_1+i\omega_2)
h_{k,0110}z_2z_3+(-i\omega_1-i\omega_2)h_{k,0101}z_2z_4\\
&~~+\dfrac{k^2}{\ell^2}
D_0h_k-L_0(h_k)-\widehat{L}_0(h_k)\delta(k),
\end{array}
\end{equation}
where $\delta(\cdot)$ is the Booean function defined as in \eqref{Boolean}.

In addition, by \eqref{f_j}, we have
\begin{equation}\label{f_2^2}
\begin{array}{ll}
f_2^2(z,0,0,0)&=(I-\pi)\widetilde{F}_2(z,0,0,0)\\
&=\widetilde{F}_2(z,0,0,0)-\Phi~\bar{\Psi}\left(
\begin{array}{ll}
\big\langle\beta_{n_1},~\widetilde{F}_2(z,0,0,0)\big\rangle\xi_{n_1}\\
\big\langle\beta_{n_2},~\widetilde{F}_2(z,0,0,0)\big\rangle\xi_{n_2}\\
\end{array}\right)\\
&=\sum_{|\iota|=2}F_{\iota_1\iota_2\iota_3\iota_4}
\xi_{n_1}^{\iota_1+\iota_2}
\xi_{n_2}^{\iota_3+\iota_4}z_1^{\iota_1}{z}_2^{\iota_2}
z_3^{\iota_3}{z}_4^{\iota_4}-\phi_1f_2^{1(1)}\xi_{n_1}\\
&~~~-\phi_2f_2^{1(2)}\xi_{n_1}-\phi_3f_2^{1(3)}\xi_{n_2}
-\phi_4f_2^{1(4)}\xi_{n_2}.
\end{array}
\end{equation}
which, together with the fact
\begin{equation}
M_2^2U_2^2(z,0)=f_2^2(z,0,0,0),
\end{equation}
yields
\begin{equation}\label{M22U22_f22}
\Big\langle \beta_k,~M_2^2\Big(\sum\nolimits_{j\geq 0}h_j(z)\xi_j(x)\Big)
 \Big\rangle
=\Big\langle\beta_k,~ f_2^2(z,0,0,0)
\Big\rangle.
\end{equation}
 Substituting \eqref{M_2^2_beat_n_k} and \eqref{f_2^2} into \eqref{M22U22_f22} and balancing power of coefficients for $z_1^{\iota_1}{z}_2^{\iota_2}
z_3^{\iota_3}{z}_4^{\iota_4}$ in \eqref{M22U22_f22} gives
\begin{eqnarray*}
	&&h_{0,2000}=\Big(2i\omega_1 I-L_0-\widehat L_0\Big)^{-1}\Big(
	\cfrac{1}{\sqrt{\ell\pi}}F_{2000}
	-\phi_1f_{2000}^{1(1)}
	-\phi_2f_{2000}^{1(2)}
	-\phi_3f_{2000}^{1(3)}
	-\phi_4f_{2000}^{1(4)}\Big),\\
	&&h_{0,1100}=\Big(-L_0-\widehat L_0\Big)^{-1}\Big(
	\cfrac{1}{\sqrt{\ell\pi}}F_{1100}
	-\phi_1f_{1100}^{1(1)}
	-\phi_2f_{1100}^{1(2)}
	-\phi_3f_{1100}^{1(3)}
	-\phi_4f_{1100}^{1(4)}
	\Big),\\
	&&h_{0,1010}=\Big(i(\omega_1+\omega_2)I-L_0-\widehat L_0\Big)^{-1}
	\Big(
	\cfrac{1}{\sqrt{\ell\pi}}F_{1010}
	-\phi_1f_{1010}^{1(1)}
	-\phi_2f_{1010}^{1(2)}
	-\phi_3f_{1010}^{1(3)}
	-\phi_4f_{1010}^{1(4)}\Big),\\
	&&h_{0,1001}=\Big(i(\omega_1-\omega_2)I-L_0-\widehat L_0\Big)^{-1}
	\Big(
	\cfrac{1}{\sqrt{\ell\pi}}F_{1001}
	-\phi_1f_{1001}^{1(1)}
	-\phi_2f_{1001}^{1(2)}
	-\phi_3f_{1001}^{1(3)}
	-\phi_4f_{1001}^{1(4)}\Big),\\
	&&h_{0,0110}=\Big(-i(\omega_1-\omega_2)I-L_0\widehat L_0\Big)^{-1}\Big(
	\cfrac{1}{\sqrt{\ell\pi}}F_{0110}
	-\phi_1f_{0110}^{1(1)}
	-\phi_2f_{0110}^{1(2)}
	-\phi_3f_{0110}^{1(3)}
	-\phi_4f_{0110}^{1(4)}\Big),\\
	&&h_{0,0011}=\Big(-L_0-\widehat L_0\Big)^{-1}\Big(
	\cfrac{1}{\sqrt{\ell\pi}}F_{0011}
	-\phi_1f_{0011}^{1(1)}
	-\phi_2f_{0011}^{1(2)}
	-\phi_3f_{0011}^{1(3)}
	-\phi_4f_{0011}^{1(4)}\Big),\\
	&&h_{0,0020}=\Big(2i\omega_2 I-L_0-\widehat L_0\Big)^{-1}\Big(
	\cfrac{1}{\sqrt{\ell\pi}}F_{0020}
	-\phi_1f_{0020}^{1(1)}
	-\phi_2f_{0020}^{1(2)}
	-\phi_3f_{0020}^{1(3)}
	-\phi_4f_{0020}^{1(4)}\Big).\\
\end{eqnarray*}

\noindent $\mathbf{Case~II:} ~n_1=0,~n_2\neq 0.$

By the fact
\begin{equation*}
    \begin{array}{ll}
    \gamma_{jn_1n_1}=\begin{cases}
    \cfrac{1}{\sqrt{\ell\pi}},&j=0,\\
    ~~~0,&j\neq 0,\\
    \end{cases},~~
    \gamma_{jn_1n_2}=\begin{cases}
    \cfrac{1}{\sqrt{\ell\pi}},&j=n_2,\\
    ~~~0,&j\neq n_2,\\
    \end{cases}\\
    \gamma_{jn_2n_2}=\begin{cases}
    \cfrac{1}{\sqrt{\ell\pi}},&j=0,\\
    \cfrac{1}{\sqrt{2\ell\pi}},&j=2n_2,\\
    ~~~0,& otherwise,\\
    \end{cases}
    \end{array}
\end{equation*}
we have
\begin{equation*}
   \begin{array}{ll}
   &D_wf_2^1(z,0,0,0)U_2^2(z,0)\\
   &=\dfrac{1}{\sqrt{\ell\pi}}\bar{\Psi}\left(
   \begin{array}{cl}
   S_{wz_1}(h_0)z_1 + S_{wz_2}(h_0)z_2
   + S_{wz_3}(h_{n_2})z_3 + S_{wz_4}(h_{n_2})z_4\\
   S_{wz_1}(h_{n_2})z_1 + S_{wz_2}(h_{n_2})z_2
   + S_{wz_3}(h_0)z_3 + S_{wz_4}(h_0)z_4+\frac{1}{\sqrt{2}}\left(
   S_{wz_3}(h_{2n_2})z_3+S_{wz_4}(h_{2n_2})z_4
   \right)
   \end{array}\right),\\
   \end{array}
\end{equation*}
and
\begin{equation*}
   \begin{array}{ll}
   D{\widehat{w}}f_2^1(z,0,0,0)\widehat{U}_2^2(z,0)=
   \cfrac{1}{\sqrt{\ell\pi}}\bar{\Psi}\left(
   \begin{array}{ll}
   S_{\widehat{w}z_1}(h_0)z_1
   + S_{\widehat{w}z_2}(h_0)z_2\\
   S_{\widehat{w}z_3}(h_0)z_3
   + S_{\widehat{w}z_4}(h_0)z_4
   \end{array}\right).\\
   \end{array}
\end{equation*}
Therefore, we obtain
\begin{equation*}
    \begin{array}{ll}
    	&\dfrac{1}{3!}\mathrm{Proj}_{\mathrm{Ker}(M_3^1)}
    \left(D_wf_2^1(z,0,0,0)U_2^2(z,0)+
    D_{\widehat w}f_2^1(z,0,0,0)\widehat U_2^2(z,0)\right)\\
    &\quad=\left(\begin{array}{ll}
    \big(E_{2100}+\widehat E_{2100}\big)z_1^2z_2+\big(E_{1011}+\widehat E_{1011}\big)z_1z_3z_4\\
    \big(E_{1200}+\widehat E_{1200}\big)z_1z_2^2+\big(E_{0111}+\widehat E_{0111}\big)z_2z_3z_4\\
    \big(E_{0021}+\widehat E_{0021}\big)z_3^2z_4+\big(E_{1110}+\widehat E_{1110}\big)z_1z_2z_3\\
    \big(E_{0012}+\widehat E_{0012}\big)z_3z_4^2+\big(E_{1101}+\widehat E_{1101}\big)z_1z_2z_4\\
    \end{array}\right).
    \end{array}
\end{equation*}
where
\begin{eqnarray*}
  &&E_{2100}=\cfrac{1}{6\sqrt{\ell\pi}}\bar{\psi}_1
  \Big(S_{wz_1}(h_{0,1100})
  +S_{wz_2}(h_{0,2000})\Big),\\
  &&E_{1011}=\cfrac{1}{6\sqrt{\ell\pi}}\bar{\psi}_1
  \Big(S_{wz_1}(h_{0,0011})
  +S_{wz_3}(h_{n_2,1001})+S_{wz_4}(h_{n_2,1010})\Big),\\
  &&E_{0021}=\cfrac{1}{6\sqrt{\ell\pi}}\bar{\psi}_3\Big(
  S_{wz_3}(h_{0,0011})
  +S_{wz_4}(h_{0,0020})+\cfrac{1}{\sqrt{2}}
  \big(S_{wz_3}(h_{2n_2,0011})+S_{wz_4}(h_{2n_2,0020})\big)\Big),\\
  &&E_{1110}=\cfrac{1}{6\sqrt{\ell\pi}}\bar{\psi}_3\Big(
  S_{wz_1}(h_{n_2,0110})
  +S_{wz_2}(h_{n_2,1010})+S_{wz_3}(h_{0,1100})
  +\cfrac{1}{\sqrt{2}}
  S_{wz_3}(h_{2n_2,1100})\Big),\\
  &&\widehat E_{2100}=\cfrac{1}{6\sqrt{\ell\pi}}\bar{\psi}_1\Big(S_{\widehat wz_1}(h_{0,1100})
  +S_{\widehat wz_2}(h_{0,2000})\Big),~~\\
  &&\widehat E_{1011}=\cfrac{1}{6\sqrt{\ell\pi}}\bar{\psi}_1 S_{\widehat wz_1}(h_{0,0011})
  ,\\
  &&\widehat E_{0021}=\cfrac{1}{6\sqrt{\ell\pi}}\bar{\psi}_3\Big(
  S_{\widehat wz_3}(h_{0,0011})
  +S_{\widehat wz_4}(h_{0,0020})\Big),~~\\
  &&\widehat E_{1110}=\cfrac{1}{6\sqrt{\ell\pi}}\bar{\psi}_3 S_{\widehat wz_3}(h_{0,1100}),\\
  &&E_{1200}=\overline{E_{2100}},~E_{0111}=\overline{E_{1011}},~
  E_{0012}=\overline{E_{0021}},~E_{1101}=\overline{E_{1110}},~\\
  &&\widehat{E}_{1200}=\overline{\widehat{E}_{2100}},
  ~\widehat{E}_{0111}=\overline{\widehat{E}_{1011}},~
  \widehat{E}_{0012}=\overline{\widehat{E}_{0021}},
  ~\widehat{E}_{1101}=\overline{\widehat{E}_{1110}}.~
\end{eqnarray*}

Using the same method as the one in $\mathbf{Case~I}$, we can obtain

\begin{eqnarray*}
&&h_{0,2000}=\Big(2i\omega_1 I-L_0-\widehat L_0\Big)^{-1}\Big(
           \cfrac{1}{\sqrt{\ell\pi}}F_{2000}
           -\phi_1f_{2000}^{1(1)}
           -\phi_2f_{2000}^{1(2)}\Big),\\
&&h_{0,1100}=\Big(-L_0-\widehat L_0\Big)^{-1}\Big(
           \cfrac{1}{\sqrt{\ell\pi}}F_{1100}
           -\phi_1f_{1100}^{1(1)}
           -\phi_2f_{1100}^{1(2)}
           \Big),\\
&&h_{0,0011}=\Big(-L_0-\widehat L_0\Big)^{-1}\Big(
           \cfrac{1}{\sqrt{\ell\pi}}F_{0011}
           -\phi_1f_{0011}^{1(1)}
           -\phi_2f_{0011}^{1(2)}\Big),\\
&&h_{0,0020}=\Big(2i\omega_2 I-L_0-\widehat L_0\Big)^{-1}\Big(
           \cfrac{1}{\sqrt{\ell\pi}}F_{0020}
           -\phi_1f_{0020}^{1(1)}
           -\phi_2f_{0020}^{1(2)}\Big),\\
&&h_{n_2,1001}=\Big(i(\omega_1-\omega_2)I+\cfrac{n_2^2}{\ell^2}
           D_0-L_0\Big)^{-1}\Big(
           \cfrac{1}{\sqrt{\ell\pi}}F_{1001}
           -\phi_3f_{1001}^{1(3)}
           -\phi_4f_{1001}^{1(4)}\Big),\\
&&h_{n_2,1010}=\Big(i(\omega_1+\omega_2)I+\cfrac{n_2^2}{\ell^2}
           D_0-L_0\Big)^{-1}\Big(
           \cfrac{1}{\sqrt{\ell\pi}}F_{1010}
           -\phi_3f_{1010}^{1(3)}
           -\phi_4f_{1010}^{1(4)}\Big),\\
&&h_{n_2,0110}=\Big(-i(\omega_1-\omega_2)I+\cfrac{n_2^2}{\ell^2}
           D_0-L_0\Big)^{-1}\Big(
           \cfrac{1}{\sqrt{\ell\pi}}F_{0110}
           -\phi_3f_{0110}^{1(3)}
           -\phi_4f_{0110}^{1(4)}\Big),\\
&&h_{2n_2,0011}=\Big(\cfrac{(2 n_2)^2}{\ell^2}D_0-L_0\Big)^{-1}
           \cfrac{1}{\sqrt{2\ell\pi}}F_{0011},\\
&&h_{2n_2,0020}=\Big(2i\omega_2 I+\cfrac{(2n_2)^2}{\ell^2}
           D_0-L_0\Big)^{-1}\cfrac{1}{\sqrt{2\ell\pi}}F_{0020},\\
&&h_{2n_2,1100}=0.
\end{eqnarray*}

\vspace{0.3cm}
\noindent $\mathbf{Case~III:} ~n_1\neq 0,~n_2\neq 0$.

In fact, we have
\begin{equation*}
   \begin{array}{ll}
   \gamma_{jn_kn_k}=\begin{cases}
   \cfrac{1}{\sqrt{\ell\pi}},&j=0,\\
   \cfrac{1}{\sqrt{2\ell\pi}},&j=2n_k,\\
   ~~~0,& otherwise,\\
   \end{cases}\qquad
   \gamma_{jn_1n_2}=\begin{cases}
   \cfrac{1}{\sqrt{\ell\pi}},&j=n_2-n_1=0,\\
   \cfrac{1}{\sqrt{2\ell\pi}},&j=n_1+n_2 ~\mathrm{or}~j=n_2-n_1\neq 0,\\
   ~~~0,& otherwise.\\
   \end{cases}
   \end{array}
\end{equation*}
Then
\begin{equation*}
    \begin{array}{ll}
    &D_wf_2^1(z,0,0,0)U_2^2(z,0)\\
    &=\bar{\Psi}\left(
    \begin{array}{lr}
    &\cfrac{1}{\sqrt{\ell\pi}}\Big(S_{wz_1}(h_0)z_1
    +S_{wz_2}(h_0)z_2\Big)
    +\cfrac{1}{\sqrt{2\ell\pi}}\Big(S_{wz_1}(h_{2n_1})z_1
    +S_{wz_2}(h_{2n_1})z_2+S_{wz_3}(h_{n_1+n_2})z_3\\
    &+S_{wz_4}(h_{n_1+n_2})z_4\Big)
    +\gamma_{(n_2-n_1)n_1n_2}\Big(S_{wz_3}(h_{n_2-n_1})z_3
    +S_{wz_4}(h_{n_2-n_1})z_4\Big)\\
    &\cfrac{1}{\sqrt{\ell\pi}}\Big(S_{wz_3}(h_0)z_3
    +S_{wz_4}(h_0)z_4\Big)
    +\cfrac{1}{\sqrt{2\ell\pi}}\Big(S_{wz_3}(h_{2n_2})z_3
    +S_{wz_4}(h_{2n_2})z_4+S_{wz_1}(h_{n_1+n_2})z_1\\
    &+S_{wz_2}(h_{n_1+n_2})z_2\Big)
    +\gamma_{(n_2-n_1)n_1n_2}\Big(S_{wz_1}(h_{n_2-n_1})z_1
    +S_{wz_2}(h_{n_2-n_1})z_2\Big)\\
    \end{array}\right),\\
    \end{array}
\end{equation*}
and
\begin{equation*}
   \begin{array}{ll}
   D{\widehat{w}}f_2^1(z,0,0,0)\widehat{U}_2^2(z,0)=
   \cfrac{1}{\sqrt{\ell\pi}}\bar{\Psi}\left(
   \begin{array}{ll}
   S_{\widehat{w}z_1}(h_0)z_1
   + S_{\widehat{w}z_2}(h_0)z_2
   +S_{\widehat{w}z_3}(h_0)\delta(n_2-n_1)z_3
   + S_{\widehat{w}z_4}(h_0)\delta(n_2-n_1)z_4\\
   S_{\widehat{w}z_1}(h_0)\delta(n_2-n_1)z_1
   + S_{\widehat{w}z_2}(h_0)\delta(n_2-n_1)z_2
   +S_{\widehat{w}z_3}(h_0)z_3
   + S_{\widehat{w}z_4}(h_0)z_4
   \end{array}\right).\\
   \end{array}
\end{equation*}
Hence, we obtain
\begin{equation*}
   \begin{array}{ll}
   &\dfrac{1}{3!}\mathrm{Proj}_{\mathrm{Ker}(M_3^1)}
   \left(D_wf_2^1(z,0,0,0)U_2^2(z,0)+
   D_{\widehat w}f_2^1(z,0,0,0)\widehat U_2^2(z,0)\right)\\
   &\quad=\left(\begin{array}{ll}
   \big(E_{2100}+\widehat E_{2100}\big)z_1^2z_2+\big(E_{1011}+\widehat E_{1011}\big)z_1z_3z_4\\
   \big(E_{1200}+\widehat E_{1200}\big)z_1z_2^2+\big(E_{0111}+\widehat E_{0111}\big)z_2z_3z_4\\
   \big(E_{0021}+\widehat E_{0021}\big)z_3^2z_4+\big(E_{1110}+\widehat E_{1110}\big)z_1z_2z_3\\
   \big(E_{0012}+\widehat E_{0012}\big)z_3z_4^2+\big(E_{1101}+\widehat E_{1101}\big)z_1z_2z_4\\
   \end{array}\right),
   \end{array}
\end{equation*}
where
\begin{equation*}
   \begin{array}{lr}
   E_{2100}=\cfrac{1}{6\sqrt{\ell\pi}}\bar{\psi}_1\Big(
         S_{wz_1}(h_{0,1100})+S_{wz_2}(h_{0,2000})
         +\cfrac{1}{\sqrt{2}}\big(
         S_{wz_1}(h_{2n_1,1100})+S_{wz_2}(h_{2n_1,2000})\big)\Big),\\
   E_{1011}=\cfrac{1}{6}\bar{\psi}_1
         \Big(\cfrac{1}{\sqrt{\ell\pi}}S_{wz_1}(h_{0,0011})+
         \cfrac{1}{\sqrt{2\ell\pi}}\big(
         	S_{wz_1}(h_{2n_1,0011})+S_{wz_3}(h_{n_1+n_2,1001})
         	+S_{wz_4}(h_{n_1+n_2,1010})\big)\\
         \qquad\qquad+\gamma_{(n_2-n_1)n_1n_2}\big(S_{wz_3}(h_{n_2-n_1,1001})
          +S_{wz_4}(h_{n_2-n_1,1010})\big)
         \Big),\\
   E_{0021}=\cfrac{1}{6\sqrt{\ell\pi}}\bar{\psi}_3\Big(
         S_{wz_3}(h_{0,0011})
         +S_{wz_4}(h_{0,0020})+\cfrac{1}{\sqrt{2}}\big
         (S_{wz_3}(h_{2n_2,0011})+S_{wz_4}(h_{2n_2,0020})\big)\Big),\\
   E_{1110}=\cfrac{1}{6}\bar{\psi}_3
         \Big(\cfrac{1}{\sqrt{\ell\pi}}S_{wz_3}(h_{0,1100})+
         \cfrac{1}{\sqrt{2\ell\pi}}\big(
         S_{wz_3}(h_{2n_2,1100})+S_{wz_1}(h_{n_1+n_2,0110})
          +S_{wz_2}(h_{n_1+n_2,1010})\big)\\
           \qquad\qquad+\gamma_{(n_2-n_1)n_1n_2}\big(
           S_{wz_1}(h_{n_2-n_1,0110})
          +S_{wz_2}(h_{n_2-n_1,1010})\big)
        \Big),\\
   \widehat E_{2100}=\cfrac{1}{6\sqrt{\ell\pi}}
         \bar{\psi}_1\Big(S_{\widehat wz_1}(h_{0,1100})
         +S_{\widehat wz_2}(h_{0,2000})\Big),~~\\
   \widehat E_{1011}=\cfrac{1}{6\sqrt{\ell\pi}}
         \bar{\psi}_1 \Big(
         S_{\widehat wz_1}(h_{0,0011})
         +\delta(n_2-n_1)\big(S_{\widehat wz_3}(h_{0,1001})
         +S_{\widehat wz_4}(h_{0,1010})
         \big)
        \Big),\\
   \widehat E_{0021}=\cfrac{1}{6\sqrt{\ell\pi}}\bar{\psi}_3\Big(
         S_{\widehat wz_3}(h_{0,0011})
         +S_{\widehat wz_4}(h_{0,0020})\Big),~~\\
   \widehat E_{1110}=\cfrac{1}{6\sqrt{\ell\pi}}
         \bar{\psi}_3 \Big(S_{\widehat wz_3}(h_{0,1100})
         +\delta(n_2-n_1)\big(S_{\widehat wz_1}(h_{0,0110})
         +S_{\widehat wz_2}(h_{0,1010})
         \big)
         \Big),\\
   E_{1200}=\overline{E_{2100}},~E_{0111}=\overline{E_{1011}},~
   E_{0012}=\overline{E_{0021}},~E_{1101}=\overline{E_{1110}},~\\
   \widehat{E}_{1200}=\overline{\widehat{E}_{2100}},
   ~\widehat{E}_{0111}=\overline{\widehat{E}_{1011}},~
   \widehat{E}_{0012}=\overline{\widehat{E}_{0021}},
   ~\widehat{E}_{1101}=\overline{\widehat{E}_{1110}}.
   \end{array}
\end{equation*}
With the same method mentioned in $\mathbf{Case~I}$, we have
\begin{eqnarray*}
  && h_{0,2000}=\Big(2i\omega_1 I-L_0-\widehat L_0\Big)^{-1}
       \cfrac{1}{\sqrt{\ell\pi}}F_{2000},\\
  && h_{0,1100}=\Big(-L_0-\widehat L_0\Big)^{-1}
       \cfrac{1}{\sqrt{\ell\pi}}F_{1100},\\
  && h_{0,0011}=\Big(-L_0-\widehat L_0\Big)^{-1}
       \cfrac{1}{\sqrt{\ell\pi}}F_{0011},\\
  && h_{0,0020}=\Big(2i\omega_2 I-L_0-\widehat L_0\Big)^{-1}
       \cfrac{1}{\sqrt{\ell\pi}}F_{0020},\\
  && h_{2n_1,2000}=\begin{cases}
       \Big(2i\omega_1 I+\cfrac{(2n_1)^2}{\ell^2}
       D_0-L_0\Big)^{-1}
       \Big(\cfrac{1}{\sqrt{2\ell\pi}}F_{2000}-\phi_3f_{2000}^{1(3)}
       -\phi_4f_{2000}^{1(4)}\Big),&n_2=2n_1,\\
       \Big(2i\omega_1 I+\cfrac{(2n_1)^2}{\ell^2}
       D_0-L_0\Big)^{-1}
       \cfrac{1}{\sqrt{2\ell\pi}}F_{2000},&n_2\neq 2n_1,
       \end{cases}\\
  && h_{2n_1,1100}=\begin{cases}
       \Big(\cfrac{(2n_1)^2}{\ell^2}
       D_0-L_0\Big)^{-1}
       \Big(\cfrac{1}{\sqrt{2\ell\pi}}F_{1100}-\phi_3f_{1100}^{1(3)}
       -\phi_4f_{1100}^{1(4)}\Big),&n_2=2n_1,\\
       \Big(\cfrac{(2n_1)^2}{\ell^2}
       D_0-L_0\Big)^{-1}
       \cfrac{1}{\sqrt{2\ell\pi}}F_{1100},&n_2\neq 2n_1,
       \end{cases}\\
  && h_{2n_1,0011}=\begin{cases}
  	   \Big(\cfrac{(2n_1)^2}{\ell^2}
  	   D_0-L_0\Big)^{-1} \cfrac{1}{\sqrt{2\ell\pi}}F_{0011},&n_2=n_1,\\
       \Big(\cfrac{(2n_1)^2}{\ell^2}
       D_0-L_0\Big)^{-1}
       \Big(-\phi_3f_{0011}^{1(3)}
       -\phi_4f_{0011}^{1(4)}\Big),&n_2=2n_1,\\
       ~~~~0,& otherwise,
       \end{cases}\\
  && h_{2n_2,0020}=\Big(2i\omega_2 I+\cfrac{(2n_2)^2}{\ell^2}
       D_0-L_0\Big)^{-1}\cfrac{1}{\sqrt{2\ell\pi}}F_{0020},\\
  && h_{2n_2,0011}=\Big(\cfrac{(2n_2)^2}{\ell^2}
       D_0-L_0\Big)^{-1}\cfrac{1}{\sqrt{2\ell\pi}}F_{0011},\\
  && h_{2n_2,1100}=\begin{cases}
  	    \Big(\cfrac{(2n_2)^2}{\ell^2}
      	D_0-L_0\Big)^{-1}
    	\cfrac{1}{\sqrt{2\ell\pi}}F_{1100},&n_2=n_1,\\
    	~~~~0,& otherwise,
      \end{cases}\\
  && h_{n_1+n_2,1001}=\Big((i\omega_1-i\omega_2)I
       +\cfrac{(n_1+n_2)^2}{\ell^2}D_0-L_0\Big)^{-1}
       \cfrac{1}{\sqrt{2\ell\pi}}F_{1001},\\
  && h_{n_1+n_2,1010}=\Big((i\omega_1+i\omega_2)I
       +\cfrac{(n_1+n_2)^2}{\ell^2}D_0-L_0\Big)^{-1}
       \cfrac{1}{\sqrt{2\ell\pi}}F_{1010},\\
  && h_{n_1+n_2,0110}=\Big((-i\omega_1+i\omega_2)I
       +\cfrac{(n_1+n_2)^2}{\ell^2}D_0-L_0\Big)^{-1}
       \cfrac{1}{\sqrt{2\ell\pi}}F_{0110},\\
  && h_{n_2-n_1,1001}=\begin{cases}
       \Big((i\omega_1-i\omega_2)I-L_0-\widehat{L}_0\Big)^{-1}
       \cfrac{1}{\sqrt{\ell\pi}}F_{1001},&n_2=n_1,\\
       \Big((i\omega_1-i\omega_2)I+\cfrac{(n_2-n_1)^2}{\ell^2}
       D_0-L_0\Big)^{-1}
       \Big(\cfrac{1}{\sqrt{2\ell\pi}}F_{1001}-\phi_1f_{1001}^{1(1)}
       -\phi_2f_{1001}^{1(2)}\Big),&n_2=2n_1,\\
       \Big((i\omega_1-i\omega_2)I+\cfrac{(n_2-n_1)^2}{\ell^2}
       D_0-L_0\Big)^{-1}
       \cfrac{1}{\sqrt{2\ell\pi}}F_{1001},& otherwise,
       \end{cases}\\
  && h_{n_2-n_1,1010}=\begin{cases}
       \Big((i\omega_1+i\omega_2)I-L_0-\widehat{L}_0\Big)^{-1}
       \cfrac{1}{\sqrt{\ell\pi}}F_{1010},&n_2=n_1,\\
       \Big((i\omega_1+i\omega_2)I+\cfrac{(n_2-n_1)^2}{\ell^2}
       D_0-L_0\Big)^{-1}
       \Big(\cfrac{1}{\sqrt{2\ell\pi}}F_{1010}-\phi_1f_{1010}^{1(1)}
       -\phi_2f_{1010}^{1(2)}\Big),&n_2=2n_1,\\
       \Big((i\omega_1+i\omega_2)I+\cfrac{(n_2-n_1)^2}{\ell^2}
       D_0-L_0\Big)^{-1}
       \cfrac{1}{\sqrt{2\ell\pi}}F_{1010},& otherwise,
       \end{cases}\\
  && h_{n_2-n_1,0110}=\begin{cases}
       \Big((-i\omega_1+i\omega_2)I-L_0-\widehat{L}_0\Big)^{-1}
       \cfrac{1}{\sqrt{\ell\pi}}F_{0110},&n_2=n_1,\\
       \Big((-i\omega_1+i\omega_2)I+\cfrac{(n_2-n_1)^2}{\ell^2}
       D_0-L_0\Big)^{-1}
       \Big(\cfrac{1}{\sqrt{2\ell\pi}}F_{0110}-\phi_1f_{0110}^{1(1)}
       -\phi_2f_{0110}^{1(2)}\Big),&n_2=2n_1,\\
       \Big((-i\omega_1+i\omega_2)I+\cfrac{(n_2-n_1)^2}{\ell^2}
       D_0-L_0\Big)^{-1}
       \cfrac{1}{\sqrt{2\ell\pi}}F_{0110},& otherwise.
       \end{cases}\\
\end{eqnarray*}

Now, we obtain the full expression of $g_3^1(z,0,0,0)$:
\begin{equation}\label{normal_forms_B4}
	\frac{1}{3!}g_3^1(z,0,0)
	=\frac{1}{3!}\text{Proj}_{\mathrm{Ker}(M_2^1)}f_3^1(z,0,0,0)
	=\left(\begin{array}{ll}
	B_{2100}z_1^2z_2+B_{1011}z_1z_3z_4\\
	B_{1200}z_1z_2^2+B_{0111}z_2z_3z_4\\
	B_{0021}z_3^2z_4+B_{1110}z_1z_2z_3\\
	B_{0012}z_3z_4^2+B_{1011}z_1z_2z_4\\
	\end{array}\right),
\end{equation}
where
\begin{equation*}
    \begin{array}{ll}
    B_{2100}=C_{2100}+\frac{3}{2}(D_{2100}+E_{2100}+\widehat E_{2100}),~
    B_{1011}=C_{1011}+\frac{3}{2}(D_{1011}+E_{1011}+\widehat E_{1011}),~\\
    B_{0021}=C_{0021}+\frac{3}{2}(D_{0021}+E_{0021}+\widehat E_{0021}),~
    B_{1110}=C_{1110}+\frac{3}{2}(D_{1110}+E_{1110}+\widehat E_{1110}),~\\
    B_{1200}=\overline{B_{2100}},~B_{0111}=\overline{B_{1011}},~
    B_{0012}=\overline{B_{0021}},~B_{1101}=\overline{B_{1110}}.
    \end{array}
\end{equation*}
Then, by \eqref{normal_forms_B2} and \eqref{normal_forms_B4}, the
normal forms for double Hopf bifurcation up to the third order take:
\begin{equation}\label{z_normal}
\begin{array}{lr}
\dot{z}_1=~~\,i\omega_1z_1+B_{11}\alpha_1z_1+B_{21}\alpha_2z_1
+B_{2100}z^2_1 z_2+B_{1011}z_1 z_3z_4+h.o.t.,\\
\dot{z}_2=-i\omega_1z_2+\overline{B_{11}}\alpha_1z_2
+\overline{B_{21}}\alpha_2z_2+\overline{B_{2100}}z_1 z^2_2 +\overline{B_{1011}}z_2z_3z_4+h.o.t.,\\
\dot{z}_3=~~\,i\omega_2z_3+B_{13}\alpha_1z_3+B_{23}\alpha_2z_3
+B_{0021}z_3^2 z_4+B_{1110}z_1z_2 z_3+h.o.t.,\\
\dot{z}_2=-i\omega_2z_4+\overline{B_{13}}\alpha_1z_4
+\overline{B_{23}}\alpha_2z_4+\overline{B_{0021}}z_3 z_4^2 +\overline{B_{1110}}z_1z_2z_4+h.o.t.,
\end{array}
\end{equation}
and by virtue of the polar coordinate transformation $z_1=\tilde{\rho}_1 e^{i \theta_1}, z_2=\tilde{\rho}_2e^{i \theta_2}$
and variable substitution:
\begin{equation*}
\epsilon_1=\text{Sign}\big(\text{Re}(B_{2100})\big),~
\epsilon_2=\text{Sign}\big(\text{Re}(B_{0021})\big),~
\rho_1=\tilde{\rho}_1\sqrt{|B_{2100}|},~
\rho_2=\tilde{\rho}_2\sqrt{|B_{0021}|},~
\tilde{t}=t\epsilon_1,
\end{equation*}
the system \eqref{z_normal}, truncated at the third order, becomes
\begin{equation}\label{rho_1_2}
\begin{array}{ll}
\dot{\rho}_1=\rho_1\big(\kappa_1(\alpha)+\rho_1^2+b_0\rho_2^2\big),\\
\dot{\rho}_2=\rho_2\big(\kappa_2(\alpha)+{c}_0\rho_1^2+d_0\rho_2^2\big),
\end{array}
\end{equation}
where
\begin{equation*}
\begin{array}{ll}
\kappa_1(\alpha)= \epsilon_1\left(\text{Re}(B_{11})\alpha_1+\text{Re}(B_{21})\alpha_2\right),\\
\kappa_2(\alpha)= \epsilon_1\left(\text{Re}(B_{13})\alpha_1+\text{Re}(B_{23})\alpha_2\right),\\
b_0= \cfrac{\epsilon_1 \epsilon_2 \text{Re}(B_{1011})}{\text{Re}(B_{0021})},~{c}_0= \cfrac{ \text{Re}(B_{1110})}{\text{Re}(B_{2100})},~d_0=\epsilon_1 \epsilon_2.
\end{array}
\end{equation*}

Clearly, $E_1=(0,0)$ is always an equilibrium and that up to three other non-negative equilibria solutions can appear:
\begin{equation*}
    \begin{array}{ll}
     &E_2=\big(\sqrt{-\kappa_1(\alpha)}, 0\big),~~~\qquad\qquad
     \qquad\qquad
     \mathrm{for}~ \kappa_1(\alpha)<0,\\
     &E_3=\big(0, \sqrt{-\frac{\kappa_2(\alpha)}{d_0}}\big),~~~~\qquad\qquad
     \qquad\quad\quad
     \mathrm{for}~ d_0\kappa_2(\alpha)<0,\\
     &E_4=\big(\sqrt{\frac{b_0\kappa_2(\alpha)-d_0\kappa_1(\alpha)}{d_0-b_0 c_0}}, \sqrt{\frac{c_0\kappa_1(\alpha)-\kappa_2(\alpha)}{d_0-b_0 c_0}}\big),
     ~\mathrm{for}~ \frac{b_0\kappa_2(\alpha)-d_0\kappa_1(\alpha)}{d_0-b_0 c_0}>0,
     \frac{c_0\kappa_1(\alpha)-\kappa_2(\alpha)}{d_0-b_0 c_0}>0.
    \end{array}
\end{equation*}
The dynamics of system \eqref{model_original} near the double Hopf bifurcation point $\mu_0$ is topologically equivalent to
that of \eqref{rho_1_2} near $(\alpha_1, \alpha_2)=(0, 0)$. Here $E_1$ is associated with the positive constant steady state, $E_2$
is associated with the spatially homogeneous periodic solution, $E_3$ is associated with the spatially nonhomogeneous periodic solution,
and $E_4$ is associated with the spatially nonhomogeneous  quasi-periodic solution. Moreover, according to \cite{Guckenheimer-1983}
there are twelve distinct types of unfoldings according to the signs of coefficients $b_0$, ${c}_0$, $d_0$ and $d_0-b_0 c_0$ (see Table \ref{twelve_unfoldings}).

\begin{table}[tbp]
	\centering
	\caption{The twelve unfoldings.}\label{twelve_unfoldings}
	
	\begin{tabular}{ccccccccccccc}
		\hline
		Case & {I}a & {I}b  & {II} & {III}& {IV}a & {IV}b  & {V}& {VI}a& {VI}b& {VII}a& {VII}b& {VIII}     \\ \hline
		$d_0$ &  +1 &+1& +1&+1&+1&+1&--1&--1&--1&--1&--1&--1\\
		$b_0 $ &  + &+& +&--&--&--&+&+&+&--&--&--\\
		$c_0$ &  + &+& --&+&--&--&+&--&--&+&+&--\\
		$d_0-b_0c_0$ &  + &--& +&+&+&--&--&+&--&+&--&--\\
		\hline
	\end{tabular}
\end{table}

\begin{remark}
	The expressions presented to calculate the normal forms seem complicated and tedious, and some very important coefficient vectors are not explicitly shown, for example, $F_{\iota_1\iota_2\iota_3\iota_4}$, $F_{w_1z_i}$, $F_{w_2z_i}$,
	$i=1,2,3,4$.
	For an abstract equation with a large number of variables, it may take a lot of work to express all these coefficient vectors, but if the system  is only of two variables, i.e., $n=2$, this complexity can be reduced. We could give the explicit representation of these coefficient vectors by  partial derivatives of $F$ and the basis of center subspace, and we put this in Appendix. 	
\end{remark}

\section{Applicaton to a predator-prey model}

In this section, we consider the following reaction-diffusion Holling-Tanner system with nonlocal prey competition:
\begin{equation}\label{model_holling_tanner}
\begin{cases}
\cfrac{\partial u}{\partial t}=d_1\Delta u+u\Big(1-\beta\widehat{u}\Big)-\cfrac{buv}{1+u}, &x\in\Omega, t>0,\\
\cfrac{\partial v}{\partial t}=d_2\Delta v+cv \Big(1-\cfrac{v}{u}\Big),&x\in\Omega, t>0,\\
\partial_{\nu}u=\partial_{\nu}v=0,&x\in \partial\Omega,t>0,\\
u(x,0)=u_0(x)>0,~~v(x,0)=v_0(x)>0,&x\in\Omega,
\end{cases}
\end{equation}
where $\Omega=(0,\ell\pi)$, $u(x,t)$ and $v(x,t)$ represent the prey and predator densities at location $x$ and time $t$ respectively, $\widehat{u}=\frac{1}{\ell\pi}\int_{0}^{\ell\pi} u(y,t)dy$ stands for the nonlocal prey competition, and $\beta, b, c, d_1, d_2$ are parameters and all positive. Particularly, $c$ is the intrinsic growth rate of predator, $d_1$ and $d_2$ are the diffusive rates, and $b$ and $\beta$ measure the strength of
interspecific and intraspecific interaction.

The model \eqref{model_holling_tanner} was first proposed and discussed by Merchant and Nagata \cite{Merchant-2011} when $\Omega=(-\infty, +\infty)$
and their results indicate that the
nonlocal competition may be an important mechanism for pattern
formation. When $\Omega=(0,\ell\pi)$, Chen {\it et al.}\cite{ChenW-2018}
studied the existence of spatially nonhomogeneous periodic solutions induced by
nonlocal competition.
When $\widehat{u}=u$, the model \eqref{model_holling_tanner} is reduced to the classical Holling -Tanner predator-prey model, of which the dynamics have been extensively studied (see e.g.
\cite{ChenS-2012, MaL-2013, PengW-2005, PengW-2007, LiJ-2013}
and references therein).
It's worth mentioning that \eqref{model_holling_tanner} is more likely to undergoes a Hopf bifurcation than the classical Holling-Tanner system
\cite{ChenW-2018,LiJ-2013}, which increases the possibility for double Hopf bifurcation.

The system \eqref{model_holling_tanner} has a unique constant positive equilibrium $E_*=(\lambda,\lambda)$ with $\lambda$ satisfying $(1-\beta\lambda)(1+\lambda)=b\lambda$, and $\lambda$ is strictly decreasing with respect to $b$. Therefore, we could choose $\lambda$ as one parameter instead of $b$ and apply the method obtained in previous sections to compute the normal forms of system \eqref{model_holling_tanner} near the double Hopf singularity.

\subsection{Linear analysis and existence of double Hopf bifurcation}

Let $U=(u,v)^{^T}$, and then the linearized system of model \eqref{model_holling_tanner} at the
positive equilibrium $(\lambda,\lambda)$ has the form
\begin{equation*}
	\dot{U}=D\Delta U+L(\lambda,c)U+\widehat{L}(\lambda,c)\widehat{U},
\end{equation*}
where $D=\mathrm{diag}(d_1,d_2)$, and
\begin{equation*}
     L(\lambda, c)= \Bigg(\begin{array}{cc}
     \frac{\lambda (1-\beta\lambda)}{1+\lambda}~&~ -(1-\beta\lambda)\\
     c ~&~ -c
     \end{array}\Bigg),
     \widehat{L}(\lambda, c)=\displaystyle
     \Bigg(\begin{array}{cc}
     -\beta\lambda~&~ 0\\
     0~&~ 0
     \end{array}\Bigg).
\end{equation*}
The sequence of characteristic equations are as follows
\begin{equation}\label{CE_model}
\eta^2 - T_n(\lambda,c)\eta+D_n(\lambda,c)=0,~n\in \mathbb{N}_0,
\end{equation}
where
\begin{equation}\label{T_0_D_0}
T_0(\lambda,c)=-c+\cfrac{\lambda(1-\beta-2\beta \lambda)}{1+\lambda},~~ D_0(\lambda,c)=\beta c \lambda+\cfrac{c(1-\beta\lambda)}{1+\lambda},
\end{equation}
and for $n\in \mathbb{N}$,
\begin{equation}\label{T_n_D_n}
\begin{array}{l}
T_n(\lambda,c)=-c+\cfrac{\lambda(1-\beta \lambda)}{1+\lambda}-(d_1+d_2)\cfrac{n^2}{\ell^2},\\
D_n(\lambda,c)=\cfrac{c(1-\beta\lambda)}{1+\lambda}+\left(d_1 c-\cfrac{d_2\lambda(1-\beta\lambda)}{1+\lambda}\right)\cfrac{n^2}{\ell^2}+d_1d_2\cfrac{n^4}{\ell^4}.
\end{array}
\end{equation}

For simplicity of notations, we denote
\begin{equation}\label{c_0_c_n}
\begin{array}{ll}
p(\lambda)= \cfrac{\lambda(1-\beta\lambda)}{1+\lambda},~~
&c_d(\lambda)= \cfrac{d_2}{d_1}(1-\beta\lambda)\Big(1-\sqrt{\cfrac{1}{1+\lambda}}\Big)^2,\\
c_0(\lambda)=p(\lambda)-\beta\lambda,
&c_n(\lambda)= p(\lambda)-\cfrac{d_1+d_2}{\ell^2}n^2.
\end{array}
\end{equation}
The following lemma is a summary of the properties on \eqref{c_0_c_n}, and the proof
is trivial and we omit it.
\begin{lemma}\label{lemma_p_c_properties}
	Suppose that $d_1, d_2>0$, $\beta>0$, and $p(\lambda), c_{d}(\lambda), c_0(\lambda), c_n(\lambda)$ are defined in \eqref{c_0_c_n}.
	\begin{enumerate}	
		\item If $c>c_d(\lambda)$, then $D_n(\lambda,c)>0$ for all $n\in \mathbb{N}_0$. Moreover, there exists $\lambda_d\in(0,1/\beta)$ such that
		$c'_d(\lambda_d)=0$, $c'_d(\lambda)>0$ for $\lambda\in(0, \lambda_d)$,
		and $c'_d(\lambda)<0$ for $\lambda\in(\lambda_d, 1/\beta)$.
		
		\item Let $\lambda_p =\sqrt{\frac{1+\beta}{\beta}}-1$. Then $p'(\lambda_p)=0$, $p'(\lambda)>0$
		for $\lambda\in(0,\lambda_p)$,  and $p'(\lambda)<0$.
		for $\lambda\in(\lambda_p, 1/\beta)$.
		
		\item If $\beta\leq 1$, then there exists $\lambda_0=\sqrt{\frac{1+\beta}{2\beta}}-1\in(0,\frac{1-\beta}{2\beta})$
		such that $c'_0(\lambda_0)=0$, $c'_0(\lambda)>0$ for $\lambda \in(0, \lambda_0)$, and  $c'_0(\lambda)<0$
		for $\lambda \in(\lambda_0,1/\beta)$. Moreover, $c_0(\lambda)>0$ for $\lambda\in(0, \frac{1-\beta}{2\beta})$, and
		$c_0(\lambda)<0$ for $\lambda\in(\frac{1-\beta}{2\beta}, \frac{1}{\beta})$.
	\end{enumerate}
\end{lemma}

The existence of spatially nonhomogeneous Hopf bifurcation was
studied by Chen {\it et al.} in \cite{ChenW-2018}, here we
state the main results below without proof.

\begin{theorem}\it{[Theorem 2, \cite{ChenW-2018}]}\label{th_1_chen}
	Suppose that $c>c_d(\lambda_d)$, $\beta \geq 1$, and define
	\begin{equation}\label{l_n}
	\ell_n= n\sqrt{\cfrac{d_1+d_2}{p(\lambda_p)-c}}, ~~if~ p(\lambda_p)>c,
	\end{equation}
	where $p(\lambda)$, $\lambda_p$ are defined as Eq.\eqref{c_0_c_n} and Lemma \ref{lemma_p_c_properties}.
	Then the following two statements are true.
	\begin{enumerate}
		\item[$(i)$] If $c\geq p(\lambda_p)$, or $c<p(\lambda_p)$ but $\ell\in(0, \ell_1)$, then $(\lambda, \lambda)$ is
		locally asymptotically stable for $\lambda\in(0, 1/\beta)$.
		\item[$(ii)$] If $c<p(\lambda_p)$ and $\ell\in(\ell_n, \ell_{n+1}]$, then there exist two sequences $\{\lambda_{j,-}^H\}$
		and $\{\lambda_{j,+}^H\}$, $1\leq j \leq n$, such that
		\begin{equation}\label{T_j}
		T_j(\lambda_{j,-}^H)=T_j(\lambda_{j,+}^H)=0 ~and~  T_i(\lambda_{j,\pm}^H)\neq 0 ~for~ i\neq j,
		\end{equation}
		and these points satisfy
		\begin{equation}\label{sqs_lambda}
		0<\lambda_{1,-}^H<\lambda_{2,-}^H<\cdots <\lambda_{n,-}^H<\lambda_p<\lambda_{n,+}^H<\cdots <
		\lambda_{2,+}^H<\lambda_{1,+}^H<1/\beta,
		\end{equation}
		where $T_j(\lambda)$ is defined as in Eq.\eqref{T_n_D_n}, such that $(\lambda, \lambda)$ is locally asymptotically stable
		for $\lambda \in(0, \lambda_{1,-}^H)\cup (\lambda_{1,+}^H, 1/\beta)$ and unstable for $\lambda\in(\lambda_{1,-}^H, \lambda_{1,+}^H)$. Moreover, system \eqref{model_holling_tanner} undergoes Hopf bifurcation at $(\lambda, \lambda)$ when
		$\lambda= \lambda_{j,\pm}^H$, $1\leq j \leq n$, and the bifurcation periodic solutions near $\lambda_{j,\pm}^H$
		are spatially nonhomogeneous.
	\end{enumerate}
\end{theorem}

\begin{theorem}\it{[Theorem 3, \cite{ChenW-2018}]}\label{th_2_chen}
	Suppose that $\beta<1$, $c>\max\big\{c_d(\lambda_d), c_0(\lambda_0)\big\}$ with $c_d, c_0, \lambda_d, \lambda_0$ are
	defined as in Eq.\eqref{c_0_c_n} and Lemma \ref{lemma_p_c_properties}  respectively, and $\ell_n$ is defined as
	in Eq.\eqref{l_n}. Then the following two statements
	are true.
	\begin{enumerate}
		\item[$(i)$] If $c\geq p(\lambda_p)$, or $c_0(\lambda_0)<c<p(\lambda_p)$ but $\ell\in(0, \ell_1)$, then $(\lambda, \lambda)$ is
		locally asymptotically stable for $\lambda\in(0, 1/\beta)$.
		\item[$(ii)$] If $c<p(\lambda_p)$ and $\ell\in(\ell_n, \ell_{n+1}]$, then $(\lambda, \lambda)$ is locally asymptotically stable
		for $\lambda \in(0, \lambda_{1,-}^H)\cup (\lambda_{1,+}^H, 1/\beta)$ and unstable for $\lambda\in(\lambda_{1,-}^H, \lambda_{1,+}^H)$. Moreover, system \eqref{model_holling_tanner} undergoes Hopf bifurcation at $(\lambda, \lambda)$ when
		$\lambda= \lambda_{j,\pm}^H$, $1\leq j \leq n$, and the bifurcation periodic solutions near $\lambda_{j,\pm}^H$
		are spatially nonhomogeneous, where $\lambda_{j,\pm}^H$ are defined as in Eq.\eqref{T_j} and Eq.\eqref{sqs_lambda}.
	\end{enumerate}
\end{theorem}

	Note that, if $c<p(\lambda_p)$, the large
	scale $\ell$ is always accompanied by nonhomogeneous Hopf bifurcation, but the corresponding branch curve will not intersect in the parameter plane.
	Therefore, the double Hopf point can only be the interaction of spatially homogeneous and  nonhomogeneous Hopf branches.
	Hence we consider the case
$$
\textsc{(H)}~~~\qquad~~~ 0<\beta<1,~~c_d(\lambda_d)<c<c_0(\lambda_0).$$
In this premise, we can restrict the parameter space to a rectangular region, namely,
\begin{equation}\label{R_tg}
R_{tg}= \Big\{(\lambda, c): 0<\lambda<1/\beta, c_d(\lambda_d)<c<c_0(\lambda_0) ~\text{with} ~0<\beta <1\Big\},
\end{equation}
and for fixed $c_d(\lambda_d)<c<c_0(\lambda_0)$, there exist two points $\lambda_{0,-}^H$, $\lambda_{0,+}^H\in (0,1/\beta)$
such that
\begin{equation}\label{T_0}
T_0(\lambda_{0,\pm}^H)=0 ~\text{and}~~
 \lambda_{0,-}^H<\lambda_0<\lambda_{0,+}^H.
\end{equation}
Then we have the following result.
\begin{theorem}\label{th_3_hopf-hopf}
	Let $\ell_*^2 = \frac{2(d_1+d_2)}{1-\beta}$ with $d_1>0, d_2>0$, $0<\beta<1$. Then the following statements hold.
	\begin{enumerate}
		\item[$(i)$]  If $\ell^2<\ell_*^2 $, then the positive
		steady state $(\lambda, \lambda)$ is locally asymptotically stable when $(\lambda, c)\in R_0$ and unstable
		when $(\lambda, c)\in R_{tg}\setminus R_0$, where $R_{tg}$ is defined as in \eqref{R_tg}, and $R_0$ is defined by
		\begin{equation}\label{R_0}
		R_{0}= \Big\{(\lambda, c): 0<\lambda<1/\beta, \tilde{c}_0(\lambda)<c<c_0(\lambda_0) ~\text{with} ~0<\beta <1\Big\},
		\end{equation}
		where $\tilde{c}_n(\lambda)=\max \big\{c_n(\lambda), c_d(\lambda_d)\big\}$, $n\in \mathbb{N}_0$.
		\item[$(ii)$] If $\ell^2>\ell_*^2 $, then there exist  a positive integer $N^*$ and a
		sequence $\{\lambda_{0,n}^{^{HH}}\}_{1\leq n\leq N^*}$, where $\lambda_{0,n}^{^{HH}}=\frac{d_1+d_2}{\beta \ell^2}n^2 \in(0, \frac{1-\beta}{2\beta})$ satisfies
		$$ T_0(\lambda_{0,n}^{^{HH}}, c)=T_n(\lambda_{0,n}^{^{HH}}, c), $$
		and
		$$T_0(\lambda,c)>T_n(\lambda,c)~\text{for}~\lambda\in(0, \lambda_{0,n}^{^{HH}}),~~
		T_0(\lambda,c)<T_n(\lambda,c)~\text{for}~\lambda\in(\lambda_{0,n}^{^{HH}}, \frac{1-\beta}{2\beta}).$$
		Moreover, the positive steady state $(\lambda, \lambda)$ is locally asymptotically stable when $(\lambda, c)\in R_0\bigcap R_1$
		and unstable when  $(\lambda, c)\in R_{tg}\setminus\big( R_0\bigcap R_1\big)$, where $R_{tg}$, $R_0$ are defined as in \eqref{R_tg}
		and \eqref{R_0} respectively, and $R_1$ is defined by
		\begin{equation}\label{R_1}
		R_{1}= \Big\{(\lambda, c) :  0<\lambda<1/\beta, \tilde{c}_1(\lambda)<c<c_0(\lambda_0) ~\text{with} ~0<\beta <1\Big\}.
		\end{equation}
	\end{enumerate}
\end{theorem}

\begin{proof}
	 Denote $S(\lambda, n)=c_0(\lambda)-c_n(\lambda)$, namely,
	\begin{equation}\label{S_lamb_n}
	S(\lambda, n)=-\lambda \beta+\cfrac{d_1+d_2}{\ell^2}n^2, ~\text{for}~\lambda\in(0, \cfrac{1-\beta}{2\beta}), ~n\in \mathbb{N}.
	\end{equation}
	Clearly, $S(0,n)>0$, and
	if $\min S>0$, then $S(\lambda,n)>0$ for any $\lambda \in (0, \frac{1-\beta}{2\beta})$ and $n\in \mathbb{N}$. Note that $S'_{\lambda}(\lambda, n)<0$, and $S'_n(\lambda, n)>0$, hence we have
	$$\min S= S\left(\frac{1-\beta}{2\beta}, 1\right)=\cfrac{\beta-1}{2}+\cfrac{d_1+d_2}{\ell^2}.$$
    It is easy to verify that $\min S=0$ when $\ell^2=\ell_*^2$.

    If $\ell^2<\ell_*^2$, then $S(\lambda, n)>\min S>0$ for any $\lambda\in(0, \frac{1-\beta}{2\beta})$
	and $n\in \mathbb{N}$, which means that $c_0(\lambda)>c_n(\lambda)$. Thus, we have that the stable region is exactly $R_0$, which proves $(i)$.
	
	If $\ell^2>\ell_*^2$, then $S\left(\frac{1-\beta}{2\beta}, 1\right)<0$, and consequently, there exists a positive integer
	\begin{equation}\label{N^*}
	N^*=\begin{cases}
	\ell/\ell_*-1, ~~&\text{if}~~\ell/\ell_* ~\text{is a integer},\\
	\left\lfloor\ell/\ell_*\right\rfloor, ~~&\text{if}~~\ell/\ell_* ~\text{is not a integer},
	\end{cases}
	\end{equation}
	such that $S(\frac{1-\beta}{2\beta}, N^*)$ maximally equals to zero. Consequently for $1\leq n\leq N^*$, we have $S(\frac{1-\beta}{2\beta}, n)<0$, which together with the fact $S(0,n)>0$ yields that there exists
	a $\lambda_{0,n}^{^{HH}}\in (0, \frac{1-\beta}{2\beta})$
	such that $S(\lambda_{0,n}^{^{HH}}, n)=0$.
	The critical point
	$\lambda_{0,n}^{^{HH}}$ can be represented explicitly by
	\begin{equation}\label{lamd_HH}
	\lambda_{0,n}^{^{HH}}=\cfrac{d_1+d_2}{\beta \ell^2}n^2 \in(0, \cfrac{1-\beta}{2\beta}).
	\end{equation}
	
	Since $c_0(\lambda)>c_n(\lambda)$ when $0<\lambda<\lambda_{0,1}^{^{HH}}$ and $c_1(\lambda)>c_j(\lambda)$, $j\in N_0\setminus\{1\}$ when $\lambda_{0,1}^{^{HH}}<\lambda<1/\beta$, we can give the representation of the stable region in the $\lambda-c$ plane as follows
	$$ \Big\{(\lambda, c) : 0<\lambda<\lambda_{0,1}^{^{HH}}, \tilde{c}_0(\lambda)<c<c_0(\lambda_0) ~\text{and}
	~\lambda_{0,1}^{^{HH}}<\lambda<1/\beta, \tilde{c}_1(\lambda)<c<c_0(\lambda_0)\Big\},$$
	which is equivalent to $R_0\bigcap R_1$.
\end{proof}

\begin{remark}
		When $\ell^2>\ell_*^2$, we illustrate Theorem \ref{th_3_hopf-hopf} geometrically in Fig.\ref{fig-branch}. The intersection point
		$P=(\lambda_{0,1}^{^{HH}},c(\lambda_{0,1}^{^{HH}}))$ of $c_0(\lambda)$ and $c_1(\lambda)$ is a possible double Hopf bifurcation point. In our analyses to follow we shall be employing $(\lambda,c)$ as our bifurcation parameters and considering the dynamics of system \eqref{model_holling_tanner} near this point.
\end{remark}

\begin{remark}\label{remark_Turing}
	If $0<c<c_d(\lambda_d)$, then Turing bifurcation or even Turing-Hopf bifurcation may occur under some conditions, the boundary of stable region will become more complex and the system may exhibit rich dynamics near the Turing-Hopf bifurcation point. If the parameters are chosen properly, the coexistence of the spatially nonhomogeneous periodic solutions and spatially nonhomogeneous steady states can be observed \cite{AnJ-2018, ShenW-2019, SongJ-2017}.
\end{remark}
\begin{figure}[htp!]
	\centering
		\includegraphics[width=0.75\textwidth]{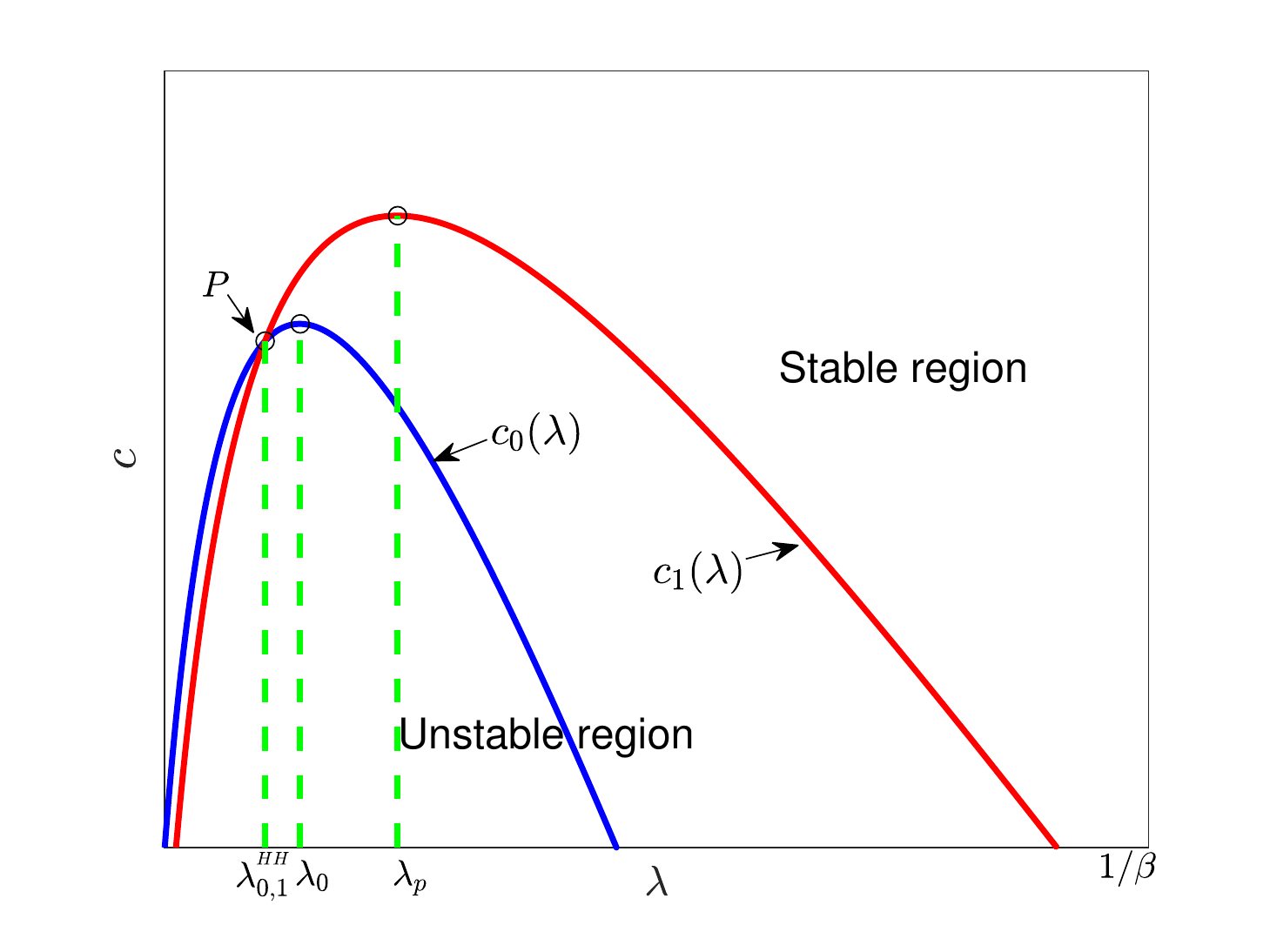}
	\caption{Stability region and bifurcation curves in $\lambda-c$ plane. The blue lines are Hopf bifurcation curves $c=c_0(\lambda)$, while the red, $c=c_1(\lambda)$. The values of parameters are chosen as follows: $d_1=0.6$, $d_2=0.2$, $\beta=0.1$, $\ell^2=8>1.7778=\ell^2_*$.}
	\label{fig-branch}       
\end{figure}

\subsection{Normal forms for double Hopf bifurcation}

It follows from Theorem \ref{th_3_hopf-hopf} that if $0<\beta<1 $ and $\ell^2>\ell^2_*$, then the spatially homogeneous Hopf bifurcation and spatially nonhomogeneous Hopf bifurcation may occur simultaneously. In this section, we shall calculate the normal forms on the center manifold to investigate the dynamics of system \eqref{model_holling_tanner} near the possible double-Hopf bifurcation singularity $(\lambda_0,c_0)=(\lambda_{0,1}^{^{HH}}, c_0(\lambda_{0,1}^{^{HH}}))$.

Let $(u_1(x,t), u_2(x,t))^{^{T}}=(u(x,t)-\lambda, v(x,t)-\lambda)^{^{T}}$, and $\mu=(\mu_1,\mu_2)$ with $\mu_1=\lambda-\lambda_0$, $\mu_2=c-c_0$, then system \eqref{model_holling_tanner} becomes

\begin{equation}\label{mdl_mu}
\cfrac{d U(t)}{dt}=D(\mu)\Delta U+L(\mu)U+\widehat L(\mu)\widehat{U}+F(U,\widehat{U},\mu).
\end{equation}
Consider the Taylor expansion
\begin{equation*}\begin{array}{ll}
&D(\mu)=D_0+\mu_1 D_1^{(1,0)}+\mu_2D_1^{(0,1)}+\cdots\\
&L(\mu)=L_0+\mu_1L_1^{(1,0)}+\mu_2 L_1^{(0,1)}+\cdots\\
&\widehat L(\mu)= \widehat L_0+\mu_1 \widehat L_1^{(1,0)}+\mu_2 \widehat L_1^{(0,1)}+\cdots
\end{array}
\end{equation*}
where $D_0=\mathrm{diag}(d_1,d_2)$, $ D_1^{(1,0)}= D_1^{(0,1)}=\textbf{0}$, $\widehat L_1^{(0,1)}=\textbf{0}$ and
\begin{equation}
\begin{array}{ll}
L_0= \Bigg(\begin{array}{cc}
\cfrac{\lambda_0(1-\beta\lambda_0)}{1+\lambda_0}~&-(1-\beta\lambda_0)
\\
c_0~&-c_0
\end{array}\Bigg),~
L_1^{(1,0)}=\Bigg(
\begin{array}{cc}
\cfrac{1-2\beta\lambda_0-\beta\lambda_0^2}{(1+\lambda_0)^2}
~&~\beta\\
0~&~0
\end{array}\Bigg),\\
L_1^{(0,1)}= \Bigg(\begin{array}{cc}
0 & 0\\
1 &-1
\end{array}\Bigg),~~~
\widehat L_0= \Bigg(\begin{array}{cc}
-\beta\lambda_0 & 0\\
0& 0
\end{array}\Bigg),~~~
\widehat L_1^{(1,0)}=\Bigg(\begin{array}{cc}
-\beta& 0\\
0& 0
\end{array}\Bigg).\\
\end{array}
\end{equation}
Then Eq.\eqref{mdl_mu} can be rewritten as
\begin{equation}\label{mdl_abstract}
\cfrac{d U(t)}{dt}=\mathscr{L}U+\widetilde{F}(U,\widehat{U},\mu).
\end{equation}
where
\begin{equation}\label{eq_F_mu}
\begin{array}{ll}
\mathscr{L}U= D_0 \Delta U+L_0 U+\widehat L_0 \widehat{U},\\ \widetilde{F}(U,\widehat{U},\mu)=(L(\mu)-L_0)U+(\widehat L(\mu)-\widehat L_0) U+F(U,\widehat{U},\mu).
\end{array}
\end{equation}

It follows from Section \ref{decomp_phase_space} that $\mathscr{L}$ and its adjoint $\mathscr{L}^*$ have two pairs of purely imaginary roots $\pm i\omega_1$ and $\pm i\omega_2$ with
\begin{equation}\label{omega}
\omega_1=\sqrt{D_0(\lambda_0, c_0)}~~\text{and}~~\omega_2=\sqrt{D_1(\lambda_0, c_0)}
\end{equation}
and other eigenvalues have negative real parts.

Denote $\{\phi_1 \xi_{n_1}, \bar{\phi}_1 \xi_{n_1}, \phi_2 \xi_{n_2}, \bar{\phi}_2 \xi_{n_2}\}$ and $\{\psi_1 \xi_{n_1}, \bar{\psi}_1 \xi_{n_1}, \psi_2 \xi_{n_2}, \bar{\psi}_2 \xi_{n_2}\}$ the eigenfunctions of $\mathscr{L}$ and its dual $\mathscr{L}^*$ relative to $\Lambda=\{\pm i\omega_1, \pm i\omega_2\}$ with
 $n_1=0$, $n_2=1$ such that
\begin{equation*}
\begin{array}{ll}
\mathscr{L}\phi_j \xi_{n_j}=i\omega_j \phi_j \xi_{n_j},~~ \mathscr{L}^*\psi_j \xi_{n_j}=-i\omega_j \psi_j \xi_{n_j},~\text{and}~~<\psi_j, \phi_j>=1,~~j=1,2,
\end{array}
\end{equation*}
where $\xi_{n_j}$ is defined as in \eqref{xi}.
Specifically, we let $\phi_j=(1, q_j)^{^{T}}, \psi_j=M_j (1, p_j)$, and after a direct calculation, we have
\begin{equation}\label{eigen_function}
\begin{array}{ll}
q_1=\cfrac{c_0}{i\omega_1+c_0}, ~~q_2=\cfrac{c_0}{i\omega_2
	+\frac{d_2}{\ell^2}+c_0},~~
p_1=\cfrac{1-\beta\lambda_0}{i\omega_1-c_0},~~ p_2=
\cfrac{1-\beta\lambda_0}{i\omega_2-\frac{d_2}{\ell^2}-c_0},\\
M_1=\cfrac{(i\omega_1-c_0)^2}{(i\omega_1-c_0)^2
	-c_0(1-\beta\lambda_0)},~~
M_2=\cfrac{(i\omega_2-
	\frac{d_2}{\ell^2}-c_0)^2}{(i\omega_2-\frac{d_2}{\ell^2}-c_0)^2
	-c_0(1-\beta\lambda_0)}.
\end{array}
\end{equation}
Noticing that $\delta(n_1)=1$ and $\delta(n_2)=0$,
then from \eqref{g_2^1} and \eqref{normal_forms_B2}, we have
\begin{equation}\label{normal_B_model}
\begin{array}{ll}
B_{11}=\bar{\psi}_1 (L_1^{(1,0)}+\widehat L_1^{(1,0)})\phi_1=\cfrac{(i\omega_1+c_0)^2}{(i\omega_1+c_0)^2
	-c_0(1-\beta\lambda_0)}\Big(\dfrac{1-2\beta\lambda_0-\beta\lambda_0^2}
{(1+\lambda_0)^2}-\dfrac{i\omega_1\beta}{i\omega_1+c_0}\Big),\\

B_{21}=\bar{\psi}_1 L_1^{(0,1)}\phi_1=\dfrac{-i\omega_1(1-\beta\lambda_0)}
{(i\omega_1+c_0)^2-c_0(1-\beta\lambda_0)},\\
B_{13}=\bar{\psi}_2 L_1^{(1,0)}\phi_2=\cfrac{(i\omega_2+\frac{d_2}{\ell^2}+c_0)^2}
{(i\omega_2+\frac{d_2}{\ell^2}+c_0)^2-c_0(1-\beta\lambda_0)}
\Big(\dfrac{1-2\beta\lambda_0-\beta\lambda_0^2}
{(1+\lambda_0)^2}+\dfrac{c_0\beta}{i\omega_2+\frac{d_2}{\ell^2}+c_0}\Big),\\
B_{23}=\bar{\psi}_2 L_1^{(0,1)}\phi_2=\dfrac{-(i\omega_2+\frac{d_2}{\ell^2})(1-\beta\lambda_0)}
{(i\omega_2+\frac{d_2}{\ell^2}+c_0)^2-c_0(1-\beta\lambda_0)}.\\

\end{array}
\end{equation}
Applying the method given in Appendix, we obtain
\begin{equation}\label{F_wz_model}
\begin{array}{ll}
F_{w_1z_1}=2(F_{uu}+F_{uv}q_1+F_{u\widehat{u}}),~
&F_{w_2z_1}=2(F_{uv}+F_{vv}q_1),\\
F_{w_1z_3}=2(F_{uu}+F_{uv}q_3),~
&F_{w_2z_3}=2(F_{uv}+F_{vv}q_3),\\
 F_{\widehat w_1z_1}=F_{\widehat w_1z_3}=2F_{u\widehat{u}},~
& F_{\widehat w_2z_1}=F_{\widehat w_2z_3}=0.\\
\end{array}
\end{equation}
and
\begin{equation}\label{F_2000_model}
\begin{array}{ll}
F_{2000}=F_{uu}+q_1^2 F_{vv}+ 2 q_1 F_{uv}+2F_{u\widehat{u}},~&
F_{1100}=2\big[F_{uu}+F_{vv}q_1\bar{q}_1+F_{uv}(q_1+\bar{q}_1)
          +2F_{u\widehat{u}}\big],\\
F_{1010}=2\big[F_{uu}+F_{vv}q_1q_2+F_{uv}(q_1+q_2)+F_{u\widehat{u}}
\big],~&
F_{1001}=2\big[F_{uu}+F_{vv}q_1\bar q_2+F_{uv}(q_1+\bar q_2)+F_{u\widehat{u}}\big],\\
F_{0020}=F_{uu}+ q_2^2 F_{vv}+2 q_2 F_{uv},&
F_{0011}=2\big[F_{uu}+F_{vv}q_2 \bar q_2+F_{uv}(q_2+\bar q_2)\big],\\
F_{0200}=\overline{F_{2000}},~F_{0101}=\overline{F_{1010}},&
F_{0002}=\overline{F_{0020}},~F_{0110}=\overline{F_{1001}}.\\
\end{array}
\end{equation}
 where $F_{uu}=\frac{\partial^2}{\partial u^2}F(0,0,\mu_0)=\frac{\partial^2}{\partial u^2}\widetilde{F}(0,0,0)$. The coefficient vectors required in $\widetilde{F}_3$ are given by

\begin{equation}\label{F_2100_model}
\begin{array}{ll}
F_{2100}=3\big[F_{uuu}+F_{uuv}(2q_1+\bar{q}_1)
+F_{uvv}q_1(2\bar{q}_1+q_1)
+F_{vvv}q_1^2\bar{q}_1 \big],\\
F_{0021}=3\big[F_{uuu}+F_{uuv}(2q_2+\bar{q}_2)
+F_{uvv}q_2(2\bar{q}_2+q_2)
+F_{vvv}q_2^2\bar{q}_2 \big],\\
F_{1110}=6\big[F_{uuu}+F_{uuv}(q_1+\bar{q}_1+q_2)
+F_{uvv}(q_1\bar{q}_1+q_1q_2+q_2\bar{q}_1)
+F_{vvv}q_1\bar{q}_1q_2 \big],\\
F_{1011}=6\big[F_{uuu}+F_{uuv}(q_1+q_2+\bar{q}_2)
+F_{uvv}(q_2\bar{q}_2+q_1q_2+q_1\bar{q}_2)
+F_{vvv}q_2\bar{q}_2q_1 \big].\\
\end{array}
\end{equation}
Other formulas appearing in the process of computing normal forms can be obtained from the above formulas.

\subsection{Numerical simulations}
In this section, we give some simulations to support our theoretical results. The dynamic classification near
the double Hopf bifurcation point is presented by applying the normal form method, and a particular bifurcation diagram and corresponding phase portraits are shown in Fig. \ref{fig-mu}.

Take
$$d_1=0.6,~~d_2=0.2,~~\beta=0.1,~~\ell^2=8.$$
It follows from Theorem \ref{th_3_hopf-hopf} and Eq.\eqref{c_0_c_n}, the double Hopf bifurcation point $(\lambda_0, c_0)=(1, 0.35)$. Note that
$b=\frac{(1-\beta\lambda)(1+\lambda)}{\lambda}$, then using the formulas
 \eqref{z_normal}, \eqref{rho_1_2} and algorithm \eqref{normal_B_model}$\sim$\eqref{F_2100_model}, we have

\begin{equation}\label{kappa_1_2}
\begin{array}{ll}
\kappa_1(\mu)=-0.0375\mu_1+0.5\mu_2,\\
\kappa_2(\mu)=-0.0875\mu_1+0.5\mu_2.
\end{array}
\end{equation}
and
$$\epsilon_1=-1,~~\epsilon_2=-1,~~d_0=1,~~b_0=-0.5363,~~\hat c_0=-0.6230,~~d_0-b_0 \hat c_0=0.6659.$$
Then system \eqref{rho_1_2} becomes
\begin{equation}\label{rho_numerical}
\begin{array}{ll}
\dot{\rho}_1=\rho_1\big(-0.0375\mu_1+0.5\mu_2+\rho_1^2-0.5363\rho_2^2\big),\\
\dot{\rho}_2=\rho_2\big(-0.0875\mu_1+0.5\mu_2-0.6230\rho_1^2+\rho_2^2\big).
\end{array}
\end{equation}

According to the classification for planar vector field in \cite{Guckenheimer-1983}, Case $(IVa)$ occurs, and we can
divide the $\mu_1-\mu_2$ plane into six dynamic regions with
\begin{equation*}
\begin{array}{ll}
L_1: \mu_2=0.075\mu_1, &L_2: \mu_2=0.175\mu_1,\\
T_1: \mu_2=0.1099\mu_1,&T_2: \mu_2=0.1366\mu_1.
\end{array}
\end{equation*}
There are four possible attractors in Fig. \ref{fig-mu}: spatially homogeneous
steady state, spatially homogeneous periodic solution, spatially nonhomogeneous periodic solution and spatially nonhomogeneous quasi-periodic solution. In the following, we give a detailed numerical simulation for these attractors, see Fig. \ref{fig-D1}$\sim$ Fig. \ref{fig-D6}.

We remark that when $(\mu_1,\mu_2)\in D_3\bigcup D_4 \bigcup D_5$, there exists a stable spatially nonhomogeneous quasi-periodic
solution, and this quasi-periodicity is not easily seen in Fig. \ref{fig-D4}. Then we present it on a Poincar\'{e} section. Fix $x=\pi$, we choose the solution curve $(u(\pi,t), v(\pi,t))$ and Poincar\'{e} section $v(\pi,t)=\lambda_0$, and the results are shown in Fig. \ref{fig-D4_poincare} in which we can see that system has a quasi-periodic solution on a 2-torus. Here we only present the case in region $D_4$, since $D_3$ and $D_5$ are similar.
We mention that the spatially nonhomogeneous periodic solution and quasi-periodic solution are new spatiotemporal dynamic behaviors
compared to the original system
without nonlocal terms. This shows that
nonlocal terms can enrich the dynamic behaviors of the system.

\begin{figure}[htp]
	\centering
	\includegraphics[width=6in]{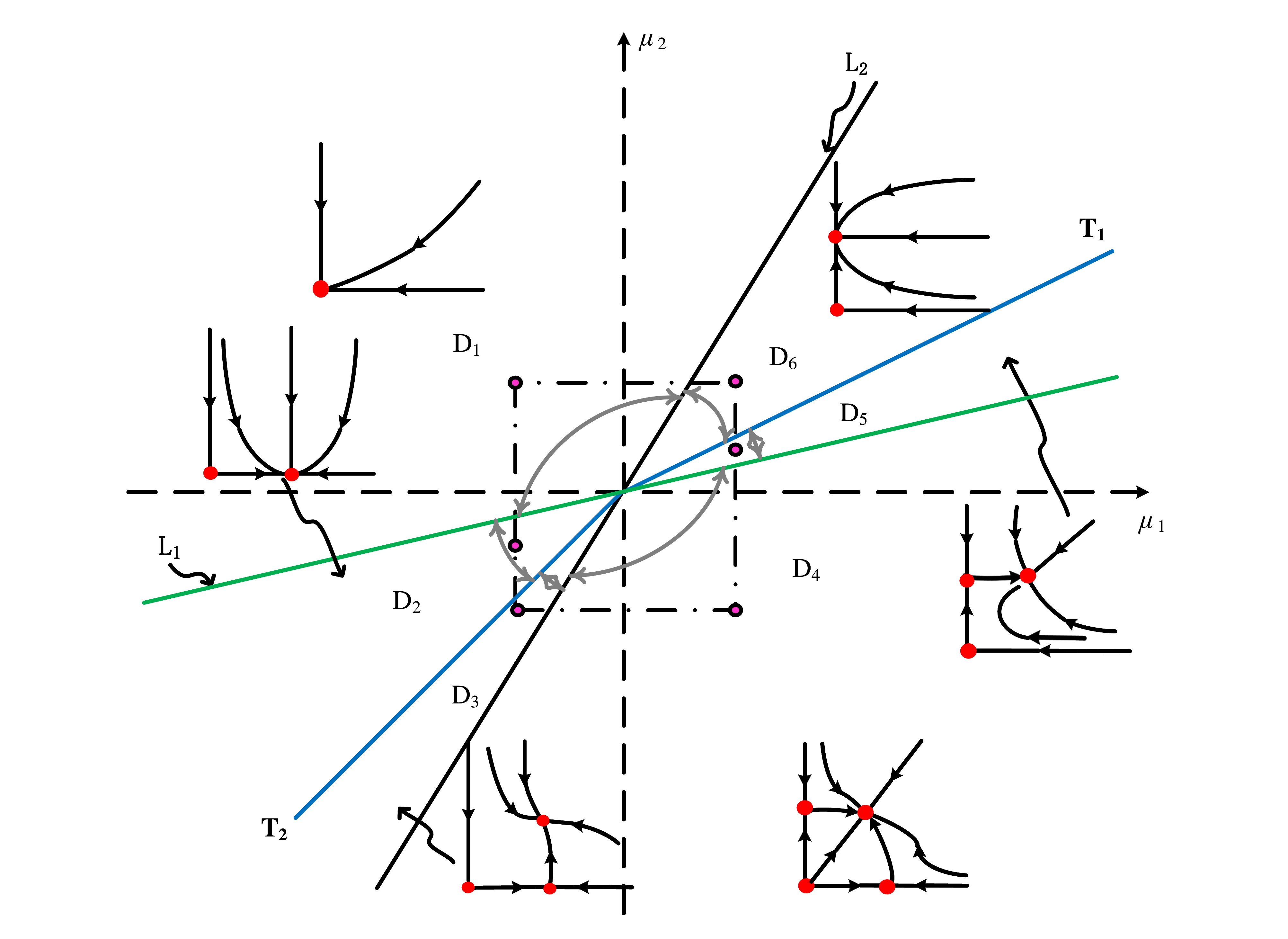}
	\caption{Bifurcation diagram and the corresponding phase portraits of system \eqref{rho_numerical} in the $\mu_1-\mu_2$ plane.}
	\label{fig-mu}       
\end{figure}

\begin{figure}[htp]
	\begin{center}
		\hspace{-0.2cm}
		\subfigure[]{\includegraphics[width=0.33\textwidth]{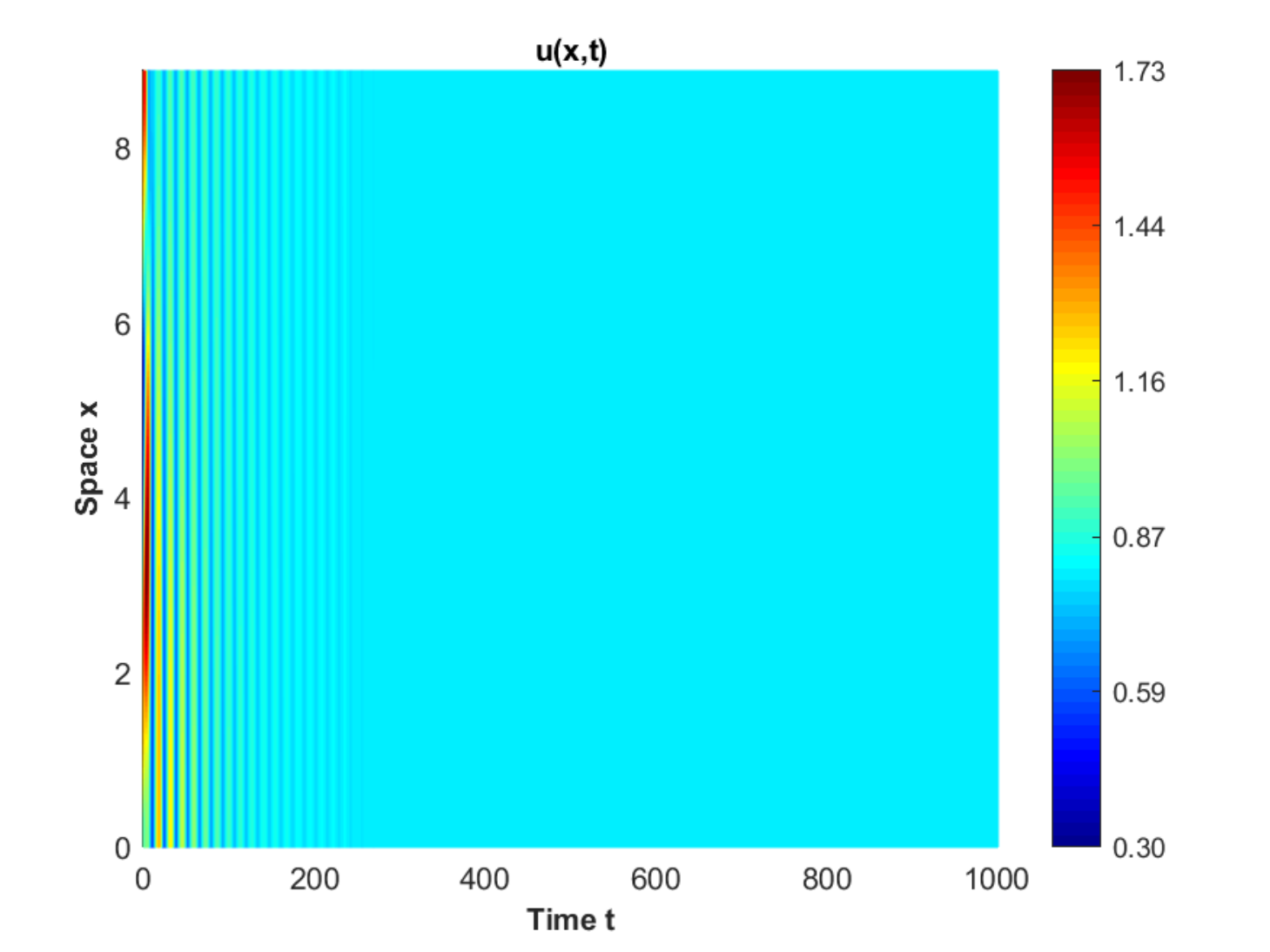}}
		\hspace{-0.2cm}
		\subfigure[]{\includegraphics[width=0.33\textwidth]{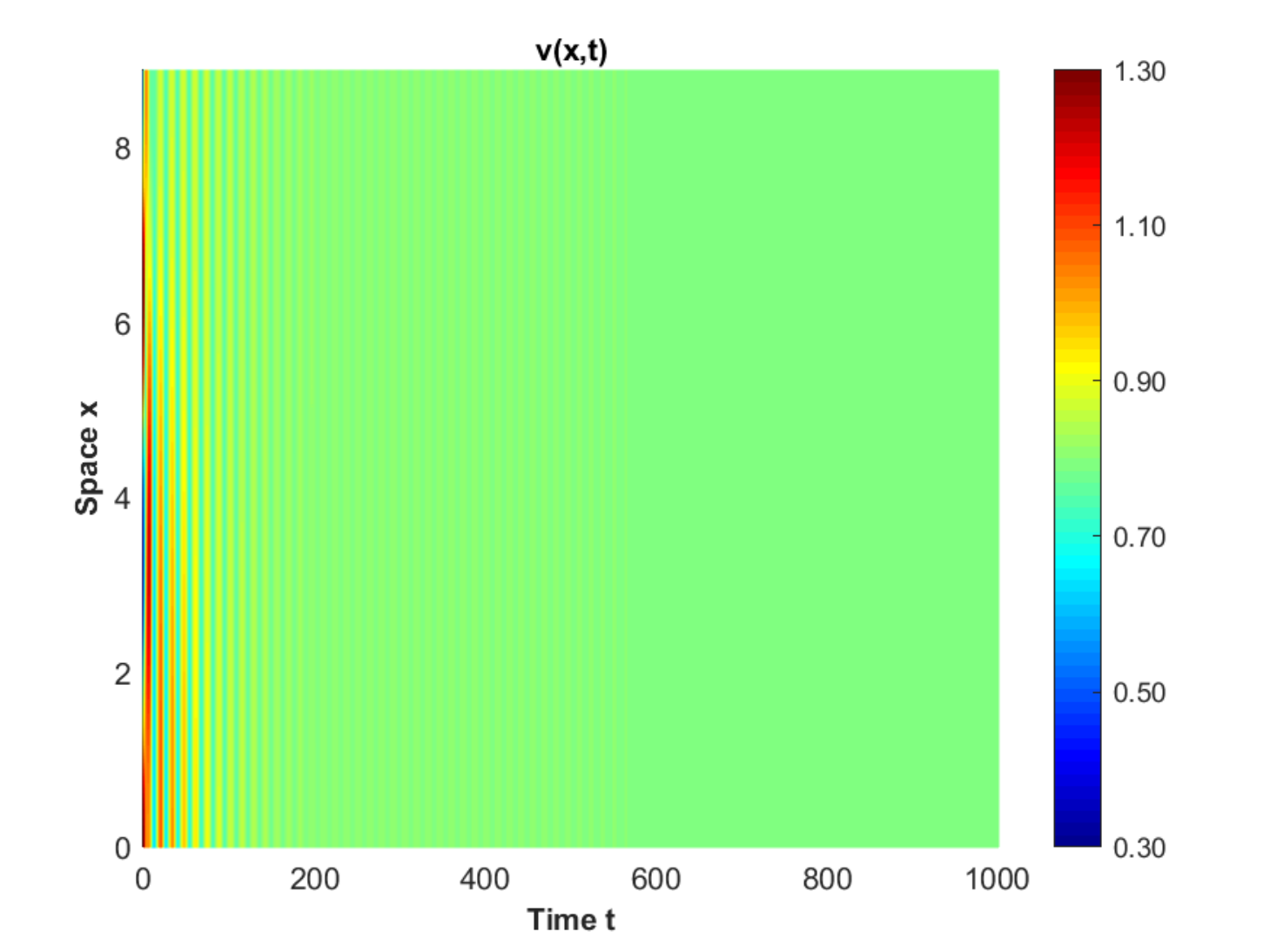}}
		\hspace{-0.2cm}
		\subfigure[]{\includegraphics[width=0.33\textwidth]{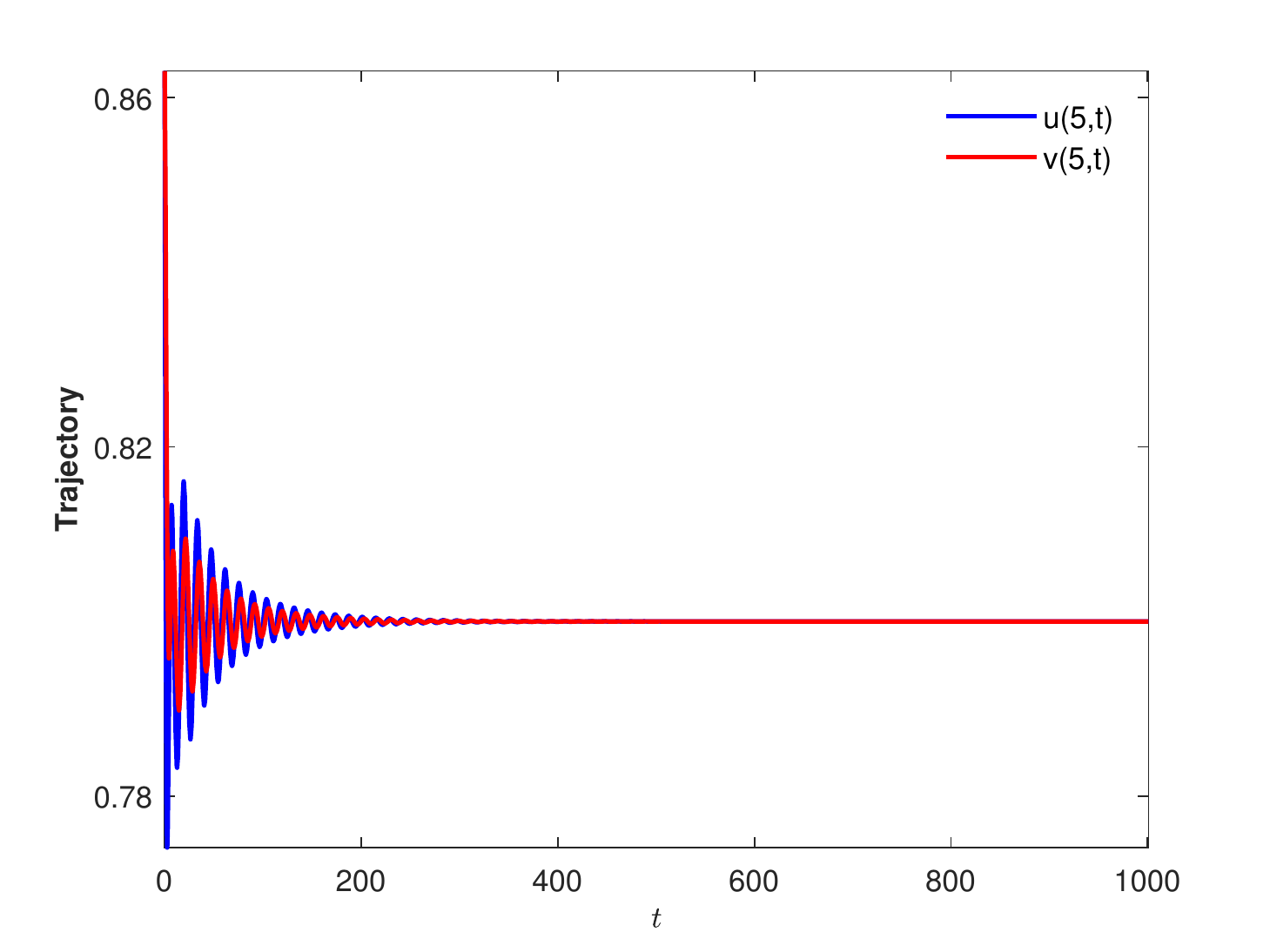}}
		\vspace{-0.2cm}
	\end{center}
	\caption{Stable positive constant steady state with
		$(\mu_1, \mu_2)=(-0.2, 0.00925)\in D_1$. The initial function are chosen as $(\lambda_0+0.5\sin x, \lambda_0+0.5\cos x)$.}\label{fig-D1}
\end{figure}

\begin{figure}[htp]
	\begin{multicols}{3}
		\begin{center}
			\subfigure[]{\includegraphics[width=0.33\textwidth]{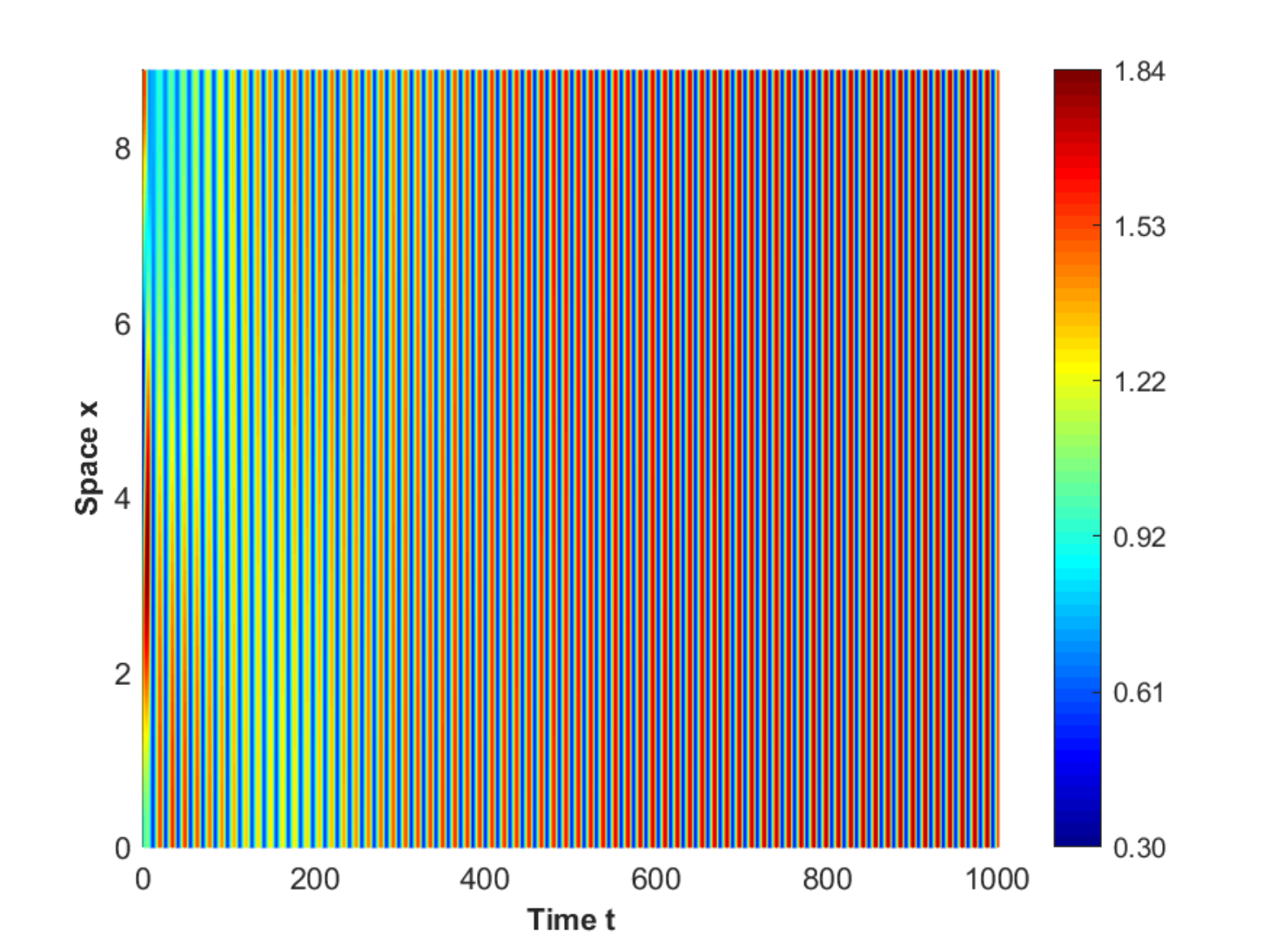}}\vspace{-0.3cm}
			\subfigure[]{\includegraphics[width=0.33\textwidth]{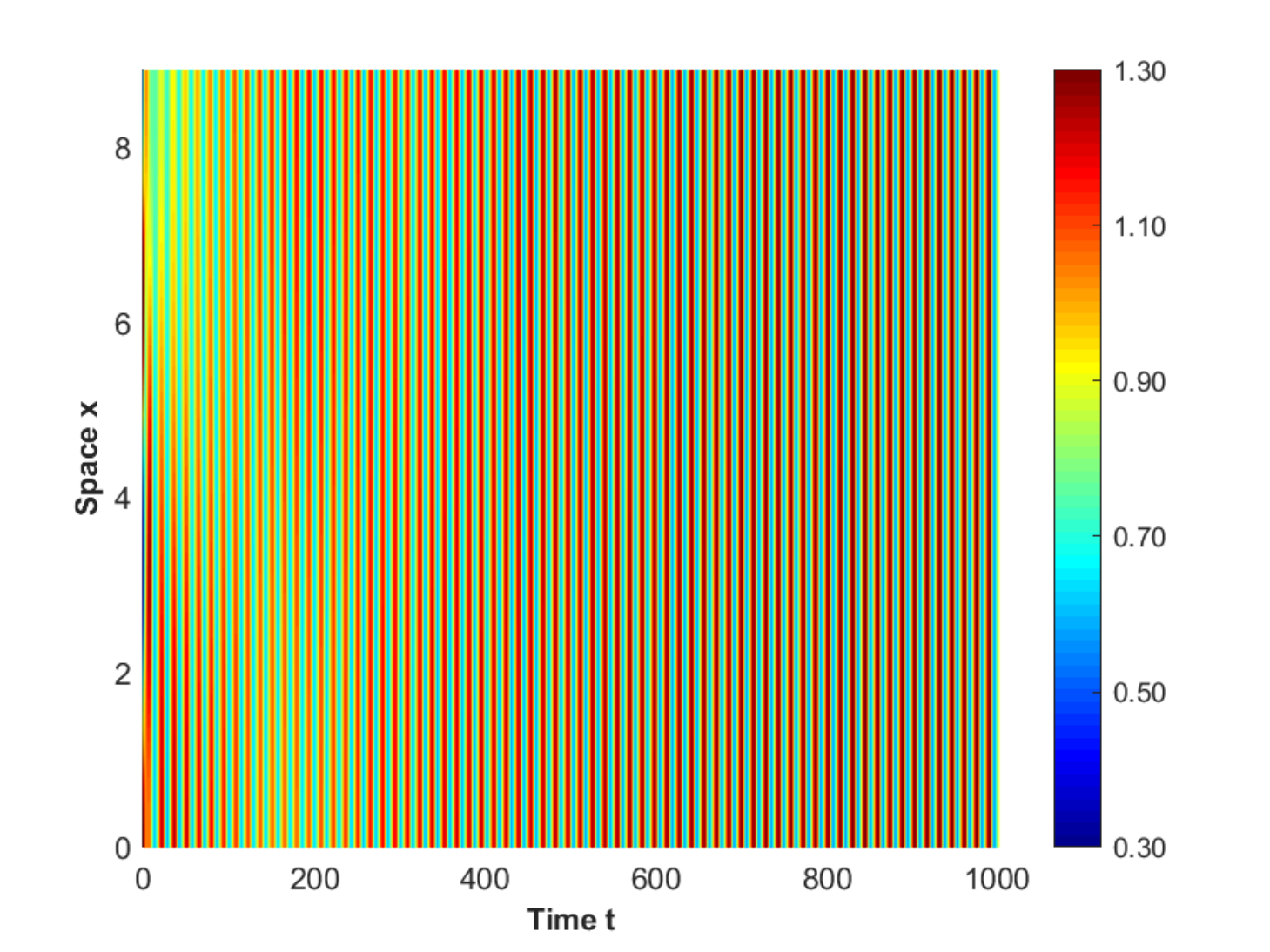}}
		\end{center}
		
		\begin{center}
			\subfigure[]{\includegraphics[width=0.33\textwidth]{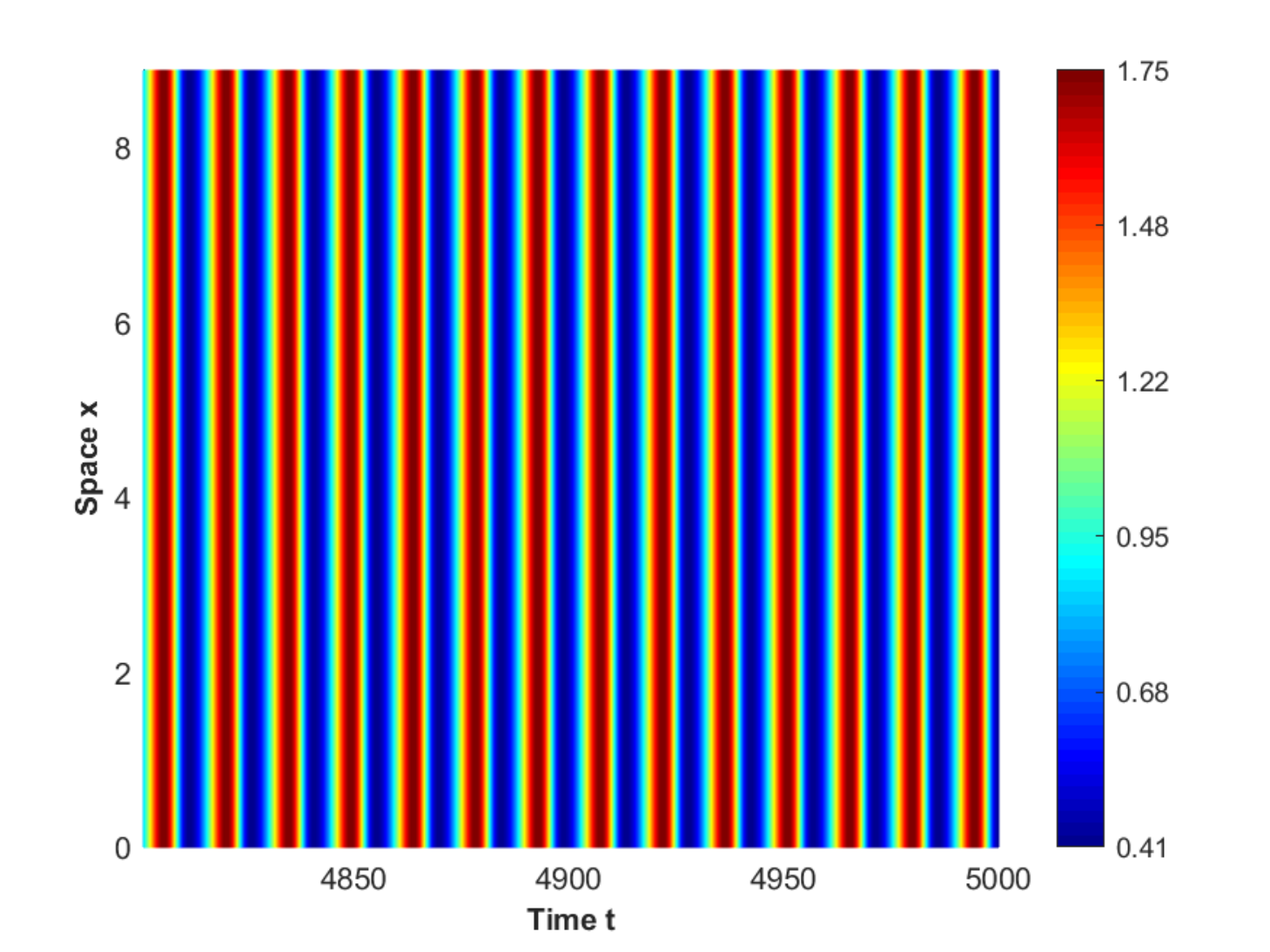}}\vspace{-0.3cm}
			\subfigure[]{\includegraphics[width=0.33\textwidth]{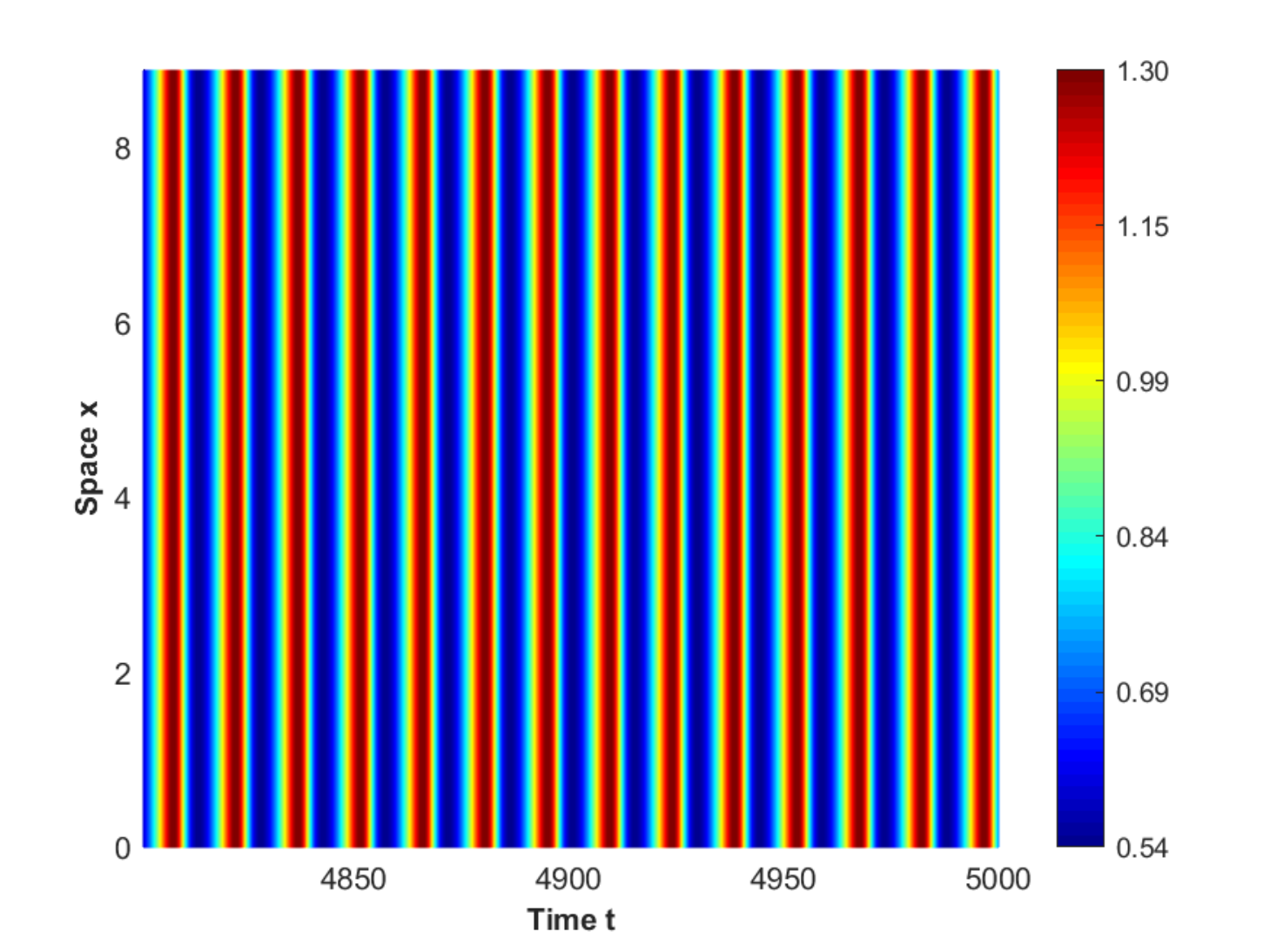}}
		\end{center}
	\begin{center}
		\subfigure[]{\includegraphics[width=0.33\textwidth]{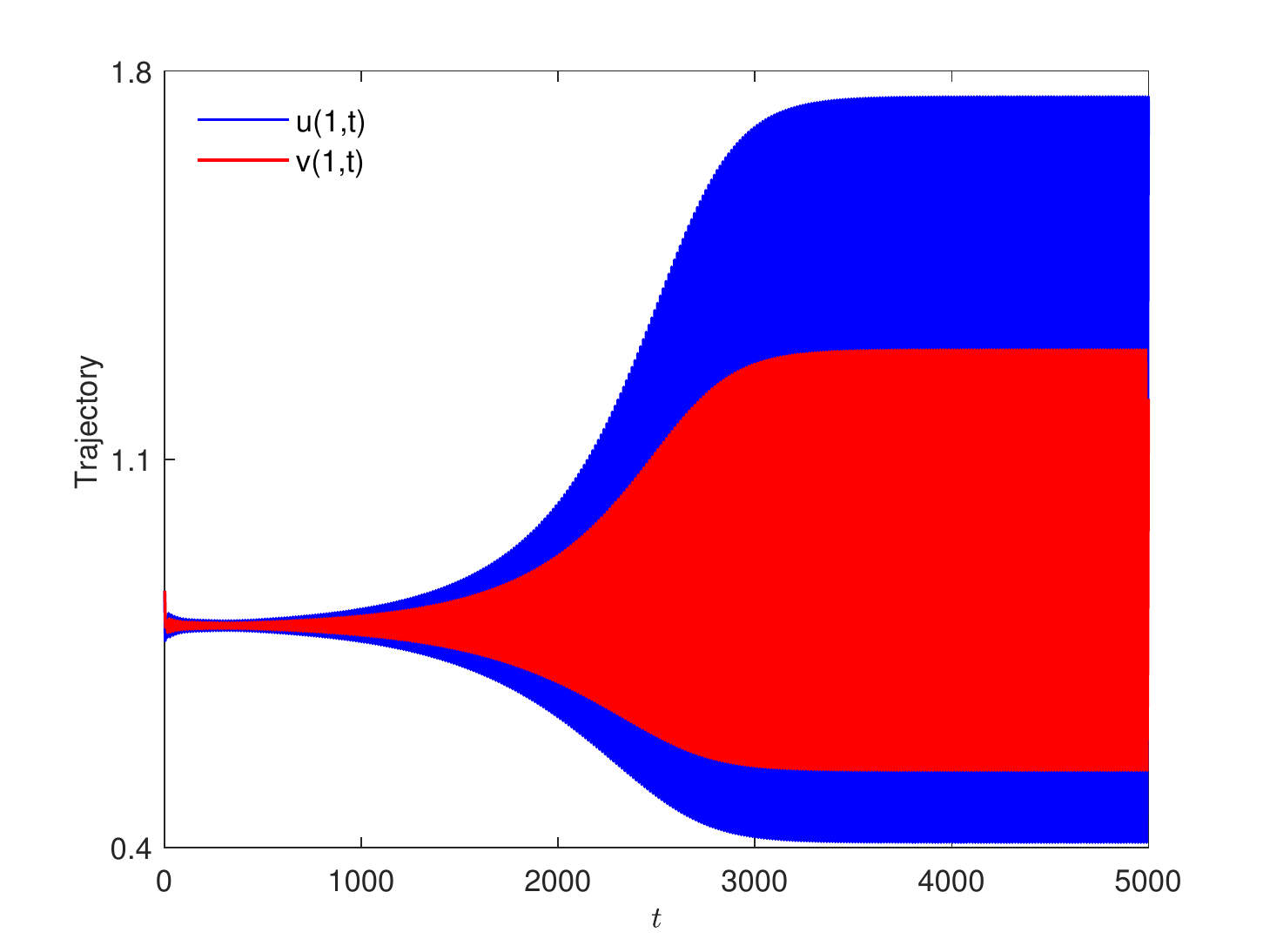}}\vspace{-0.3cm}
		\subfigure[]{\includegraphics[width=0.33\textwidth]{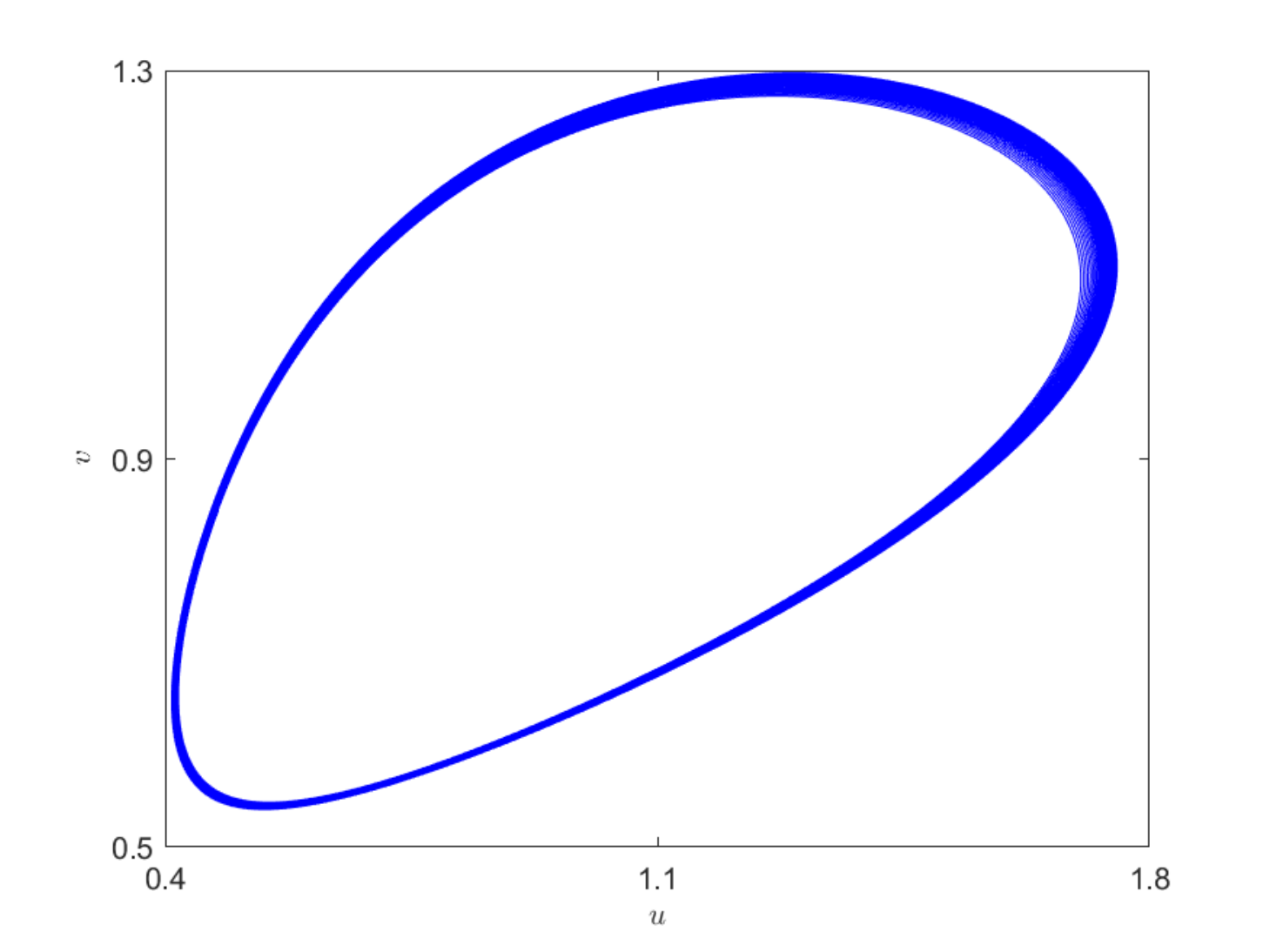}}
	\end{center}
	\end{multicols}
\vspace{-0.2cm}
	\caption{Stable spatially homogeneous periodic solution with
		$(\mu_1, \mu_2)=(-0.2, -0.02)\in D_2$, and the initial function are $(\lambda_0+0.5\sin x, \lambda_0+0.5\cos x)$. (a) and (c): the dynamics of $u$; (b) and (d): the dynamics of $v$; (e) and (f): the trajectories and corresponding periodic orbits at $x=1$.}\label{fig-D2}
	
\end{figure}

%
%
%
%

\begin{figure}[htp]
	\begin{multicols}{3}
		\begin{center}
			\subfigure[]{\includegraphics[width=0.33\textwidth]{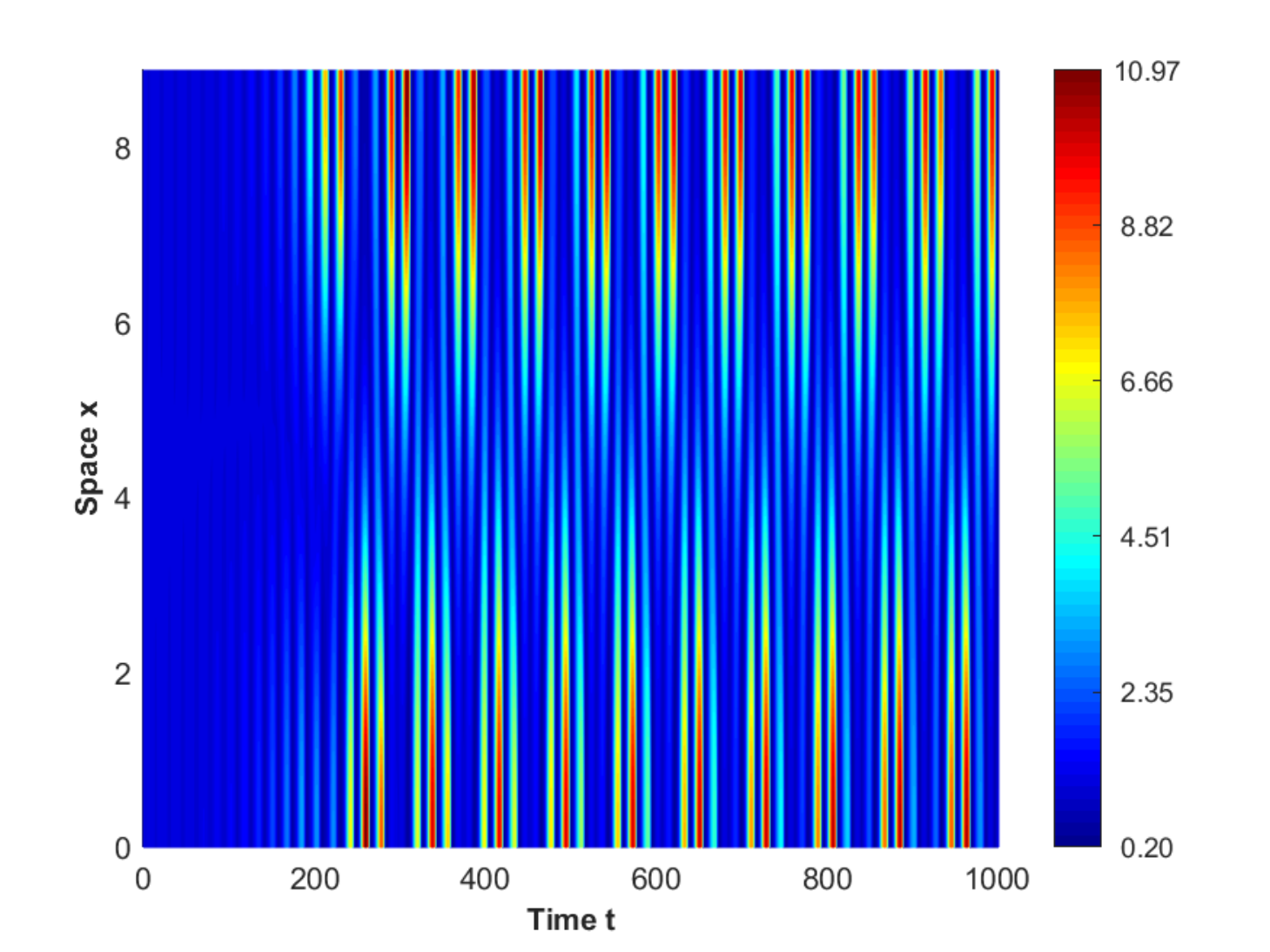}}\vspace{-0.3cm}
			\subfigure[]{\includegraphics[width=0.33\textwidth]{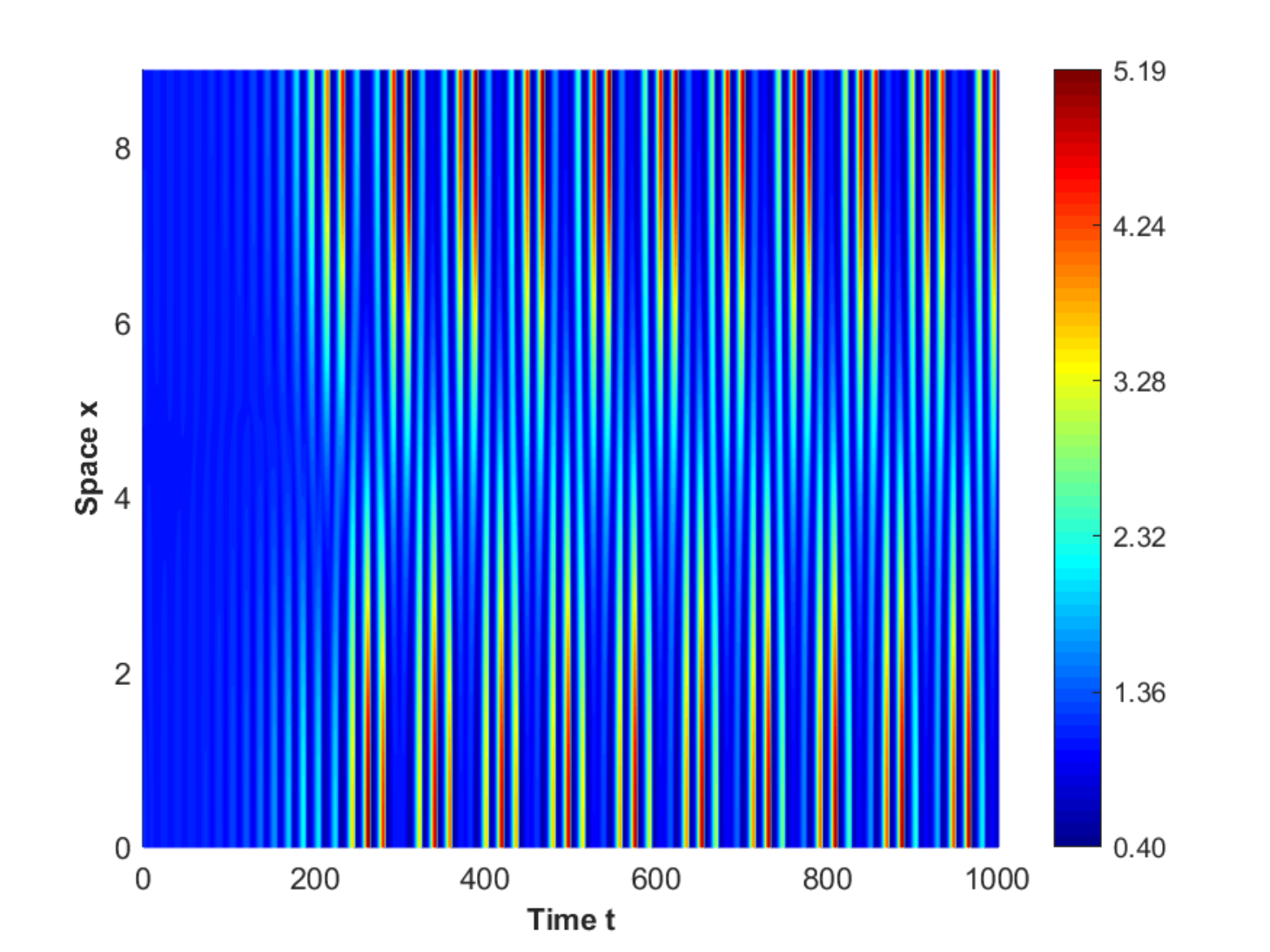}}
			
		\end{center}
		
		\begin{center}
			\subfigure[]{\includegraphics[width=0.33\textwidth]{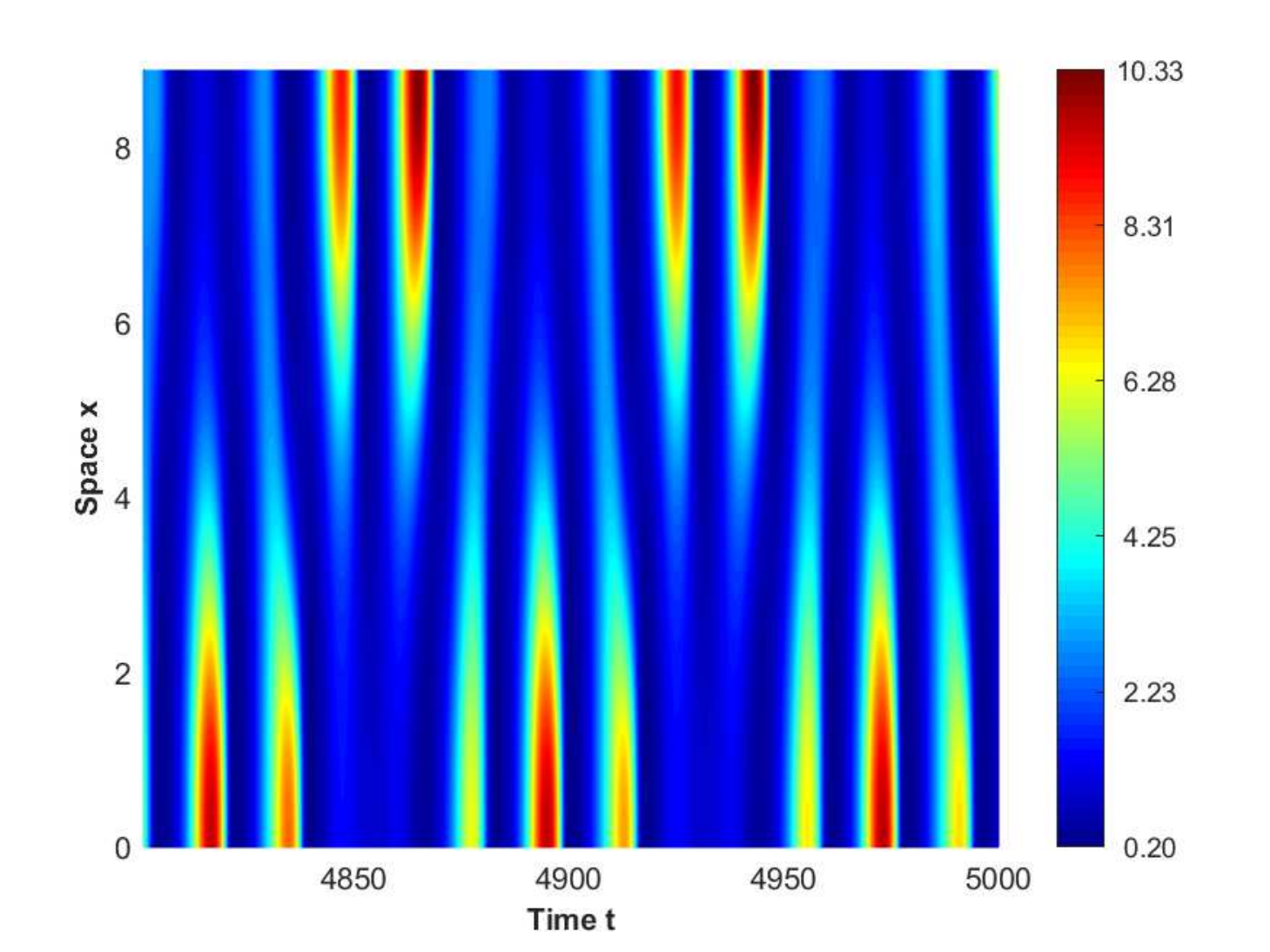}}\vspace{-0.3cm}
			\subfigure[]{\includegraphics[width=0.33\textwidth]{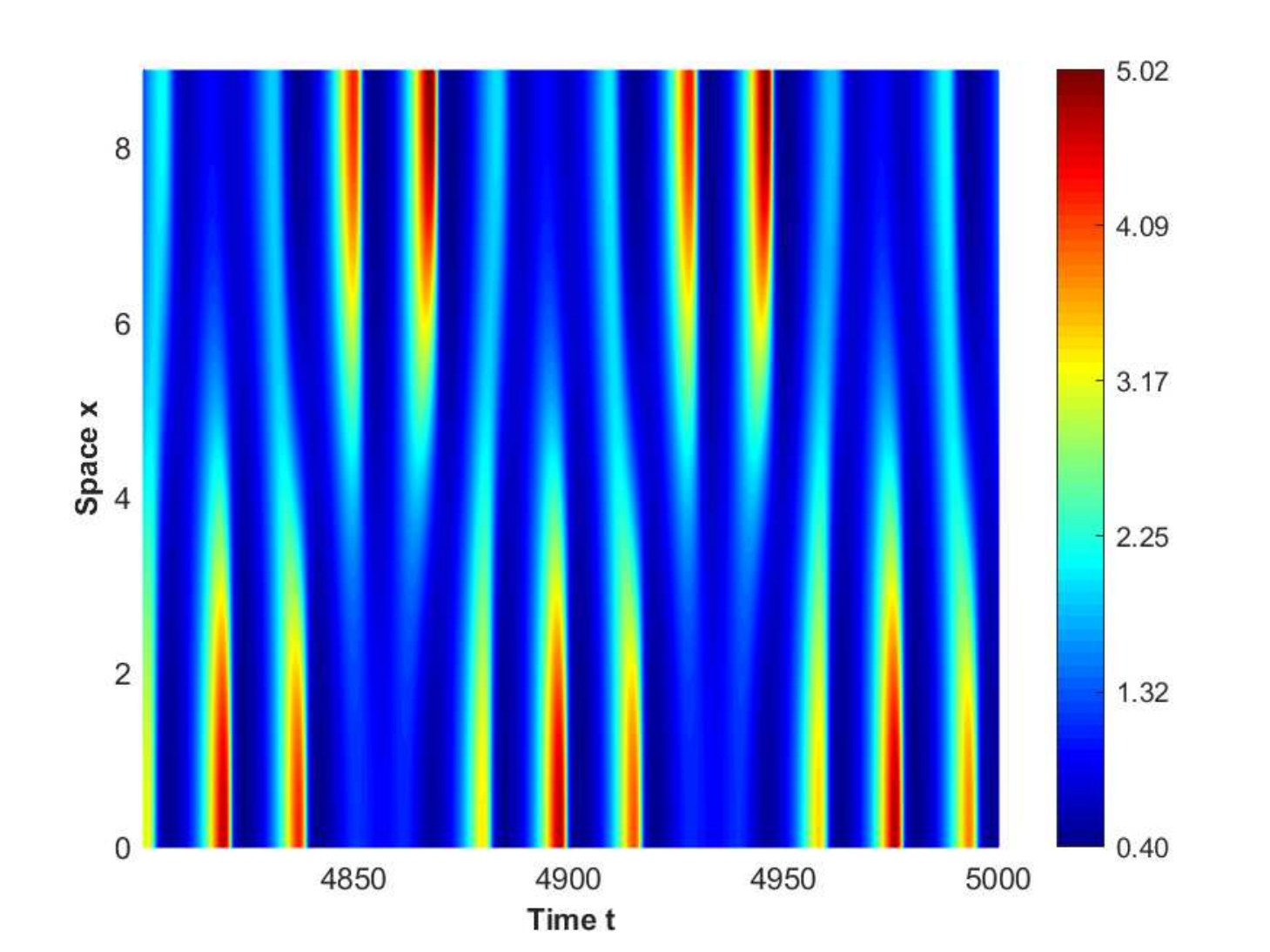}}
		\end{center}
	\begin{center}
		\subfigure[]{\includegraphics[width=0.33\textwidth]{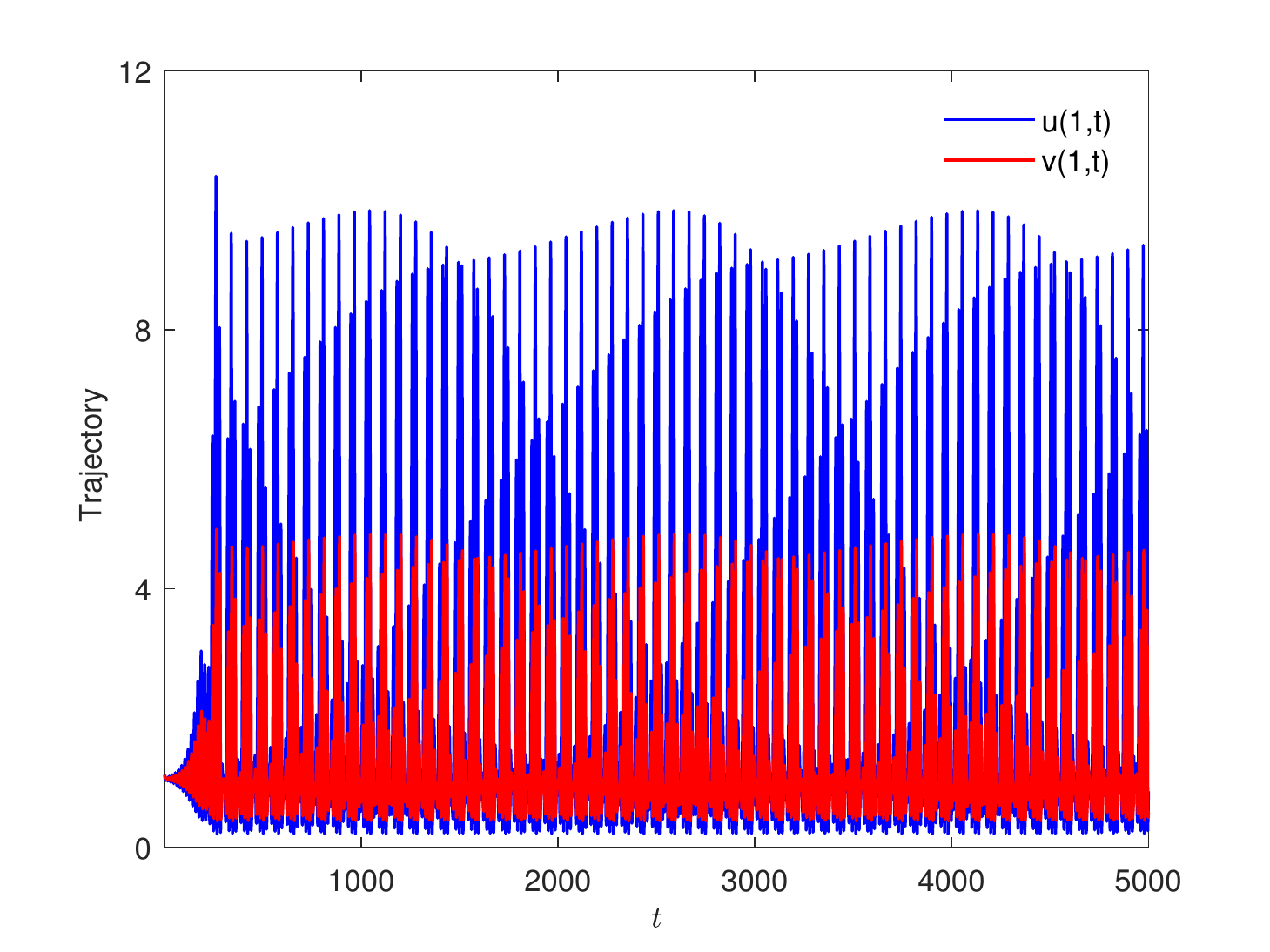}}\vspace{-0.3cm}
		\subfigure[]{\includegraphics[width=0.33\textwidth]{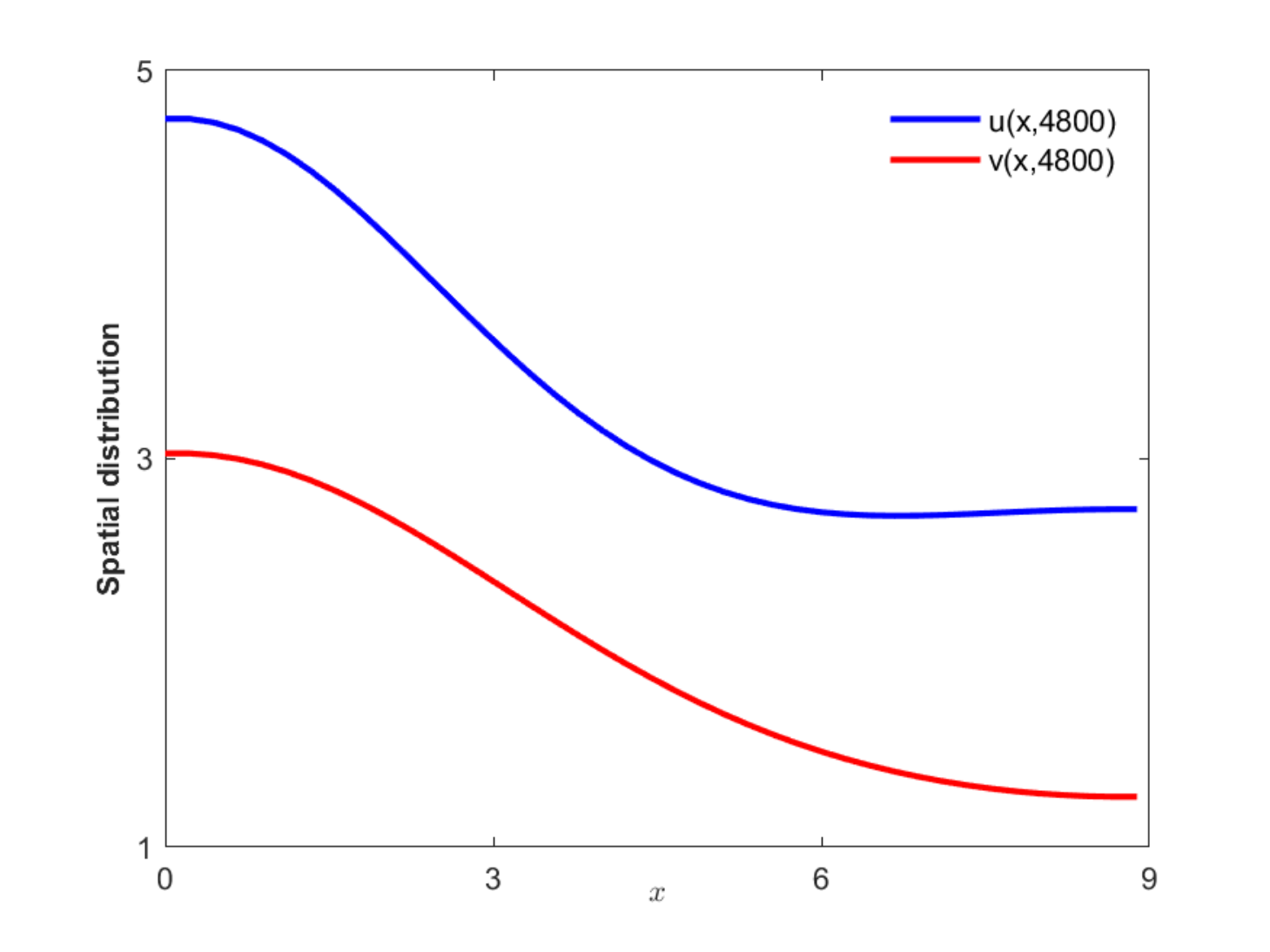}}
	\end{center}
	\end{multicols}
\vspace{-0.2cm}
	\caption{Stable spatially nonhomogeneous quasi-periodic solution with
		$(\mu_1, \mu_2)=(0.06, -0.03)\in D_4$, and the initial function are $(\lambda_0+0.1\cos x, \lambda_0+0.1\cos x)$. (a) and (c): the dynamics of $u$; (b) and (d): the dynamics of $v$; (e): the trajectories at $x=1$; (f): the spatial distribution at $t=4800$.}\label{fig-D4}
	
\end{figure}

%
%
%

\begin{figure}[htp]
	\begin{multicols}{3}
		\begin{center}
			\subfigure{\includegraphics[width=0.33\textwidth]{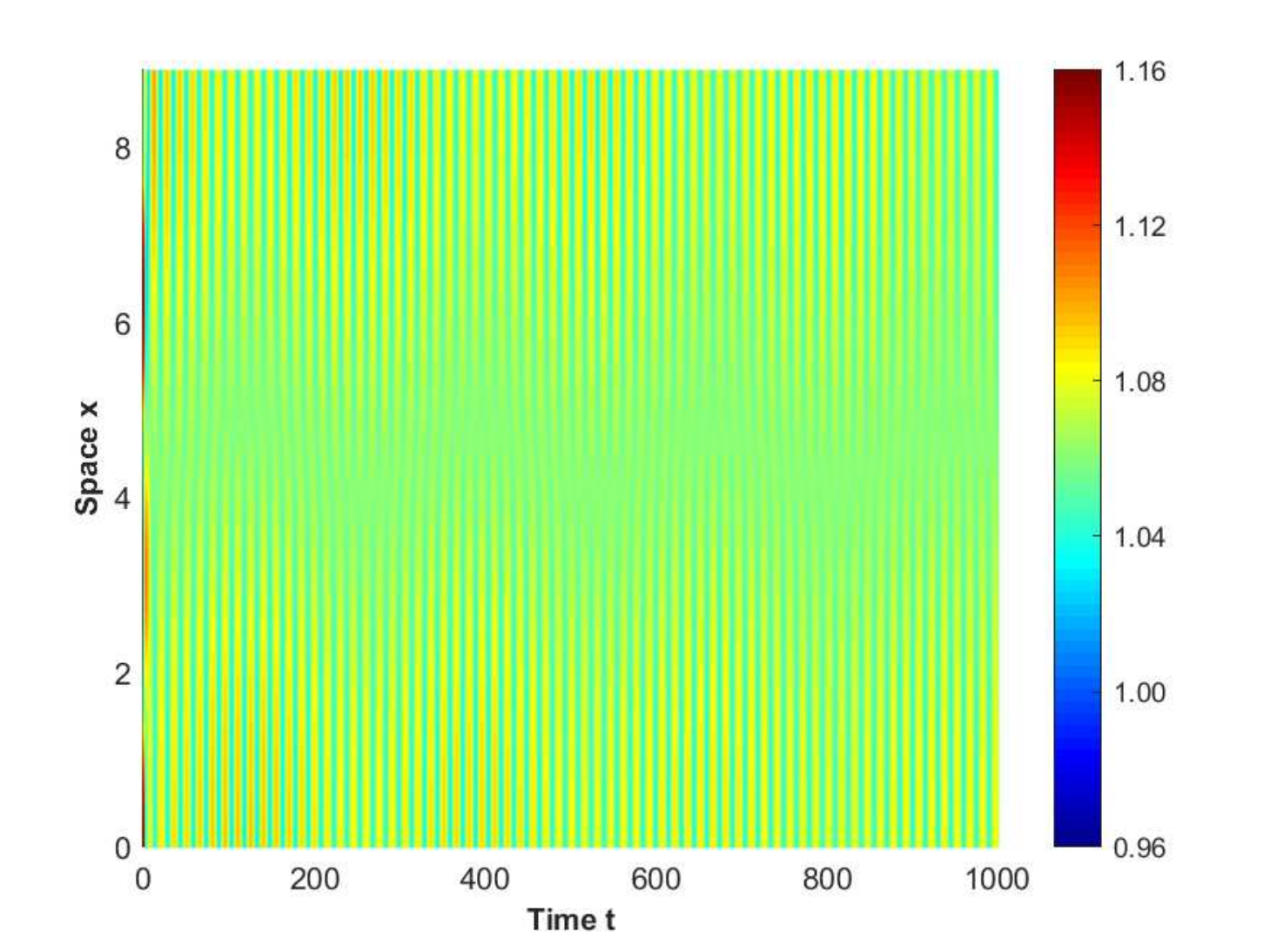}}\\
			\subfigure{\includegraphics[width=0.33\textwidth]{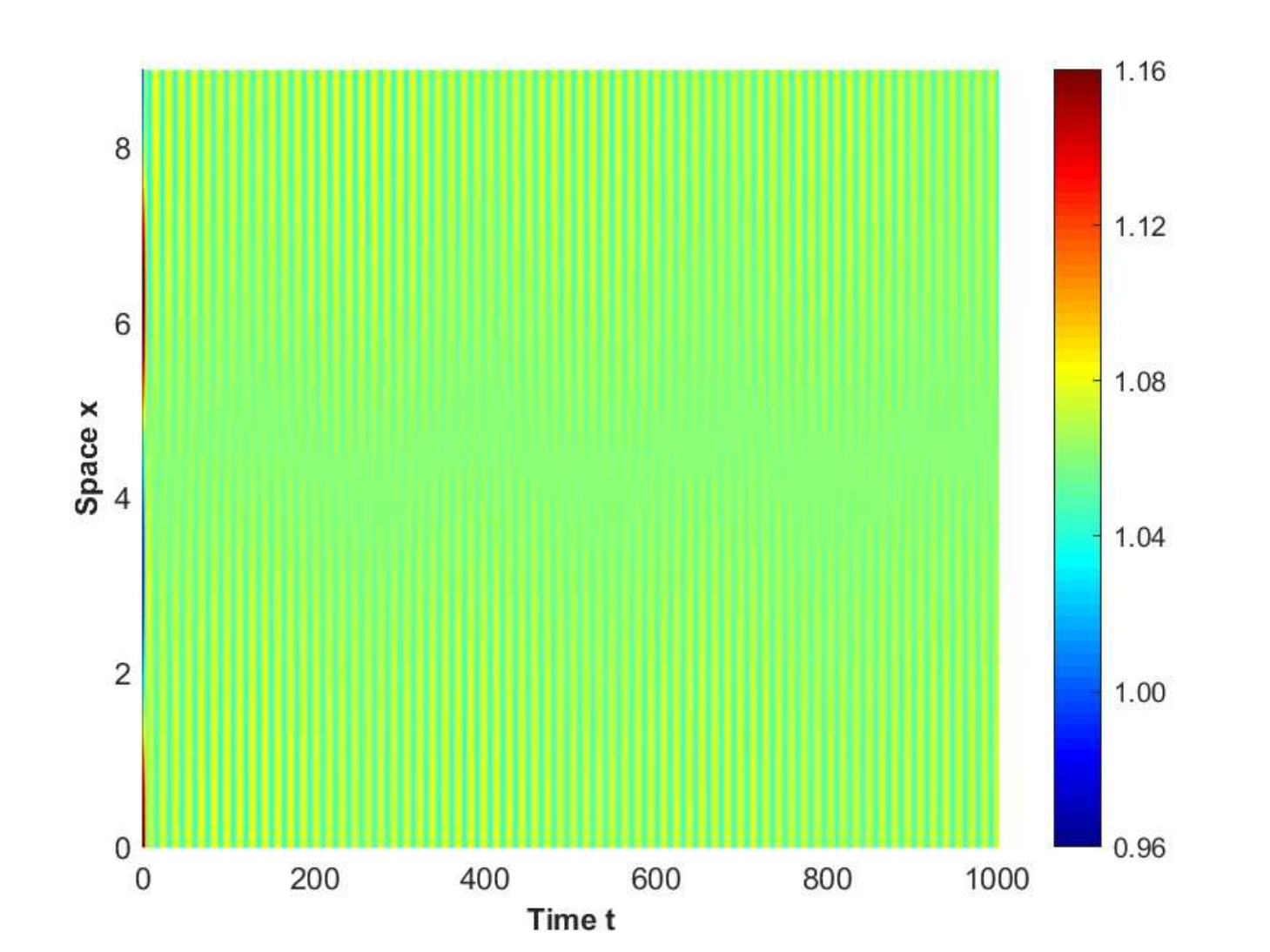}}
		\end{center}
		
		\begin{center}
			\subfigure{\includegraphics[width=0.33\textwidth]{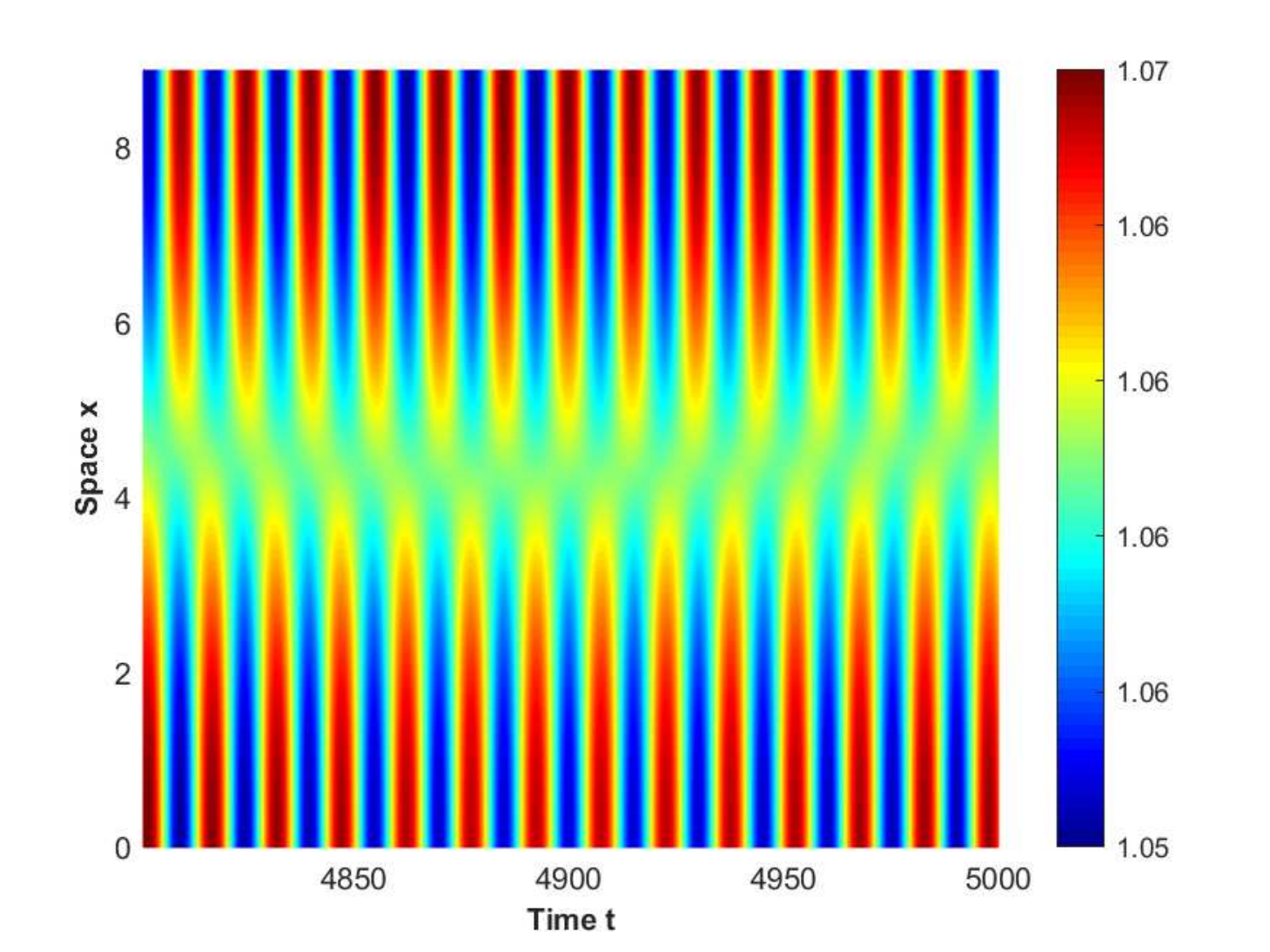}}\\
			\subfigure{\includegraphics[width=0.33\textwidth]{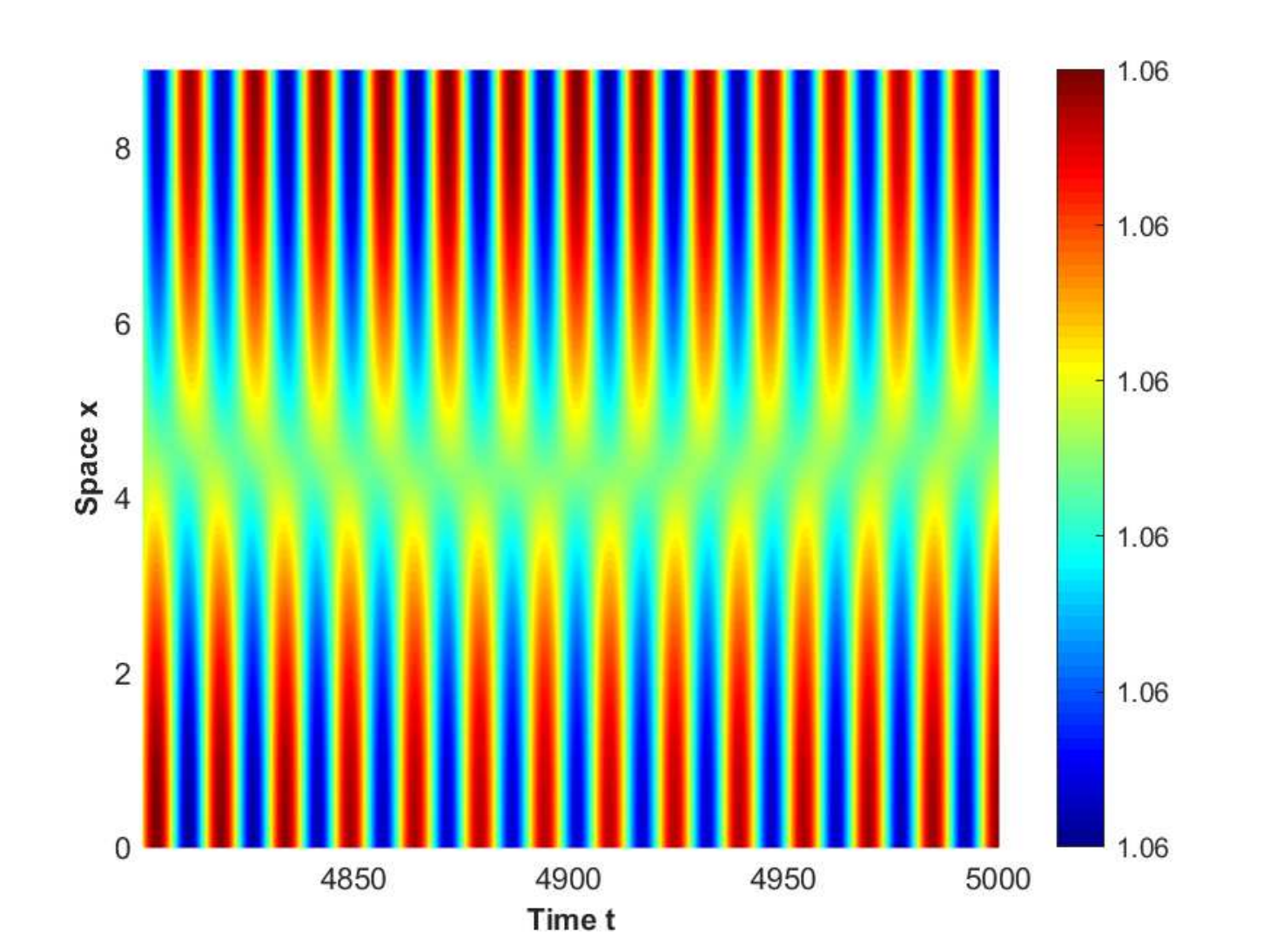}}
		\end{center}
				
		\begin{center}
			\subfigure{\includegraphics[width=0.33\textwidth]{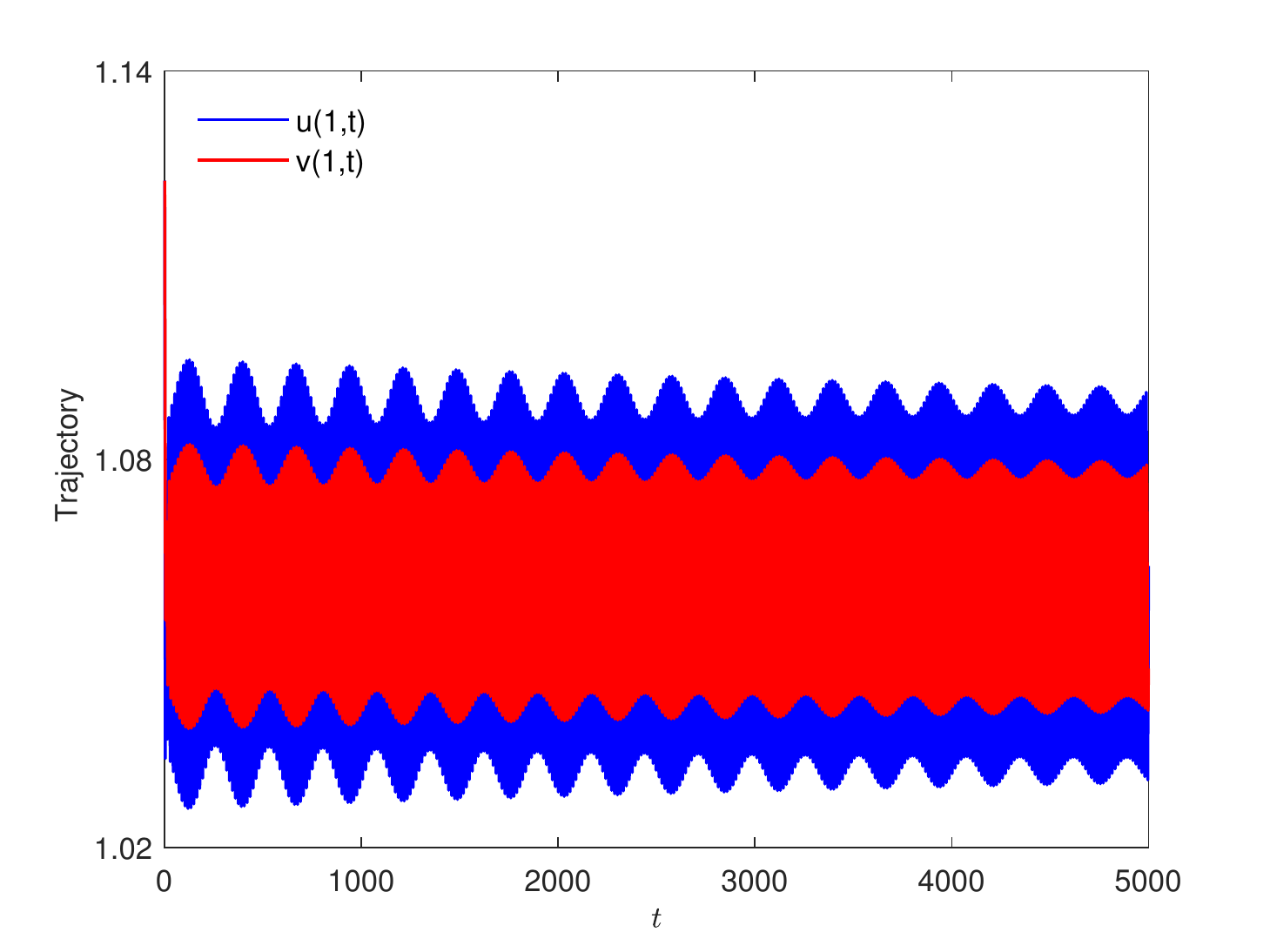}}\\
			\subfigure{\includegraphics[width=0.33\textwidth]{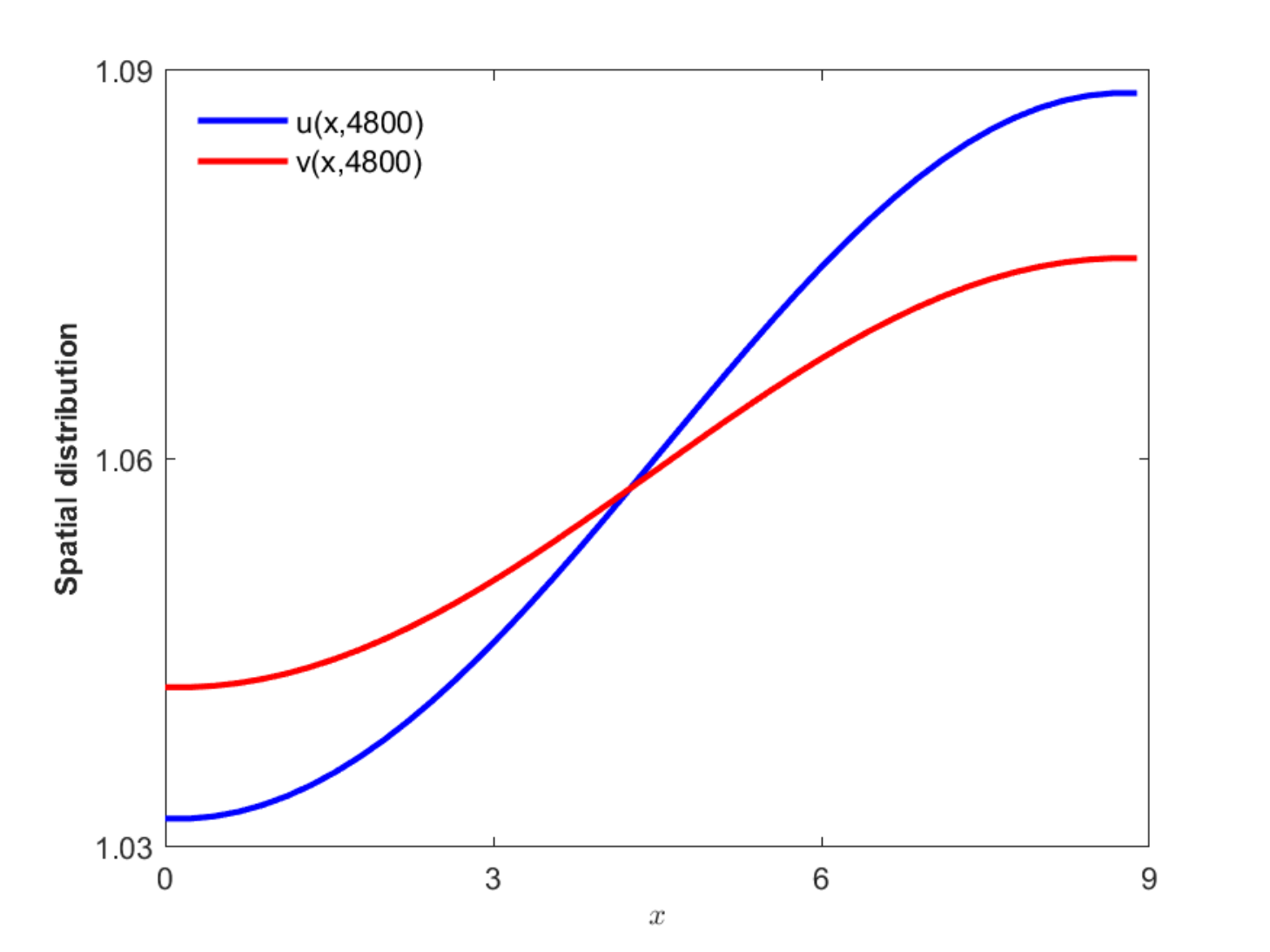}}
		\end{center}
		
	\end{multicols}
\vspace{-0.2cm}
	\caption{Stable spatially nonhomogeneous periodic solution with
		$(\mu_1, \mu_2)=(0.06, 0.00925)\in D_6$, and the initial function are $(\lambda_0+0.1\cos x, \lambda_0+0.1\cos x)$.(a) and (c): the dynamics of $u$; (b) and (d): the dynamics of $v$; (e): the trajectories at $x=1$; (f): the spatial distribution at $t=4800$.}\label{fig-D6}
	
\end{figure}

\begin{figure}[htp]
	\centering
	\subfigure{\includegraphics[width=0.45\textwidth]{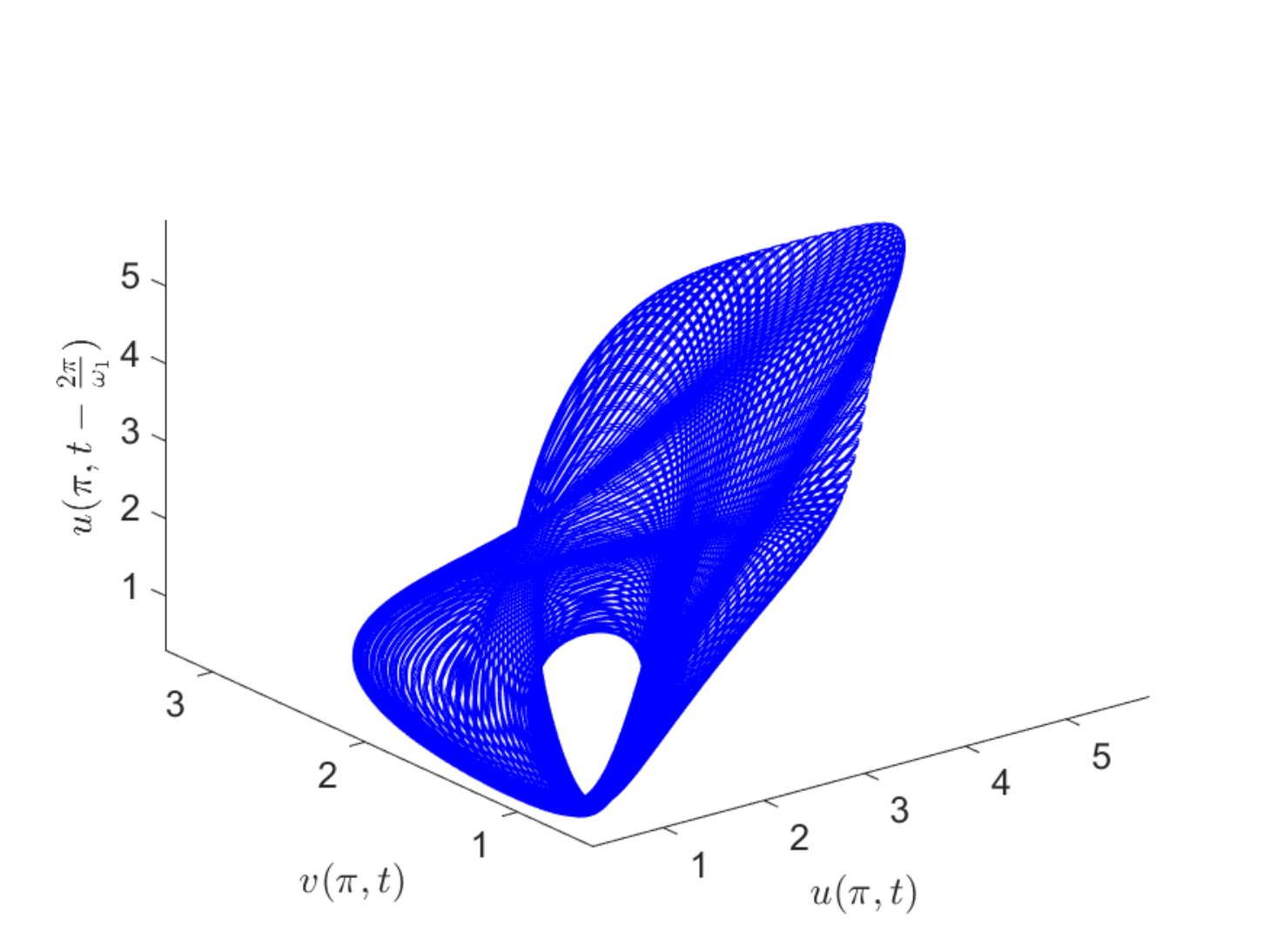}}
	\subfigure{\includegraphics[width=0.45\textwidth]{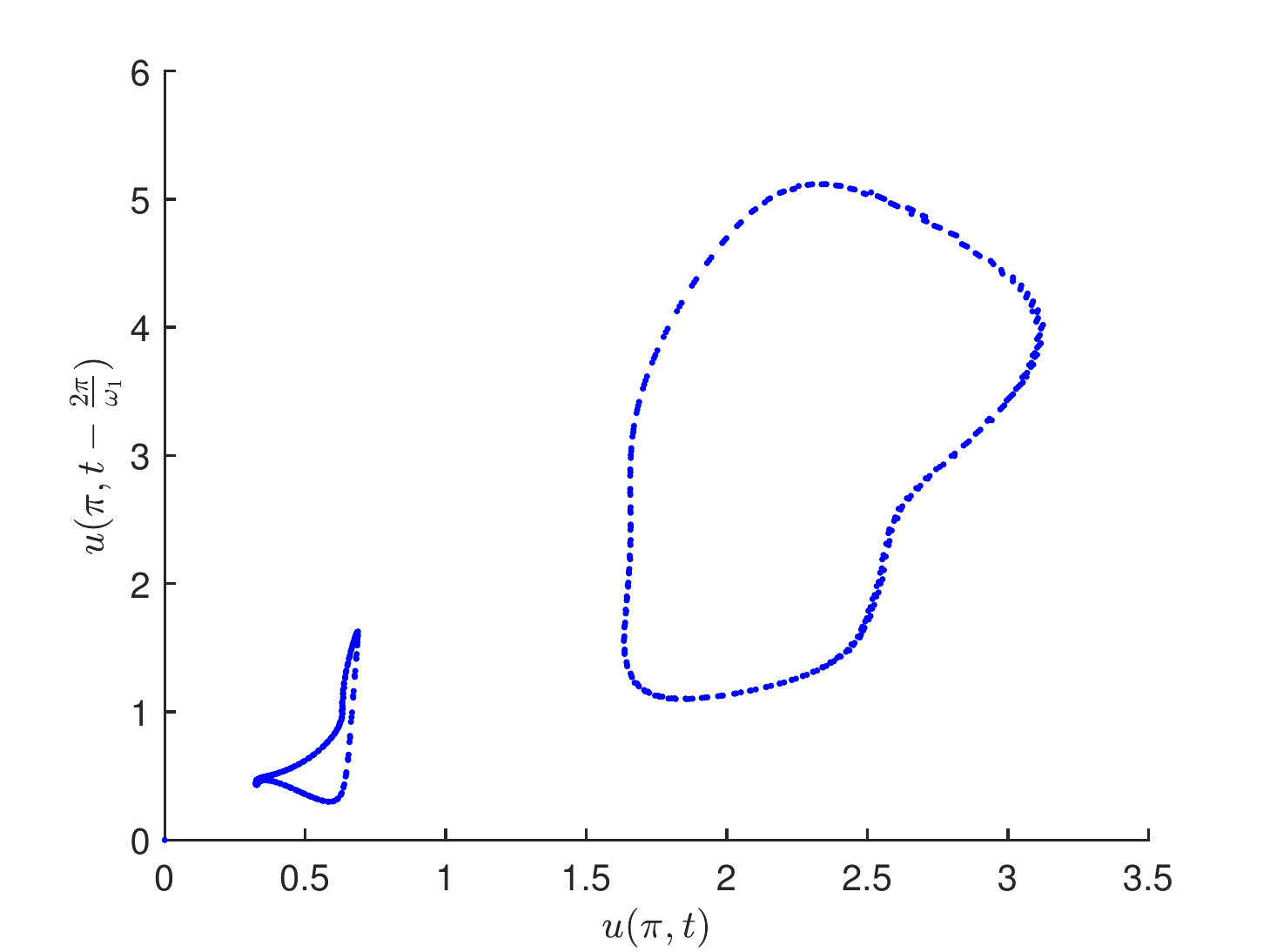}}
	\caption{The phase portraits in $u(\pi,t)-v(\pi,t)-u(\pi,t-\frac{2\pi}{\omega_1})$ coordinate and the corresponding Poincar\'{e} map on $v(\pi,t)=\lambda_0$ when $(\mu_1,\mu_2)=(0.06,-0.03)\in D_4$.}\label{fig-D4_poincare}
\end{figure}

\section{Discussion and conclusion}

In this paper, we develop an algorithm for computing normal forms associated with the codimension-two double Hopf bifurcation for a general reaction-diffusion system with spatial average nonlocal kernel and Neumann boundary conditions.
The algorithm looks complicated, but it is actually easy for computer implementation, especially when the system consists of only two variables.
We introduce a Boolean function to handle the effects of nonlocal terms on the computation of normal forms. The system can exhibit rich dynamics near the double Hopf bifurcation singularity, and the possible attractors near this degenerated point mainly include spatially homogeneous/nonhomogeneous periodic solutions,
spatially nonhomogeneous quasi-periodic solutions.

We apply
our result to a reaction-diffusion Holling-Tanner system with nonlocal prey competition. The qualitative analysis reveals that the dynamic behaviors of the system with nonlocal terms is more complex than that of the original system.
The unfolding of $(IVa)$ (see Table.\ref{twelve_unfoldings}) occurs in our numerical simulations, and the spatially homogeneous and 
nonhomogeneous periodic solutions are observed from numerical simulations.
Furthermore, The existence of the spatially nonhomogeneous quasi-periodic solution is
verified in Poincar\'{e} section.

\appendix
\section{Appendix}
This appendix is an extension of the case that $n=2$, and we shall show the
details of the coefficient vectors that appear in Section \ref{section_Normal_Form}.
Without loss of generality, we denote
\begin{equation*}
	\phi_1=\left(\begin{array}{c}
	1\\
	q_1
	\end{array}\right),~
		\phi_2=\left(\begin{array}{c}
	1\\
	q_2
	\end{array}\right),~
		\phi_3=\left(\begin{array}{c}
	1\\
	q_3
	\end{array}\right),~
		\phi_4=\left(\begin{array}{c}
	1\\
	q_4
	\end{array}\right),~
\end{equation*}
with $q_1=\bar{q}_2$, $q_3=\bar{q}_4$. Note that $n_2>n_1\geq 0$, i.e., $\delta(n_2)=0$, then for $U=(u,v)^{^T}\in X_{\mathbb{C}}$,
we have
\begin{equation*}
    \begin{array}{ll}
	U&=\left(\begin{array}{c}
	u\\ v
	\end{array}\right)=\left(\begin{array}{c}
	1\\q_1 \end{array}\right)z_1\xi_{n_1}+\left(\begin{array}{c}
	1\\q_2 \end{array}\right)z_2\xi_{n_1}+\left(\begin{array}{c}
	1\\q_3 \end{array}\right)z_3\xi_{n_2}+\left(\begin{array}{c}
	1\\q_4 \end{array}\right)z_4\xi_{n_2}+w,\\
	\widehat{U}&=\left(\begin{array}{c}
	\widehat{u}\\ \widehat{v}
	\end{array}\right)=\left(\begin{array}{c}
	1\\q_1 \end{array}\right)z_1\widehat{\xi}_{n_1}+\left(\begin{array}{c}
	1\\q_2 \end{array}\right)z_2\widehat{\xi}_{n_1}
	+\widehat{w}.
	\end{array}
\end{equation*}
Then the coefficient vectors $F_{\iota_1\iota_2\iota_3\iota_4}$, $F_{wz_i}$, $F_{\widehat{w}z_i}$ shown in section \ref{section_Normal_Form} can be obtained by computing the following partial derivatives, where $F_{uu}=\frac{\partial^2}{\partial u^2}F(0,0,\mu_0)$, and other  symbols are similarly defined:
\begin{equation*}
  \begin{array}{ll}
  F_{w_1z_1}=2\big(F_{uu}+F_{uv}q_1+F_{u\widehat{u}}\delta(n_1)
             +F_{u\widehat{v}}q_1\delta(n_1)\big),~
  &F_{\widehat w_1z_1}=2\big(F_{u\widehat{u}}+F_{v\widehat{u}}q_1
             +F_{\widehat{u}\widehat{u}}\delta(n_1)
             +F_{\widehat{u}\widehat{v}}q_1\delta(n_1)\big),\\

  F_{w_2z_1}=2\big(F_{uv}+F_{vv}q_1+F_{v\widehat{u}}\delta(n_1)
             +F_{v\widehat{v}}q_1\delta(n_1)\big),~\,
  &F_{\widehat w_2z_1}=2\big(F_{u\widehat{v}}+F_{v\widehat{v}}q_1
             +F_{\widehat{u}\widehat{v}}\delta(n_1)
             +F_{\widehat{v}\widehat{v}}q_1\delta(n_1)\big),\\

  F_{w_1z_3}=2\big(F_{uu}+F_{uv}q_3\big),~
 &F_{\widehat w_1z_3}=2\big(F_{u\widehat{u}}+F_{v\widehat{u}}q_3
             \big),\\

  F_{w_2z_3}=2\big(F_{uv}+F_{vv}q_3\big),~\,
  &F_{\widehat w_2z_3}=2\big(F_{u\widehat{v}}+F_{v\widehat{v}}q_3
             \big),\\
 \end{array}
\end{equation*}
and
\begin{equation*}
     \begin{array}{ll}
    F_{w_1z_2}=\overline{F_{w_1z_1}},~
    F_{w_2z_2}=\overline{F_{w_2z_1}},~
    F_{w_1z_4}=\overline{F_{w_1z_3}},~
    F_{w_2z_4}=\overline{F_{w_2z_3}},\\
    F_{\widehat{w}_1z_2}=\overline{F_{\widehat{w}_1z_1}},~
    F_{\widehat{w}_2z_2}=\overline{F_{\widehat{w}_2z_1}},~
    F_{\widehat{w}_1z_4}=\overline{F_{\widehat{w}_1z_3}},~
    F_{\widehat{w}_2z_4}=\overline{F_{\widehat{w}_2z_3}}.
     \end{array}
\end{equation*}
The coefficient vectors required in $\widetilde{F}_2(z,0,0,0)$ are given by
\begin{equation*}
  \begin{array}{ll}
   F_{2000}=&F_{uu}+ F_{vv}q_1^2+ F_{\widehat{u}\widehat{u}}\delta(n_1)+
            F_{\widehat{v}\widehat{v}}q_1^2\delta(n_1)
            +2\big[F_{uv}q_1 +F_{u\widehat{u}}\delta(n_1)
            +F_{u\widehat{v}}q_1\delta(n_1)\\
            &+F_{v\widehat{u}}q_1\delta(n_1)
            +F_{v\widehat{v}}q_1^2\delta(n_1)
            +F_{\widehat{u}\widehat{v}}q_1\delta(n_1)\big],\\
   F_{1100}=&2\big[F_{uu}+F_{uv}(q_1+q_2)+2F_{u\widehat{u}}\delta(n_1)
            +F_{u\widehat{v}}(q_1+q_2)\delta(n_1)
            +F_{vv}q_1q_2+F_{v\widehat{u}}(q_1+q_2)\delta(n_1)\\
            &+2F_{v\widehat{v}}q_1q_2\delta(n_1)
            +F_{\widehat{u}\widehat{u}}\delta(n_1)
            +F_{\widehat{u}\widehat{v}}(q1+q2)\delta(n_1)
            +F_{\widehat{v}\widehat{v}}q_1q_2\delta(n_1)\big],\\
   F_{1010}=&2\big[F_{uu}+F_{uv}(q_1+q_3)
            +F_{u\widehat{u}}\delta(n_1)
            +F_{u\widehat{v}}q_1\delta(n_1)
            +F_{vv}q_1q_3
            +F_{v\widehat{u}}
            q_3\delta(n_1)
            +F_{v\widehat{v}}q_1q_3\delta(n_1)
            \big],\\
   F_{1001}=&2\big[F_{uu}+F_{uv}(q_1+q_4)
            +F_{u\widehat{u}}\delta(n_1)
            +F_{u\widehat{v}}q_1\delta(n_1)
            +F_{vv}q_1q_4
            +F_{v\widehat{u}}q_4\delta(n_1)
            +F_{v\widehat{v}}q_1q_4\delta(n_1)
           \big],\\
   F_{0020}=&F_{uu}+F_{vv} q_3^2
             +2 F_{uv}q_3,\\
   F_{0011}=&2\big[F_{uu}+F_{uv}(q_3+q_4)
             +F_{vv}q_3q_4
             \big],\\

   F_{0200}=&\overline{F_{2000}},~F_{0101}=\overline{F_{1010}},
   F_{0002}=\overline{F_{0020}},~F_{0110}=\overline{F_{1001}},\\
  \end{array}
\end{equation*}
and those in $\widetilde{F}_3(z,0,0,0)$ are as follows:
\begin{equation*}
  \begin{array}{ll}
    F_{2100}=&3\big[F_{uuu}+F_{uuv}(2q_1+q_2)
             +3F_{uu\widehat{u}}\delta(n_1)
             +F_{uu\widehat{v}}(2q_1+q_2)\delta(n_1)
             +F_{uvv}q_1(2q_2+q_1)\\
             &+2F_{uv\widehat{u}}(2q_1+q_2)\delta(n_1)
             +2F_{uv\widehat{v}}q_1(2q_2+q_1)\delta(n_1)
             +3F_{u\widehat{u}\widehat{u}}\delta(n_1)
             +2F_{u\widehat{u}\widehat{v}}\delta(n_1)(2q_1+q_2)\\
             &+F_{u\widehat{v}\widehat{v}}q_1(2q_2+q_1)\delta(n_1)
             +F_{vvv}q_1^2q_2
             +F_{vv\widehat{u}}q_1(2q_2+q_1)\delta(n_1)
             +3F_{vv\widehat{v}}q_1^2q_2\delta(n_1)\\
             &+F_{v\widehat{u}\widehat{u}}(2q_1+q_2)\delta(n_1)
             +2F_{v\widehat{u}\widehat{v}}\delta(n_1)(q_1^2+2q_1q_2)
             +3F_{v\widehat{v}\widehat{v}}q_1^2q_2\delta(n_1)
             +F_{\widehat{u}\widehat{u}\widehat{u}}\delta(n_1)\\
             &+F_{\widehat{u}\widehat{u}\widehat{v}}\delta(n_1)(q_2+2q_1)
             +F_{\widehat{u}\widehat{v}\widehat{v}}\delta(n_1)(q_1^2+2q_1q_2)
             +F_{\widehat{v}\widehat{v}\widehat{v}}q_1^2q_2\delta(n_1)
              \big],\\

    F_{1011}=&6\big[F_{uuu}+F_{uuv}(q_1+q_3+q_4)
             +F_{uu\widehat{u}}\delta(n_1)
             +F_{uu\widehat{v}}q_1\delta(n_1)
             +F_{uvv}(q_3q_4+q_1q_3+q_1q_4)\\
             &+F_{uv\widehat{u}}(q_3+q_4)\delta(n_1)
             +F_{uv\widehat{v}}(q_1q_4+q_1q_3)\delta(n_1)
             +F_{vvv}q_1q_3q_4
             +F_{vv\widehat{u}}q_3q_4\delta(n_1)
             +F_{vv\widehat{v}}q_1q_3q_4\delta(n_1)\big].\\

     F_{0021}=&3\big[F_{uuu}+F_{uuv}(2q_3+q_4)
             +F_{uvv}q_3(2q_4+q_3)
             +F_{vvv}q_3^2q_4
             \big],\\
     F_{1110}=&6\big[F_{uuu}+F_{uuv}(q_3+q_1+q_2)
              +2F_{uu\widehat{u}}\delta(n_1)
              +F_{uu\widehat{v}}\delta(n_1)(q_1+q_2)
              +F_{uvv}(q_1q_2+q_3q_1+q_3q_2)\\
              &+F_{uv\widehat{u}}(2q_3+q_1+q_2)\delta(n_1)
              +F_{uv\widehat{v}}(2q_1q_2+q_3q_2+q_3q_1)\delta(n_1)
              +F_{u\widehat{u}\widehat{u}}\delta(n_1)
              +F_{u\widehat{u}\widehat{v}}\delta(n_1)(q_1+q_2)\\
              &+F_{u\widehat{v}\widehat{v}}\delta(n_1)q_1q_2
              +F_{vvv}q_3q_1q_2
              +F_{vv\widehat{u}}(q_3q_2+q_3q_1)\delta(n_1)
              +2F_{vv\widehat{v}}q_1q_2q_3 \delta(n_1)\\
              &+F_{v\widehat{u}\widehat{u}}\delta(n_1)q_3
              +F_{v\widehat{u}\widehat{v}}\delta(n_1)(q_3q_1+q_3q_2)
              +F_{v\widehat{v}\widehat{v}}q_3q_1q_2\delta(n_1)
              \big].\\
  \end{array}
\end{equation*}

\end{document}